\documentclass[12pt]{article}
\usepackage[]{amsmath,amssymb}
\usepackage{amscd}
\usepackage{latexsym}
\usepackage{cite}
\usepackage{amsthm}

\newtheorem{definition}{Definition}[section]
\newtheorem{theorem}[definition]{Theorem}
\newtheorem{lemma}[definition]{Lemma}
\newtheorem{corollary}[definition]{Corollary}

\newtheorem{example}[definition]{Example}

\newtheorem{problem}[definition]{Problem}
\newtheorem{note}[definition]{Note}

\newtheorem{proposition}[definition]{Proposition}

\def\F{\mathbb F}

\begin{document}
\title{\bf  
Billiard Arrays and finite-dimensional irreducible 
$U_q(\mathfrak{sl}_2)$-modules
}
\author{
Paul Terwilliger}
\date{}

\maketitle
\begin{abstract}
We introduce the notion of a Billiard Array. This
is an equilateral triangular array of one-dimensional
subspaces of a vector space $V$, subject to several
conditions that specify which sums are direct.
We show that the Billiard Arrays on $V$ are in bijection
with the 3-tuples of totally opposite flags on $V$.
We classify the Billiard Arrays up to isomorphism.
We use Billiard Arrays to describe the finite-dimensional
irreducible modules for the quantum algebra
$U_q(\mathfrak{sl}_2)$ and the Lie algebra 
$\mathfrak{sl}_2$.

\bigskip
\noindent
{\bf Keywords}. 
Quantum group,
quantum enveloping algebra,
Lie algebra, flag.
\hfil\break
\noindent {\bf 2010 Mathematics Subject Classification}. 
Primary: 17B37. Secondary: 15A21.
 \end{abstract}

\section{Introduction}
Our topic is informally described as follows.
As in the game of Billiards,
we start with an array of billiard balls arranged to form an
equilateral triangle. We assume that there are $N+1$ balls
along each boundary, with $N\geq 0$. For $N=3$ the balls
are centered at the following locations:
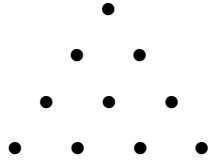
\begin{eqnarray*}
    &&             \quad \quad \quad  \bullet
\\
    &&              \quad \quad  \bullet \quad \;\bullet
\\
    &&           \quad      \bullet \quad \; \bullet \quad \; \bullet
\\
    &&      \bullet \;\quad \bullet \;\quad  \bullet \;\,\quad \bullet
\end{eqnarray*}

\noindent For us, each ball in the array represents a one-dimensional
subspace of an $(N+1)$-dimensional vector space $V$ over a field
$\mathbb F$. 
\medskip

\noindent 
We impose two conditions on the array, that
specify which sums are direct. The first condition is that,
for each set of balls on a line parallel to a boundary, their sum
is direct. The second condition is described as follows.
Three mutually adjacent balls in the array are said to form
a 3-clique. There are two kinds of 3-cliques: $\Delta$ (black) and
$\nabla$ (white). The second condition is that, for any three
balls in the array that form a black 3-clique, their sum is
not direct.

\medskip
\noindent Whenever the above two conditions are met,
our array is called a Billiard Array on $V$.
We say that the Billiard
Array is over $\mathbb F$, and call $N$ the diameter.

\medskip
\noindent We have some remarks about notation. 
 For $1 \leq i \leq N+1$
let 
 ${\mathcal P}_i(V)$  denote the set of subspaces of $V$ that
 have dimension $i$. Recall the natural numbers
 $\mathbb N = \lbrace 0,1,2,\ldots\rbrace$. Let $\Delta_N$
 denote the set consisting of the 3-tuples of natural numbers
 whose sum is $N$. Thus
\begin{equation*}
\Delta_N = \lbrace (r,s,t) \; |\;  r,s,t \in \mathbb N, \;\;r+s+t=N\rbrace.
\end{equation*}
We arrange the elements of $\Delta_N$ in a triangular array.
For $N=3$, the array looks as follows after deleting all punctuation:
\begin{eqnarray*}
    &&                \centerline {030}
\\
    &&                \centerline {120 \;\; 021}
\\
    &&                 \centerline  {210 \;\; 111 \;\; 012}
\\
    &&                 \centerline{300 \;\; 201 \;\; 102 \;\; 003}
\end{eqnarray*}
 An element in $\Delta_N$ is called a location.
 We view our Billiard Array on $V$ as a function
 $B:\Delta_N \to {\mathcal P}_1(V), \lambda \mapsto B_\lambda$.
 For $\lambda \in \Delta_N$, $B_\lambda$ is the billiard ball/subspace
 at location $\lambda$.
\medskip

\noindent  In this paper we obtain three main results, which
are summarized as follows: (i) we show that the Billiard 
Arrays on $V$ are in bijection with the 3-tuples of totally
opposite flags on $V$; (ii) we classify the
Billiard Arrays up to isomorphism; (iii) we use Billiard
Arrays to describe the finite-dimensional irreducible modules
for the quantum algebra 
 $U_q(\mathfrak{sl}_2)$ and the Lie algebra 
 $\mathfrak{sl}_2$.
\medskip

\noindent 
We now describe our results in more detail.  By a flag on $V$
we mean  a sequence $\lbrace U_i \rbrace_{i=0}^N$
such that $U_i \in {\mathcal P}_{i+1}(V)$ for $0 \leq i \leq N$
and $U_{i-1}\subseteq U_i$ for $1 \leq i \leq N$. Suppose we are given
three flags on $V$, denoted 
\begin{equation}
\label{eq:Int3flags}
\lbrace U_i \rbrace_{i=0}^N, \qquad
\lbrace U'_i \rbrace_{i=0}^N, \qquad
\lbrace U''_i \rbrace_{i=0}^N.
\end{equation}
These flags are called totally opposite whenever
$U_{N-r}\cap 
U'_{N-s}\cap 
U''_{N-t} = 0$  
for all $r,s,t$ $(0 \leq r,s,t\leq N)$ such that
$r+s+t>N$.
Assume that the flags in line
(\ref{eq:Int3flags}) are totally opposite. Using these flags we now
construct a Billiard Array on $V$. For each location $\lambda = (r,s,t)$
in $\Delta_N$ define
\begin{eqnarray*}
B_\lambda = 
U_{N-r} \cap U'_{N-s} \cap U''_{N-t}.
\end{eqnarray*}
We will show that $B_\lambda$ has dimension one,
and the map $B:\Delta_N \to \mathcal P_1(V)$,
$\lambda \mapsto B_\lambda$ is a Billiard
Array on $V$.
We just went from flags to Billard Arrays; we now reverse the direction.
Let $B$ denote a Billiard Array on $V$. Using
$B$ we now construct a 3-tuple of totally opposite flags on $V$.
 By the 1-corner 
of $\Delta_N$ we mean the location $(N,0,0)$. The 2-corner and
3-corner of $\Delta_N$ are similarly defined. For $0 \leq i \leq N$
let $U_i$ (resp. $U'_i$) (resp. $U''_i$) denote the sum
of the balls in $B$ that are at most $i$ balls over from the $1$-corner
(resp. $2$-corner) (resp. $3$-corner). We will show that 
\begin{equation*}
\lbrace U_i \rbrace_{i=0}^N, \qquad
\lbrace U'_i \rbrace_{i=0}^N, \qquad
\lbrace U''_i \rbrace_{i=0}^N
\end{equation*}
are totally opposite flags on $V$.
Consider the following two sets:
\begin{enumerate}
\item[\rm (i)] the Billiard Arrays on $V$;
\item[\rm (ii)] the 3-tuples of totally opposite
flags on $V$.
\end{enumerate}
We just described 
a function from (i) to (ii) and a function from
(ii) to  (i). We will show that these functions are inverses, and
hence bijections.
\medskip

\noindent We now describe our classification of Billiard
Arrays up to isomorphism. Let $B$ denote a Billiard Array
on the above vector space $V$.
Let $V'$ denote a vector space over $\mathbb F$
with dimension $N+1$, and let $B'$ denote a Billiard
Array on $V'$. The Billiard Arrays $B,B'$ are called
isomorphic whenever there exists an $\mathbb F$-linear
bijection  $V\to V'$ that sends $B_\lambda \mapsto B'_\lambda$
for all $\lambda \in \Delta_N$. By a value function on
$\Delta_N$ we mean a function $\Delta_N \to \mathbb F\backslash \lbrace 0\rbrace$. 
For $N\leq 1$, up to isomorphism there exists a unique Billiard Array
over $\mathbb F$ that has diameter $N$. For $N\geq 2$ we will obtain
a bijection between the following two sets:
\begin{enumerate}
\item[\rm (i)] the isomorphism classes of Billiard Arrays over $\mathbb F$
that have
diameter $N$;
\item[\rm (ii)] the value functions on $\Delta_{N-2}$.
\end{enumerate}
We now describe the bijection. Let $V$ denote a vector space over
$\mathbb F$ with dimension $N+1$, and let $B$ denote
a Billiard Array on $V$. Given adjacent locations $\lambda, \mu$
in $\Delta_N$, we define an $\mathbb F$-linear map
$\tilde B_{\lambda,\mu}: B_\lambda \to B_\mu$ as follows.
There exists a unique location $\nu \in \Delta_N$  such that
$\lambda, \mu, \nu$ form a black 3-clique. Pick $0 \not=u \in B_\lambda$.
There exists a unique $v \in B_\mu$ such that $u+v\in B_\nu$. The
map 
$\tilde B_{\lambda,\mu}$ sends $u\mapsto v$.
By construction the maps
$\tilde B_{\lambda,\mu}:B_\lambda \to B_\mu$ and
$\tilde B_{\mu,\lambda}:B_\mu \to B_\lambda $
are inverses. 
Let $\lambda, \mu, \nu$ denote locations in $\Delta_N$
that form a $3$-clique, either black or white.
Consider the composition of
maps around the clique:
\begin{equation*}
\begin{CD} 
B_\lambda  @>>  \tilde B_{\lambda,\mu} >  
 B_\mu  @>> \tilde B_{\mu,\nu} > B_\nu 
               @>>\tilde B_{\nu,\lambda} > B_\lambda.
                  \end{CD}
\end{equation*}
This composition is a nonzero scalar multiple of the
identity map on $B_\lambda$. First assume that the 3-clique is
black. Then the scalar is 1. Next assume that the 3-clique
is white. Then the scalar is called the
 clockwise $B$-value 
(resp. counterclockwise $B$-value)
of the 3-clique whenever
the sequence 
 $\lambda, \mu, \nu$ runs clockwise (resp. counterclockwise)
 around the clique.
The clockwise $B$-value and counterclockwise $B$-value are reciprocal.
By the $B$-value of the 3-clique, we mean its clockwise $B$-value.
For $N\leq 1$ the set $\Delta_N$ has no white $3$-clique. For $N\geq 2$
we now give a bijection from $\Delta_{N-2}$ to the set of
white 3-cliques in $\Delta_N$. The bijection sends
each element $(r,s,t)$ in $\Delta_{N-2}$ to the white 
3-clique in $\Delta_N$ consisting of the locations
\begin{eqnarray*}
(r,s+1,t+1), \qquad
(r+1,s,t+1), \qquad
(r+1,s+1,t).
\end{eqnarray*}
Using $B$ we define a function ${\hat B}:\Delta_{N-2} \to \mathbb F$
as follows: $\hat B$ sends each element $(r,s,t)$ in $\Delta_{N-2}$
to the $B$-value of the corresponding white 3-clique in $\Delta_N$.
By construction $\hat B$ is a value function on $\Delta_{N-2}$. The
map $B\mapsto \hat B$ induces our bijection, from the set of
isomorphism classes of Billiard Arrays over $\mathbb F$ that have
diameter $N$, to the set of value functions on $\Delta_{N-2}$.
\medskip

\noindent We now use Billiard Arrays to describe the finite-dimensional
irreducible
 $U_q(\mathfrak{sl}_2)$-modules.
We recall 
 $U_q(\mathfrak{sl}_2)$. We will use the equitable presentation,
 which was introduced in
\cite{equit}. See also
\cite{alnajjar, neubauer, huang, irt, uqsl2hat, nonnil, qtet, tersym, uawe,fduq,boyd}.
Fix a nonzero $q \in \mathbb F$ such that $q^2\not=1$. By
\cite[Theorem~2.1]{equit}
the equitable presentation of 
 $U_q(\mathfrak{sl}_2)$ has generators
$x,y^{\pm 1},z$ and relations $yy^{-1} =1$, $y^{-1}y = 1$,
\begin{eqnarray*}
\frac{qxy-q^{-1}yx}{q-q^{-1}}=1,
\qquad
\frac{qyz-q^{-1}zy}{q-q^{-1}}=1,
\qquad
\frac{qzx-q^{-1}xz}{q-q^{-1}}=1.
\end{eqnarray*}
Following \cite[Definition~3.1]{uawe}
define
\begin{eqnarray*}
\nu_x = q(1-yz), \qquad
\nu_y = q(1-zx), \qquad
\nu_z = q(1-xy).
\end{eqnarray*}
By \cite[Lemma~3.5]{uawe},
\begin{eqnarray*}
&&x \nu_y = q^2 \nu_y x, \qquad x \nu_z = q^{-2} \nu_z x,
\\
&&y \nu_z = q^2 \nu_z y, \qquad y \nu_x = q^{-2} \nu_x y,
\\
&&z \nu_x = q^2 \nu_x z, \qquad z \nu_y = q^{-2} \nu_y z.
\end{eqnarray*}
For the moment, assume that $q$ is not a root of unity.
Pick $N \in \mathbb N$ and let $V$ denote an irreducible
 $U_q(\mathfrak{sl}_2)$-module of dimension $N+1$.
 By \cite[Lemma~8.2]{fduq} each of
 $
 \nu^{N+1}_x,
 \nu^{N+1}_y,
 \nu^{N+1}_z
 $ 
is zero on $V$. Moreover by 
 \cite[Lemma~8.3]{fduq}, each of the following three sequences
 is a flag on $V$:
 \begin{equation*}
\lbrace \nu^{N-i}_{ x}V\rbrace_{i=0}^N,
\qquad
\lbrace \nu^{N-i}_{ y} V\rbrace_{i=0}^N,
\qquad
\lbrace \nu^{N-i}_{ z}V\rbrace_{i=0}^N.
\end{equation*}
We will show that these three flags
are totally opposite. We will further show
that for the corresponding Billiard Array on $V$, the
value of each white 3-clique is $q^{-2}$. 
\medskip

\noindent
We just obtained a Billiard Array from each finite-dimensional
irreducible 
 $U_q(\mathfrak{sl}_2)$-module, under the assumption that
 $q$ is not a root of unity. Similarly, for the
 Lie algebra    
 $\mathfrak{sl}_2$ over $\mathbb F$,
 we will obtain a Billiard Array from each finite-dimensional irreducible
 $\mathfrak{sl}_2$-module, under the assumption that
 $\mathbb F$ has characteristic 0. We will show that for these Billiard
 Arrays, the value of each white 3-clique is 1.
 \medskip
 
 \noindent We have been discussing 
 Billiard Arrays. Later in this paper we will introduce
 the concept of a Concrete Billiard Array
 and an edge-labelling of $\Delta_N$. These concepts help
 to clarify the Billiard Array theory, and will be used in our proofs.
 They may be of independent
 interest.
 \medskip

 \noindent This paper is organized as follows. Section 2 contains some
 preliminaries. In Sections 3--5 we consider the set $\Delta_N$
 from various points of view. Section 6 is about flags.
 Section 7 contains the definition and basic facts about
 Billiard Arrays. Section 8 contains a similar treatment
 for Concrete Billiard Arrays. Sections 9--12 are devoted to
 the correspondence between Billiard Arrays and 3-tuples of totally opposite
 flags. Sections 13--19 are devoted to our classification of Billiard Arrays
 up to isomorphism. Section 20 contains some examples of Concrete
  Billiard Arrays. In Section 21 we use Billiard Arrays to
 describe the finite-dimensional irreducible modules for $U_q(\mathfrak{sl}_2)$
 and 
  $\mathfrak{sl}_2$.

\section{Preliminaries}

\noindent 
We now begin our formal argument.
Let
$\mathbb R$ denote the field of real numbers.
We will be discussing the vector space $\mathbb R^3$
(row vectors). We will refer to the basis
\begin{equation*}
e_1 = (1,0,0),\qquad
e_2 = (0,1,0),\qquad
e_3 = (0,0,1).
\end{equation*}
\noindent 
Define a subset $\Phi \subseteq \mathbb R^3$ by
\begin{equation*}
\Phi = \lbrace e_i - e_j \;|\; 1 \leq i,j\leq 3, \; i\not=j\rbrace.
\end{equation*}
The set $\Phi$ is often called the root system $A_2$.
For notational convenience define
\begin{equation}
\label{eq:alphaBeta}
\alpha =  e_1-e_2,
\qquad \qquad 
\beta = e_2-e_3,
\qquad \qquad 
\gamma = e_3-e_1.
\end{equation}
Note that
\begin{eqnarray*}
\Phi = \lbrace \pm \alpha,\pm \beta,
\pm \gamma\rbrace,
\qquad \qquad
\alpha+\beta+\gamma=0.
\end{eqnarray*}
\noindent An element $(r,s,t)\in \mathbb R^3$ will be called
{\it nonnegative} 
whenever each of $r,s,t$ is nonnegative.
Define a partial order $\leq$ on $\mathbb R^3$
such that for $\lambda, \mu \in \mathbb R^3$,
$\mu \leq \lambda $ if and only if $\lambda-\mu$ is nonnegative.


\medskip
\noindent Let $\lbrace u_i\rbrace_{i=0}^n$ denote a  finite
sequence. We call $u_i$ the {\it $i$-component} or {\it $i$-coordinate}
of the
sequence. By the {\it inversion} of the sequence
$\lbrace u_i\rbrace_{i=0}^n$ we mean the sequence
$\lbrace u_{n-i}\rbrace_{i=0}^n$.

\section{The set $\Delta_N$}

Throughout this section fix $N \in \mathbb N$. 
\begin{definition}\rm
Let $\Delta_N$ denote the subset of $\mathbb R^3$ consisting of the three-tuples of natural
numbers
 whose sum is $N$. Thus
\begin{equation}
\Delta_N = \lbrace (r,s,t) \; |\;  r,s,t \in \mathbb N, \;\;r+s+t=N\rbrace.
\end{equation}
\noindent We arrange the elements of $\Delta_N$ in a triangular array,
as discussed in Section 1.
An element in $\Delta_N$ is called a {\it location}. 
For notational convenience define $\Delta_{-1}=\emptyset$.
\end{definition}

%


\begin{definition} 
\label{def:corner}
\rm
For $\eta \in \lbrace 1,2,3\rbrace$ the {\it $\eta$-corner}
of $\Delta_N$ is the location in $\Delta_N$
that has $\eta$-coordinate $N$ and all other
coordinates 0.  By a {\it corner} of $\Delta_N$
we mean the $1$-corner or
 $2$-corner or 
 $3$-corner.
The corners in $\Delta_N$ are listed below.
\begin{equation*}
Ne_1 = (N,0,0), \qquad Ne_2=(0,N,0), \qquad Ne_3=(0,0,N).
\end{equation*}
\end{definition}

\begin{definition}\rm 
For $\eta \in \lbrace 1,2,3\rbrace$ the
{\it $\eta$-boundary} of $\Delta_N$ is
the set of locations in $\Delta_N$ that have $\eta$-coordinate 0.
\end{definition}

\begin{example}
\rm
 The $1$-boundary of $\Delta_N$ consists of the locations
\begin{equation*}
(0,N-i,i) \qquad \qquad i=0,1,\ldots, N.
\end{equation*}
\end{example}

\begin{definition}\rm
The {\it boundary} of $\Delta_N$ is the union of its
1-boundary, 2-boundary, and 3-boundary.
By the {\it interior} of $\Delta_N$ we mean
the set of locations in $\Delta_N$
that are not on the boundary.
\end{definition}

\begin{definition}\rm
For $\eta \in \lbrace 1,2,3\rbrace$ we define a binary
relation on $\Delta_N$ called $\eta$-collinearity. 
By definition, locations $\lambda, \lambda' $ in $\Delta_N$
are {\it $\eta$-collinear} whenever the $\eta$-coordinate
  of $\lambda-\lambda'$ is 0.
Note that $\eta$-collinearity is an equivalence relation.
Each equivalence class will be called an {\it $\eta$-line}.
\end{definition}

\begin{example}
\rm
Pick $\eta \in \lbrace 1,2,3\rbrace$. For $0 \leq i \leq N$
there exists a unique $\eta$-line of $\Delta_N$ that has cardinality $i+1$.
For $i=0$ (resp. $i=N$) this $\eta$-line is the $\eta$-corner 
(resp. $\eta$-boundary) of $\Delta_N$.
\end{example}

\begin{example} 
\label{ex:line}
\rm For $0 \leq i \leq N$ the
following locations make up 
the unique $1$-line of $\Delta_N$ that has
cardinality $i+1$:
\begin{equation}
\label{eq:oneline}
(N-i, i-j, j) \qquad \qquad j=0,1,\ldots,i.
\end{equation}
\end{example}

\begin{definition}\rm
Locations $\lambda, \lambda'$ in $\Delta_N$ are called {\it collinear}
whenever they are 1-collinear or 2-collinear or 3-collinear. 
By a {\it line} in $\Delta_N$ we mean
a 1-line
 or 2-line
 or 3-line.
\end{definition}


\section{$\Delta_N$ as a graph}

\noindent  
Throughout this section fix $N \in \mathbb N$.
In this section we describe $\Delta_N$ using
notions from graph theory. For each result that we mention,
 the proof is routine and omitted.

\begin{definition}
\rm Locations $\lambda, \mu $ in $\Delta_N$ are called {\it adjacent}
whenever $\lambda - \mu \in \Phi$.
\end{definition}

\begin{example}\rm Assume $N\geq 1$, and pick a location 
$\lambda \in \Delta_N$.
\begin{enumerate}
\item[\rm (i)] Assume that $\lambda$ is a corner.
Then $\lambda$ is adjacent to exactly 2 locations in $\Delta_N$.
\item[\rm (ii)] Assume that $\lambda$ is on the boundary,
but not a corner.
Then $\lambda$ is adjacent to exactly 4 locations in $\Delta_N$.
\item[\rm (iii)] Assume that $\lambda$ is in the interior.
Then 
$\lambda$ is adjacent to exactly six locations in $\Delta_N$.
\end{enumerate}
\end{example}

\begin{definition}\rm By  an {\it edge} in $\Delta_N$
we mean a set of two adjacent locations.
\end{definition}

\begin{definition}\rm
For $n \in \mathbb N$, by a {\it walk of length $n$}
in $\Delta_N$ we mean a sequence of locations
$\lbrace\lambda_i\rbrace_{i=0}^n$  in $\Delta_N$
such that $\lambda_{i-1}, \lambda_i$
are adjacent for $1 \leq i \leq n$.
This walk is said to be {\it from $\lambda_0$ to $\lambda_n$}.
By a {\it path of length $n$}
in $\Delta_N$ we mean a walk
$\lbrace\lambda_i\rbrace_{i=0}^n$  in $\Delta_N$ 
such that $\lambda_{i-1} \not=\lambda_{i+1}$
for $1 \leq i \leq n-1$.
By a {\it cycle} in $\Delta_N$ we mean a path
$\lbrace \lambda_i\rbrace_{i=0}^n$ in $\Delta_N$ of length $n\geq 3$
such that 
$\lambda_0, \lambda_1, \ldots, \lambda_{n-1}$
are mutually distinct
and $\lambda_0 = \lambda_n$. 
\end{definition}

\begin{definition}\rm
For locations $\lambda, \lambda' \in \Delta_N$
let $\partial(\lambda, \lambda')$ denote the length of
a shortest path from $\lambda$ to $\lambda'$.
We call $\partial(\lambda,\lambda')$ the {\it distance}
between $\lambda$ and $\lambda'$.
\end{definition}

\begin{example}\rm
A location $\lambda = (r,s,t)$ in $\Delta_N$
is at distance $N-r$ (resp. $N-s$) (resp. $N-t$) from
the $1$-corner
(resp. $2$-corner) (resp. $3$-corner) of $\Delta_N$.
\end{example}

\begin{definition}\rm For a location $\lambda \in \Delta_N$
and a nonempty subset $S \subseteq \Delta_N$, define
\begin{equation*}
\partial (\lambda, S) = {\rm min} \lbrace \partial (\lambda, \lambda')
\;|\; \lambda' \in S\rbrace.
\end{equation*}
We call 
$\partial (\lambda, S)$ the {\it distance} between
$\lambda$ and $S$.
\end{definition}

\begin{example} \rm
A location $\lambda = (r,s,t)$ in $\Delta_N$ is
at distance $r$ (resp. $s$) (resp. $t$)
from the $1$-boundary
(resp. $2$-boundary) 
(resp. 
$3$-boundary) of $\Delta_N$.
\end{example}

\begin{lemma}
\label{lem:linedist}
 For $\eta \in \lbrace 1,2,3\rbrace$ and
$0 \leq n \leq N$ the following sets coincide:
\begin{enumerate}
\item[\rm (i)] the locations in $\Delta_N$
that have $\eta$-coordinate $n$;
\item[\rm (ii)] the locations in $\Delta_N$ at distance $n$ from
the $\eta$-boundary;
\item[\rm (iii)] the locations in $\Delta_N$ at distance $N-n$ from
the $\eta$-corner;
\item[\rm (iv)] the $\eta$-line of cardinality $N-n+1$.
\end{enumerate}
\end{lemma}

\noindent We now consider the distance function $\partial$ in
more detail.

\begin{lemma}
\label{lem:distCalc}
Let $\lambda$ and $\lambda'$ denote
locations in $\Delta_N$, written $\lambda=(r,s,t)$
and $\lambda'=(r',s',t')$.
Then 
 $\partial(\lambda, \lambda')$ is equal to 
the maximum of the absolute values
\begin{equation}
|r-r'|, \qquad 
|s-s'|, \qquad 
|t-t'|.
\label{eq:abs}
\end{equation}
\end{lemma}

\begin{lemma}
\label{lem:d1d2d3}
Given locations $(r,s,t)$
and $(r',s',t')$ in $\Delta_N$,
consider the three quantities in line
{\rm (\ref{eq:abs})}.
Let $d_1, d_2, d_3$ denote an ordering of these
quantities such that $d_1 \leq d_2 \leq d_3$.
Then $d_1 + d_2 = d_3$.
\end{lemma}

\begin{lemma}
\label{lem:mustbecorner}
For locations $\lambda, \lambda'$ in $\Delta_N$ we have
$\partial(\lambda, \lambda')\leq N$.  Moreover 
the following
are equivalent:
\begin{enumerate}
\item[\rm (i)]
$\partial(\lambda, \lambda')=N$;
\item[\rm (ii)]
there exists $\eta \in \lbrace 1,2,3\rbrace$ such that
one of 
$\lambda, \lambda'$ is the $\eta$-corner of $\Delta_N$, and
the other one is on the $\eta$-boundary of $\Delta_N$.
\end{enumerate}
\end{lemma}

\noindent We mention several types of paths in $\Delta_N$.

\begin{definition}
\label{def:pathtype}
\rm
Pick $\eta \in \lbrace 1,2,3\rbrace$.
A path 
$\lbrace \lambda_i \rbrace_{i=0}^n$ in $\Delta_N$
is called an 
{\it $\eta$-boundary path}
whenever $\lambda_i$ is on the $\eta$-boundary 
for $0 \leq i \leq n$.
The path 
$\lbrace \lambda_i \rbrace_{i=0}^n$ is called {\it $\eta$-linear}
whenever there exists an $\eta$-line  that contains 
$\lambda_i$ for $0 \leq i \leq n$.
\end{definition}


\begin{definition}
\label{lem:distcoord2}
\rm
Pick $\eta \in \lbrace 1,2,3\rbrace$.
A subset $S$ of $\Delta_N$
is called
{\it $\eta$-geodesic}
whenever distinct elements in $S$ do not
have the same $\eta$-coordinate.
\end{definition}

\begin{lemma}
\label{def:etaGeo}
Pick $\eta \in \lbrace 1,2,3\rbrace$.
For a path
$\lbrace \lambda_i \rbrace_{i=0}^n$ in $\Delta_N$
the following are equivalent:
\begin{enumerate}
\item[\rm (i)]
the path $\lbrace \lambda_i \rbrace_{i=0}^n$ 
is $\eta$-geodesic;
\item[\rm (ii)]
there exists $\varepsilon \in \lbrace 1,-1\rbrace$ such that
for $1 \leq i \leq n$, 
the $\eta$-coordinate of $\lambda_{i}$ is equal to
$\varepsilon$ plus the $\eta$-coordinate of $\lambda_{i-1}$.
\end{enumerate}
\end{lemma}

\begin{lemma} 
Pick $\eta \in \lbrace 1,2,3\rbrace$. For
a path $\lbrace \lambda_i\rbrace_{i=0}^n$ in $\Delta_N$
the following are equivalent:
\begin{enumerate}
\item[\rm (i)]
the path 
 $\lbrace \lambda_i\rbrace_{i=0}^n$ 
is $\eta$-linear;
\item[\rm (ii)]
the path
 $\lbrace \lambda_i\rbrace_{i=0}^n$ 
is $\xi$-geodesic for each $\xi \in \lbrace 1,2,3\rbrace$
other than $\eta$.
\end{enumerate}
\end{lemma}

\begin{definition}
\label{def:Geo}
\rm A path
$\lbrace \lambda_i\rbrace_{i=0}^n$ 
in $\Delta_N$ is said to be {\it geodesic} whenever
$\partial(\lambda_0, \lambda_n)=n$.
\end{definition}

\begin{lemma} 
\label{lem:etaGeo}
For a path 
$\lbrace \lambda_i\rbrace_{i=0}^n$  in $\Delta_N$
the following are equivalent:
\begin{enumerate}
\item[\rm (i)] the path 
$\lbrace \lambda_i\rbrace_{i=0}^n$
is geodesic;
\item[\rm (ii)] there exists $\eta \in \lbrace 1,2,3\rbrace$
such that 
the path $\lbrace \lambda_i\rbrace_{i=0}^n$  is $\eta$-geodesic.
\end{enumerate}
\end{lemma}

\begin{lemma}  Let $\lambda$ and $\lambda'$ denote
locations in $\Delta_N$, written $\lambda=(r,s,t)$
and $\lambda'=(r',s',t')$.
Then the number of geodesic paths in $\Delta_N$ from
$\lambda$ to $\lambda'$ is equal to the binomial
coefficient $\binom{d_1+d_2}{d_1}$
with $d_1, d_2$  from Lemma
\ref{lem:d1d2d3}.
\end{lemma}

\begin{lemma}
\label{lem:collinGeo}
For locations $\lambda, \lambda'$ in  $\Delta_N$
the following are equivalent:
\begin{enumerate}
\item[\rm (i)] the locations $\lambda, \lambda'$ are collinear;
\item[\rm (ii)] there exists a unique geodesic path in $\Delta_N$ from
$\lambda$ to $\lambda'$.
\end{enumerate}
Assume that {\rm (i), (ii)} hold. Define $\eta \in \lbrace 1,2,3\rbrace$
such that $\lambda, \lambda'$ are $\eta$-collinear.
Then the geodesic path mentioned in {\rm (ii)}
is $\eta$-linear.
\end{lemma}

\noindent To motivate the next definition we make some comments.
Let $\lambda, \lambda'$ 
denote distinct corner locations
in $\Delta_N$. Then $\lambda, \lambda'$ are collinear, and
$\partial(\lambda,\lambda') = N$.
There exists a unique geodesic path
in $\Delta_N$ from
$\lambda $ to $\lambda' $.
This is a boundary path. 

\begin{definition}
\label{def:sixpaths}
\rm For distinct $\eta, \xi \in \lbrace 1,2,3\rbrace$ 
let $\lbrack \eta, \xi\rbrack$ denote the unique
geodesic path in $\Delta_N$
from the $\eta$-corner of $\Delta_N$ to the
 $\xi$-corner of $\Delta_N$.
\end{definition}

\begin{lemma}
\label{lem:boundarypathInv}
 For distinct $\eta, \xi \in \lbrace 1,2,3\rbrace$ 
the path 
$\lbrack \eta, \xi\rbrack$ is the inversion of
the path $\lbrack \xi, \eta\rbrack$.
\end{lemma}

\begin{lemma} 
\label{lem:sixpathsinfo}
 Pick distinct $\eta, \xi \in \lbrace 1,2,3\rbrace$ 
and let 
$\lbrace \lambda_i \rbrace_{i=0}^N$
 denote the
 path
 $\lbrack \eta, \xi\rbrack$.
For $0 \leq i \leq N$ the location $\lambda_i$
is described in the table below:
\bigskip

\centerline{
\begin{tabular}[t]{c|c}
 $ \;\;\lbrack \eta,\xi\rbrack \;\;$ &
    $\lambda_i$ 
   \\ \hline  \hline
$\lbrack 1,2\rbrack$ &
$(N-i,i,0)$ 
  \\ 
$\lbrack 2,1\rbrack$ &
 $(i,N-i,0)$
   \\
\hline
$\lbrack 2,3\rbrack$ &
 $(0,N-i,i)$ 
  \\ 
$\lbrack 3,2\rbrack$ &
  $(0,i,N-i)$
   \\
\hline
$\lbrack 3,1\rbrack$ &
$(i,0,N-i)$ 
  \\ 
$\lbrack 1,3\rbrack$ &
 $(N-i,0,i)$
   \\
     \end{tabular}}
     \bigskip

\end{lemma}

\begin{definition}
\label{def:GC}
\rm 
A subset $S$ of $\Delta_N$ is
 said to be {\it geodesically closed}
whenever for each geodesic path $\lbrace \lambda_i\rbrace_{i=0}^n$
in $\Delta_N$, if $S$ contains $\lambda_0, \lambda_n$
then $S$ contains $\lambda_i$ for $0 \leq i \leq n$.
\end{definition}

\begin{example}
\label{ex:GC}
\rm
For $\eta \in \lbrace 1,2,3\rbrace $ and
$0 \leq n \leq N$, the set of
locations in $\Delta_N$ that have
$\eta$-coordinate at least $n$ is geodesically closed.
\end{example}

\begin{lemma} 
\label{lem:GCint}
Let $S$ and $S'$ denote geodesically
closed subsets of $\Delta_N$. Then
$S\cap S'$ is geodesically closed.
\end{lemma}

\begin{definition} 
\label{def:spanningT}
\rm By a {\it spanning tree} of $\Delta_N$
we mean a set $T$ of edges for  
$\Delta_N$ that has the following property: for any two 
distinct locations $\lambda, \lambda'$ in $\Delta_N$
there exists a unique path in $\Delta_N$ from 
$\lambda$ to $\lambda'$ that involves only edges in $T$. 
\end{definition}

\begin{note}\rm Let $T$ denote a spanning tree of $\Delta_N$.
Then the cardinality of $T$ is one less than the cardinality of
$\Delta_N$.
\end{note}

\begin{definition}\rm
By a {\it 3-clique} in $\Delta_N$ we mean a set of three mutually
adjacent locations in $\Delta_N$. 
There are two kinds of 3-cliques: $\Delta$ (black)
and $\nabla$ (white).
\end{definition}

\begin{example}\rm Assume $N=0$. Then $\Delta_N$ does not
contain a 3-clique. Assume $N=1$. Then $\Delta_N$ contains 
a unique black 3-clique and no white 3-clique.
Assume  $N=2$. Then $\Delta_N$ contains three black 3-cliques
and a unique white 3-clique.
\end{example}

\begin{lemma}
\label{lem:white}
Assume $N\geq 1$.
We describe a bijection from $\Delta_{N-1}$ to the
set of black 3-cliques in $\Delta_N$. The
bijection sends each
$(r,s,t) \in \Delta_{N-1}$ to the black 3-clique consisting
of the locations 
\begin{equation*}
(r+1,s,t), \qquad (r,s+1,t), \qquad (r,s,t+1).
\end{equation*}
\end{lemma}

\begin{lemma}
\label{lem:black}
\rm Assume $N\geq 2$. 
We describe a bijection from $\Delta_{N-2}$ to the
set of white 3-cliques in $\Delta_N$. The
bijection sends each
$(r,s,t) \in \Delta_{N-2}$ to the white 3-clique consisting
of the locations 
\begin{equation*}
(r,s+1,t+1), \qquad (r+1,s,t+1), \qquad (r+1,s+1,t).
\end{equation*}
\end{lemma}

\begin{lemma} 
\label{lem:edgegivesclique}
For $\Delta_N$, each edge is contained in a unique black 3-clique
and at most one white 3-clique.
\end{lemma}


\section{The poset $\Delta_{\leq N}$ }

\noindent Throughout this section fix $N \in \mathbb N$.
We now describe $\Delta_N$ 
using
notions from the theory of posets. We will follow
the notational conventions from
\cite{posets}.

\begin{definition}\rm Let 
$\Delta_{\leq N}$ denote the poset consisting of
the set $\cup_{n=0}^N \Delta_n$ together with the partial
order $\leq$ from Section 2. 
An element 
$\lambda$ in the poset is said to have
{\it rank $n$} whenever $\lambda \in \Delta_n$.
\end{definition}

\noindent We mention some facts about 
the poset
$\Delta_{\leq N}$. 
Pick elements $\lambda, \mu$ in the poset.
Then $\lambda$ covers $\mu$ if and only if
$\lambda-\mu$ is one of $e_1, e_2, e_3$. 
In this case ${\rm rank}(\lambda)=
1+ {\rm rank}(\mu)$.
For $0 \leq n \leq N-1$, each element in $\Delta_n$
is covered by precisely three elements in the poset.
An element in $\Delta_N$ is not covered by any element 
in the poset.
For $1 \leq n \leq N$, each element in $\Delta_n$
 covers at least one element in the poset.
The element in $\Delta_0$ does not cover any
element in the poset.

\medskip

\noindent Given elements $\lambda, \mu$ in 
$\Delta_{\leq N}$
we define an element $\lambda \wedge \mu $ in   
$\Delta_{\leq N}$ as follows: for $\eta \in \lbrace 1,2,3\rbrace$
the $\eta$-coordinate of 
 $\lambda \wedge \mu $ is the minimum of the
$\eta$-coordinates for $\lambda$ and $\mu$. 
Note that $\lambda \wedge\mu \leq \lambda $ and
$\lambda \wedge\mu \leq \mu$.
Moreover 
$\nu \leq \lambda\wedge\mu$ for all
elements $\nu$ in
$\Delta_{\leq N}$ such that
$\nu\leq \lambda$ and
$\nu \leq \mu$.
\medskip

\noindent Given elements $\lambda, \mu$ in 
$\Delta_{\leq N}$
we define an element
$\lambda\vee \mu$ in $\mathbb R^3$
as follows:
for $\eta \in \lbrace 1,2,3\rbrace$
the $\eta$-coordinate of 
 $\lambda \vee \mu $ is the maximum of the
$\eta$-coordinates for $\lambda$ and $\mu$. 
Note that 
$\lambda\vee \mu$ is contained in $\Delta_{\leq N}$ 
if and only if the sum of its coordinates is
at most $N$. In this case we say that
{\it $\lambda\vee \mu$ exists} in $\Delta_{\leq N}$. 
Assume that 
$\lambda\vee \mu$ exists in $\Delta_{\leq N}$. 
Then  $\lambda \leq \lambda\vee\mu$ and
$\mu \leq \lambda\vee\mu$.
Moreover 
$\lambda\vee\mu \leq \nu$ for all
elements $\nu$ in 
$\Delta_{\leq N}$ such that
$\lambda \leq \nu$ and
$\mu \leq \nu$. Now assume that $\lambda \vee \mu$ does
not exist in $\Delta_{\leq N}$. Then
$\Delta_{\leq N}$ does not contain an element
$\nu$ such that
$\lambda \leq \nu$ and
$\mu \leq \nu$.

\begin{lemma} 
\label{lem:findrank}
For locations $\lambda, \mu$ in
$\Delta_N$ the rank of 
$\lambda \wedge \mu$ is equal to $N-\partial(\lambda, \mu)$.
\end{lemma}
\begin{proof} Write $\lambda = (r,s,t)$ and
$\mu=(r',s',t')$. By construction $r+s+t=N$ 
and
$r'+s'+t'=N$. The sum of $r-r', s-s', t-t'$ is zero. 
Permuting the coordinates of $\mathbb R^3$  and
interchanging $\lambda, \mu$ if necessary,
we may assume without loss that each of
$r-r',s'-s,t'-t$ is nonnegative. 
By Lemma
\ref{lem:distCalc}
$\partial(\lambda, \mu)=r-r'$.
By construction
$\lambda\wedge \mu = (r',s,t)$ has
 rank
 $r'+s+t$.
The result follows.
\end{proof}

\noindent
For $\mu \in \Delta_{\leq N}$ we now describe the set
\begin{equation}
\label{eq:upset}
\lbrace \lambda \in \Delta_N\;|\; \mu \leq \lambda\rbrace.
\end{equation}

\begin{lemma}
\label{lem:posetmeaning}
For  $\mu=(r,s,t) \in  \Delta_{\leq N}$
and $\lambda \in \Delta_N$, the following are equivalent:
\begin{enumerate}
\item[\rm (i)] $\mu \leq \lambda$;
\item[\rm (ii)] the 
$1$-coordinate (resp. $2$-coordinate)
(resp. $3$-coordinate) of $\lambda$ is at least 
$r$ (resp. $s$) (resp. $t$);
\item[\rm (iii)] $\lambda$ is at distance at least 
$r$ (resp. $s$) (resp. $t$) from the $1$-boundary
(resp. $2$-boundary) (resp. $3$-boundary) of $\Delta_N$;
\item[\rm (iv)] $\lambda$ is at distance at most
$N-r$ (resp. $N-s$) (resp. $N-t$) from the $1$-corner
(resp. $2$-corner) (resp. $3$-corner) of $\Delta_N$;
\item[\rm (v)] the 
$1$-line (resp. $2$-line)
(resp. $3$-line) of $\Delta_N$ that contains $\lambda$ has cardinality
at most 
$N-r+1$ (resp. $N-s+1$) (resp. $N-t+1$).
\end{enumerate}
\end{lemma}
\begin{proof}
${\rm (i)}\Leftrightarrow {\rm (ii)}$ By the definition
of the partial order $\leq$ in Section 2.
\\
\noindent
${\rm (ii)}\Leftrightarrow {\rm (iii)}
\Leftrightarrow {\rm (iv)}
\Leftrightarrow {\rm (v)}$
By Lemma
\ref{lem:linedist}.
\end{proof}

\begin{lemma} 
\label{lem:upsetdesc}
\label{eq:upsetdesc}
For $0 \leq n \leq N$ and $\mu \in \Delta_{N-n}$ 
the set 
{\rm (\ref{eq:upset})} is equal to 
$\Delta_n +\mu$.
\end{lemma}
\begin{proof} By construction.
\end{proof}

\begin{lemma}
\label{lem:gcsubset}
For $\mu \in \Delta_{\leq N}$ the set
{\rm (\ref{eq:upset})} is 
 geodesically closed.
\end{lemma}
\begin{proof}
By 
Example \ref{ex:GC} and Lemma
\ref{lem:GCint},
along with
Lemma
\ref{lem:posetmeaning}(i),(ii).
\end{proof}

\noindent We have been discussing the set
{\rm (\ref{eq:upset})}.
We now consider
 a special case.

\begin{definition}\rm
Pick $\eta \in \lbrace 1,2,3\rbrace$.
Note by Definition
\ref{def:corner} that for $0 \leq n \leq N$
the element $n e_\eta$ is the $\eta$-corner of
$\Delta_n$. By an {\it $\eta$-corner} of $\Delta_{\leq N}$
we mean one of
$\lbrace n e_\eta\rbrace_{n=0}^N$.
By a {\it corner} of
$\Delta_{\leq N}$ we mean a $1$-corner or $2$-corner or
$3$-corner of $\Delta_{\leq N}$.
\end{definition}

\noindent Referring to Lemma 
\ref{lem:posetmeaning},
we now consider the case in
which $\mu$ is a corner of $\Delta_{\leq N}$.

\begin{lemma} 
\label{lem:moredetail}
For $\eta \in \lbrace 1,2,3\rbrace$ and
 $0 \leq n \leq N$ and 
 $\lambda \in \Delta_N$, the following are equivalent:
\begin{enumerate}
\item[\rm (i)] $\mu \leq \lambda$, where $\mu$ is
the $\eta$-corner of $\Delta_{N-n}$;
\item[\rm (ii)] the $\eta$-coordinate of 
$\lambda$ is at least $N-n$;
\item[\rm (iii)] $\lambda$ is at distance at least $N-n$ from
the $\eta$-boundary of $\Delta_N$;
\item[\rm (iv)] $\lambda $ is at distance  at most $n$ from
the $\eta$-corner of $\Delta_N$;
\item[\rm (v)] the $\eta$-line containing $\lambda$ has cardinality
at most $n+1$.
\end{enumerate}
\end{lemma}
\begin{proof}
For the $\eta$-corner $\mu$ of $\Delta_{N-n}$, the
$\eta$-coordinate is $N-n$ and the other
two coordinates are 0. The result
follows in view of 
Lemma
\ref{lem:posetmeaning}.
\end{proof}

\begin{definition}\rm
 By a {\it flag} in $\Delta_{\leq N}$ we mean
a sequence $\lbrace \lambda_n \rbrace_{n=0}^N$
such that $\lambda_n \in \Delta_n$ for 
$0 \leq n\leq N$ and $\lambda_{n-1}\leq \lambda_{n}$
for $1 \leq n \leq N$.
\end{definition}

\begin{definition}
\label{eq:flagex}
\rm
Pick $\eta \in \lbrace 1,2,3\rbrace $.
Observe that the sequence of $\eta$-corners 
$\lbrace n e_{\eta} \rbrace_{n=0}^N$ is a flag in
 $\Delta_{\leq N}$. We denote this flag by
$\lbrack \eta \rbrack$.
\end{definition}

\section{The vector space $V$ and the poset $\mathcal P(V)$}

\noindent From now on, let $\mathbb F$ denote
a field. Fix $N \in \mathbb N$, and 
let $V$ denote a vector space over
$\mathbb F$ with dimension $N+1$.
Let
 ${\rm End}(V)$ denote the $\mathbb F$-algebra consisting
 of
 the $\mathbb F$-linear maps from $V$ to $V$.
 Let 
 $I \in {\rm End}(V)$ denote the identity map on $V$.
By a {\it decomposition} of $V$ we mean a sequence 
$\lbrace V_i\rbrace_{i=0}^N$
consisting 
of
one-dimensional subspaces 
of $V$ such that
$V = \sum_{i=0}^N V_i$ (direct sum).
For a decomposition of $V$, its inversion is a 
decomposition of $V$.
Let $\lbrace v_i\rbrace_{i=0}^N$ denote a basis for $V$
and let 
$\lbrace V_i\rbrace_{i=0}^N$ denote a decomposition of $V$.
We say that
$\lbrace v_i\rbrace_{i=0}^N$ {\it induces}
$\lbrace V_i\rbrace_{i=0}^N$
whenever $v_i \in V_i$ for $0 \leq i \leq N$.
\medskip

\noindent 
Let $\mathcal P(V)$ denote the set of
nonzero subspaces of $V$.
The containment relation $\subseteq$  is a partial order on $\mathcal P(V)$.
For $1 \leq i \leq N+1$
let $\mathcal P_i(V)$ denote the set of
$i$-dimensional subspaces of $V$.
This yields a partition
$\mathcal P(V) = \cup_{i=1}^{N+1} \mathcal P_i(V)$.
\medskip

\noindent 
By a {\it flag} in $\mathcal P(V)$
we mean a sequence
$\lbrace U_i\rbrace_{i=0}^N$ such that
 $U_i \in \mathcal P_{i+1}(V)$ for $0\leq i \leq N$
and $U_{i-1}\subseteq U_i$ for $1 \leq i \leq N$.
For notational convenience, a flag in
${\mathcal P}(V)$ is said to be {\it 
on 
$V$}.
The following construction yields a flag 
on $V$.
Let $\lbrace V_i\rbrace_{i=0}^N$ denote a decomposition
of $V$. For $0 \leq i \leq N$ define
$U_i = V_0 + \cdots + V_i$. Then the sequence
$\lbrace U_i\rbrace_{i=0}^N$ is a flag on $V$.
This flag is said to be {\it induced} by 
the decomposition $\lbrace V_i\rbrace_{i=0}^N$.
Suppose we are given two flags on $V$, denoted
$\lbrace U_i\rbrace_{i=0}^N$ and
$\lbrace U'_i\rbrace_{i=0}^N$.
Then the following are equivalent:
(i) $U_i \cap U'_j= 0 $  if $i+j<N$ $(0 \leq i,j\leq N)$;
(ii) 
there exists a decomposition 
$\lbrace V_i\rbrace_{i=0}^N$ of $V$ that induces 
$\lbrace U_i\rbrace_{i=0}^N$ and
whose inversion induces
$\lbrace U'_i\rbrace_{i=0}^N$.
The flags
are called {\it opposite} whenever (i), (ii) hold.
In this case $V_i = U_i \cap U'_{N-i}$ for
$0 \leq i \leq N$. 
\medskip

\noindent We mention a result for later use.
\begin{lemma} 
\label{lem:OppDim}
Let 
$\lbrace U_i\rbrace_{i=0}^N$
and $\lbrace U'_i\rbrace_{i=0}^N$ denote
opposite flags on $V$. Pick integers
$r,s$ $(0 \leq r,s\leq N)$ such that
$r+s\geq N$. Then $U_r\cap U'_s$ has dimension $r+s-N+1$.
\end{lemma}
\begin{proof} There exists a decomposition $\lbrace V_i\rbrace_{i=0}^N$
of $V$ that induces 
$\lbrace U_i\rbrace_{i=0}^N$
and whose inversion induces
$\lbrace U'_i\rbrace_{i=0}^N$.
Observe that $U_r \cap U'_s = \sum_{i=N-s}^r V_i$.
The sum is direct, so it has dimension $r+s-N+1$.
\end{proof}

\section{Billiard Arrays}

Throughout this section the following notation is
in effect. 
Fix $N \in \mathbb N$.
Let $V$ denote a vector space over $\mathbb F$ with
dimension $N+1$.

\begin{definition}
\label{def:ba}
\rm
By a {\it Billiard Array on $V$} we mean a function
$B:\Delta_N \to \mathcal P_1(V)$, $\lambda \mapsto B_\lambda$
that satisfies the following conditions:
\begin{enumerate}
\item[\rm (i)] for each line $L$ in $\Delta_N$ 
the sum $\sum_{\lambda \in L} B_\lambda$ 
 is direct;
\item[\rm (ii)] for each black 3-clique $C$ in $\Delta_N$ the sum 
 $\sum_{\lambda \in C} B_\lambda$ 
is not direct.
\end{enumerate}
We say that $B$ is {\it over} $\mathbb F$.
We call $V$ the {\it underlying vector space}.
We call $N$ the {\it diameter} of $B$.
\end{definition}


\begin{note}
\label{note:interp}
\rm 
\rm We will show that the Billiard Array $B$
in Definition
\ref{def:ba} is injective.
\end{note}

\noindent 
Let $B$ denote a Billiard Array on $V$,
as in 
Definition
\ref{def:ba}.
We view $B$ as an arrangement
of one-dimensional subspaces of $V$ into a triangular
array, with  $B_{\lambda}$ at location $\lambda$ 
for all $\lambda \in \Delta_N$.
 For
 $U \in \mathcal P_1(V)$ and
 $\lambda \in \Delta_N$,
 we say that {\it  $U$ is included in $B$ at location $\lambda$}
 whenever
 $U=B_{\lambda}$.


\begin{lemma} 
\label{lem:threedim2}
Let $B$ denote a Billiard Array on $V$.
Let $\lambda,\mu, \nu$ denote locations in $\Delta_N$ that form
a black 3-clique. Then the subspace $B_\lambda+B_\mu+B_\nu$ 
is equal to each of
\begin{equation}
\label{eq:directsum}
B_\lambda+B_\mu, 
\qquad
B_\mu+B_\nu, 
\qquad
B_\nu+B_\lambda.
\end{equation}
Moreover, in line 
{\rm (\ref{eq:directsum})} each sum is direct.
\end{lemma}
\begin{proof}
Any two of $\lambda, \mu,\nu$ are collinear, so
in line
{\rm (\ref{eq:directsum})} each sum is direct by
Definition 
\ref{def:ba}(i).
The sum
 $B_\lambda+B_\mu+B_\nu$  is not direct by Definition
\ref{def:ba}(ii).
The result follows.
\end{proof}

\begin{corollary}
\label{cor:usefuldep}
Let $B$ denote a Billiard Array on $V$.
Let $\lambda,\mu, \nu$ denote locations in $\Delta_N$ that form
a black 3-clique. Then each of
$B_\lambda, B_\mu, B_\nu$ is contained in the sum
of the other two.
\end{corollary}

\begin{lemma} 
\label{lem:sum}
Let $B$ denote a Billiard Array on $V$.
Let $\lambda,\mu, \nu$ denote locations in $\Delta_N$ that form
a white 3-clique. Then the sum
$B_\lambda+B_\mu+B_\nu$ is direct.
\end{lemma}
\begin{proof} Permuting $\lambda, \mu, \nu$ if
necessary, we may assume that
\begin{equation*}
\mu-\lambda =\alpha,
\qquad \quad
\nu-\mu =\beta,
\qquad \quad
\lambda-\nu =\gamma,
\end{equation*}
where $\alpha,\beta,\gamma$ are from
(\ref{eq:alphaBeta}).
Consider the locations $\mu+\gamma,\mu,\mu-\gamma$ in 
$\Delta_N$. These locations are collinear, so
by Definition
\ref{def:ba}(i)
the sum
$B_{\mu+\gamma}+ B_\mu + B_{\mu-\gamma}$ is direct.
The locations $\lambda,\mu, \mu+\gamma$
form a black 3-clique, so $B_{\mu+\gamma}$ is contained in
$B_\lambda + B_\mu$ by Corollary
\ref{cor:usefuldep}.
The locations $\mu, \nu,\mu-\gamma$
form a black 3-clique, so $B_{\mu-\gamma}$ is contained in
$B_\mu + B_\nu$ by
 Corollary
\ref{cor:usefuldep}.
By these comments
$B_{\mu+\gamma}+ B_\mu + B_{\mu-\gamma}$ is contained in
$B_\lambda+B_\mu+B_\nu$.
The result follows.
\end{proof}

\begin{lemma} 
\label{lem:BspanV}
Let $B$ denote a Billiard Array on $V$.
Pick $\eta \in \lbrace 1,2,3\rbrace $ and let $L$  denote
the $\eta$-boundary of $\Delta_N$. Then the sum
$V=\sum_{\lambda \in L} B_\lambda$ is direct.
\end{lemma}
\begin{proof} The 
sum
$\sum_{\lambda \in L} B_\lambda$ is direct by
Definition
\ref{def:ba}(i) and since
$L$ is a line. The sum is equal to $V$, since
$L$ has cardinality $N+1$ and
$V$ has dimension $N+1$.
\end{proof}

\begin{corollary} 
\label{lem:BspanV2}
Let $B$ denote a Billiard Array on $V$.
 Then $V$ is
spanned by $\lbrace B_\lambda\rbrace_{\lambda \in \Delta_N}$.
\end{corollary}
\begin{proof}
By Lemma
\ref{lem:BspanV}.
\end{proof}

\noindent We now describe the Billiard Arrays of diameter 0 and 1.

\begin{example}
Let $B:\Delta_N\to \mathcal P_1(V)$ denote any function.
\begin{enumerate}
\item[\rm (i)] Assume $N=0$,
so that $\Delta_N$ and $\mathcal P_1(V)$ each have
cardinality 1.
Then $B$ is a Billiard Array 
on $V$.
\item[\rm (ii)] Assume $N=1$.
Then $B$ is a Billiard Array on 
$V$ if and only if
$B$ is injective.
\end{enumerate}
\end{example}

\begin{definition}
\label{def:isoBA}
\rm
Let $V'$ denote a vector space over $\mathbb F$ with
dimension $N+1$.
Let $B$ (resp. $B'$) denote a Billiard Array
on $V$ (resp. $V'$).
By an {\it isomorphism of Billiard Arrays from 
$B$ to
$B'$} we mean an $\mathbb F$-linear bijection
$\sigma: V \to V'$ that sends
$
B_\lambda
\mapsto 
B'_{\lambda} $
for all $\lambda \in \Delta_N$.
The Billiard Arrays $B$ and
$B'$ are called {\it isomorphic} whenever
there exists an isomorphism of Billiard Arrays from
$B$ to
$B'$.
\end{definition}

\begin{example} For $N=0$ or $N=1$, any two Billiard Arrays
over $\mathbb F$ of diameter $N$ are isomorphic.
\end{example}

\section{Concrete Billiard Arrays}

Throughout this section the following notation is
in effect. Fix $N \in \mathbb N$. Let $V$ denote
a vector space over $\mathbb F$ with dimension $N+1$.
\medskip

\begin{definition}
\label{def:cba}
\rm By a {\it Concrete Billiard Array on $V$}
we mean a function $\mathcal B: \Delta_N \to V$,
$\lambda \mapsto \mathcal B_\lambda$ that satisfies
the following conditions:
\begin{enumerate}
\item[\rm (i)] for each line $L$ in $\Delta_N$
the vectors $\lbrace \mathcal B_\lambda \rbrace_{\lambda \in L}$
are linearly independent;
\item[\rm (ii)] for each black 3-clique $C$
 in $\Delta_N$
the vectors $\lbrace \mathcal B_\lambda \rbrace_{\lambda \in C}$
are linearly dependent.
\end{enumerate}
We say that $\mathcal B$ is {\it over} $\mathbb F$.
We call $V$ the {\it underlying vector space}.
We call $N$ the {\it diameter} of $\mathcal B$.
\end{definition}
\begin{lemma} 
\label{lem:betanon0}
Let $\mathcal B$ denote a Concrete Billiard Array
on $V$. Then $\mathcal B_\lambda \not=0$ for all $\lambda\in \Delta_N$.
\end{lemma}
\begin{proof} There exists a line $L$ in $\Delta_N$ that contains
$\lambda$. The result follows in view of 
Definition
\ref{def:cba}(i).
\end{proof}

\noindent Billiard Arrays and Concrete Billiard
Arrays are related as follows. 
Let $B$
denote a Billiard Array on $V$. For $\lambda \in \Delta_N$
let $\mathcal B_\lambda$ denote a nonzero vector in $B_\lambda$.
Then the function $\mathcal  
B:\Delta_N \to V$,
$\lambda \mapsto \mathcal B_\lambda$ is
 a Concrete Billiard Array on $V$.
Conversely,
let $\mathcal B$ denote a Concrete Billiard Array on $V$.
For $\lambda \in \Delta_N$ let
 $B_\lambda$ denote the
subspace of $V$ spanned by 
$\mathcal B_\lambda$.
 Note that $B_\lambda \in \mathcal P_1(V)$
by Lemma
\ref{lem:betanon0}.
The function $B:\Delta_N \to \mathcal P_1(V)$,
$\lambda \mapsto B_\lambda$ is
 a Billiard Array on $V$.

\begin{definition}
\label{def:corr}
\rm Let 
 $B$ denote a Billiard Array on $V$,
 and let $\mathcal B$ denote a Concrete Billiard Array on $V$.
 We say that $B, \mathcal B$ {\it correspond} whenever
 $\mathcal B_\lambda$ spans $B_\lambda$ for all 
 $\lambda \in \Delta_N$.
 \end{definition}

\noindent Let $\mathcal B$ denote a Concrete Billiard Array
on $V$. For $v \in V$ and $\lambda\in \Delta_N$, we say that
{\it $v$ is included in $\mathcal B$ at location $\lambda$}
whenever $v= {\mathcal B}_\lambda$.

\begin{lemma} 
\label{lem:threedim2CBA}
Let $\mathcal B$ denote a Concrete Billiard Array on $V$.
Let $\lambda,\mu, \nu$ denote locations in $\Delta_N$ that form
a black 3-clique. Then $\mathcal B_\lambda, \mathcal B_\mu, \mathcal B_\nu$ 
span a 2-dimensional subspace of $V$. Any two of
$\mathcal B_\lambda, \mathcal B_\mu, \mathcal B_\nu$  form a basis
for this subspace.
\end{lemma}
\begin{proof}
Use Lemma \ref{lem:threedim2}.
\end{proof}

\begin{lemma} 
\label{lem:sumCBA}
Let $\mathcal B$ denote a Concrete Billiard Array on $V$.
Let $\lambda,\mu, \nu$ denote locations in $\Delta_N$ that form
a white 3-clique. Then the vectors
$\mathcal B_\lambda, \mathcal B_\mu, \mathcal B_\nu$ are linearly
independent.
\end{lemma}
\begin{proof} 
Use Lemma \ref{lem:sum}.
\end{proof}

\begin{lemma} 
\label{lem:BspanVCBA}
Let $\mathcal B$ denote a Concrete Billiard Array on $V$.
Pick $\eta \in \lbrace 1,2,3\rbrace $ and let $L$  denote
the $\eta$-boundary of $\Delta_N$. Then the vectors
$\lbrace \mathcal B_\lambda\rbrace_{\lambda \in L}$ form
a basis for $V$.
\end{lemma}
\begin{proof}
Use 
Lemma \ref{lem:BspanV}.
\end{proof}

\begin{corollary} 
\label{lem:BspanV2CBA}
Let $\mathcal B$ denote a Concrete Billiard Array on $V$.
 Then $V$ is
spanned by the vectors
$\lbrace \mathcal B_\lambda\rbrace_{\lambda \in \Delta_N}$.
\end{corollary}
\begin{proof}
By Corollary 
\ref{lem:BspanV2} or
Lemma \ref{lem:BspanVCBA}.
\end{proof}

\noindent We now describe the Concrete Billiard Arrays of diameter 0 and 1.

\begin{example}
\label{ex:CBAd01}
Let $\mathcal B:\Delta_N\to V $ denote any function.
\begin{enumerate}
\item[\rm (i)] Assume $N=0$, so that $\Delta_N$ contains a unique element
$\lambda$. 
Then $\mathcal B$ is a Concrete Billiard Array 
on $V$ if and only if $\mathcal B_\lambda \not=0$. 
\item[\rm (ii)] Assume $N=1$.
Then $\mathcal B$ is a Concrete Billiard Array on 
$V$ if and only if
any two of $\lbrace \mathcal B_\lambda\rbrace_{\lambda \in \Delta_N}$ 
are linearly independent.
\end{enumerate}
\end{example}


\begin{lemma}
\label{lem:kappaAdj}
Let
 $\mathcal B$ denote a Concrete Billiard Array on $V$. 
For all $\lambda \in \Delta_N$ let $\kappa_\lambda$ denote a nonzero
scalar in $\mathbb F$. 
 Then the function
 $\mathcal B':\Delta_N \to V$,
 $\lambda \mapsto \kappa_\lambda \mathcal B_\lambda$
 is a Concrete Billiard Array on $V$.
\end{lemma}

\begin{definition}
\label{def:assoc}
\rm
 Let $\mathcal B,\mathcal B'$ denote Concrete
 Billiard Arrays on $V$. We say that  
 $\mathcal B, \mathcal B'$ are {\it associates}
whenever 
 there exist nonzero scalars 
 $\lbrace \kappa_\lambda\rbrace_{\lambda \in \Delta_N}$
 in $\mathbb F$ such that
 $\mathcal B'_\lambda=\kappa_\lambda \mathcal B_\lambda$
for all $\lambda \in \Delta_N$. 
The relation of being associates is an equivalence relation.
\end{definition}

\begin{lemma} Let
 $\mathcal B, \mathcal B'$ denote Concrete
 Billiard Arrays on $V$. 
Then 
 $\mathcal B, \mathcal B'$ are associates
if and only if they correspond to the same
Billiard
Array.
\end{lemma}
\begin{proof} By
Definitions
\ref{def:corr},
\ref{def:assoc}.
\end{proof}

\begin{example}\rm Referring to Definition
\ref{def:assoc},
if $N=0$ then
  $\mathcal B, \mathcal B'$ are associates. 
  \end{example}

\begin{definition}
\label{def:rel}
\rm
 Let $\mathcal B, \mathcal B'$ denote Concrete Billiard Arrays on $V$. 
Then 
  $\mathcal B, \mathcal B'$ are called {\it relatives}
whenever there exists $0 \not=\kappa \in \mathbb F$ such that
$\mathcal B'_\lambda = \kappa \mathcal B_\lambda$ for all
$\lambda \in \Delta_N$.
The relation of being relatives is an equivalence relation.
\end{definition}

\begin{example}\rm Referring to Definition
\ref{def:rel}, if $N=0$ then
  $\mathcal B, \mathcal B'$ are relatives.
  \end{example}

\begin{definition}
\label{def:cbaIso}
\rm
Let $V'$ denote a vector space over $\mathbb F$ with
dimension $N+1$.
Let $\mathcal B$ (resp. $\mathcal B'$) denote a Concrete Billiard Array
on $V$ (resp. $V'$).
By an {\it isomorphism of Concrete Billiard Arrays from 
$\mathcal B$ to
$\mathcal B'$} we mean an $\mathbb F$-linear bijection
$\sigma: V \to V'$ that sends
$
\mathcal B_\lambda
\mapsto 
\mathcal B'_{\lambda} $
for all $\lambda \in \Delta_N$.
We say that $\mathcal B,
\mathcal B'$ are {\it isomorphic} whenever
there exists an isomorphism of Concrete Billiard Arrays from
$\mathcal B$ to
$\mathcal B'$.
In this case the isomorphism is unique.
Isomorphism is an equivalence relation.
\end{definition}

\begin{example}
\rm Referring to Definition
\ref{def:cbaIso}, if $N=0$ then
$\mathcal B,
\mathcal B'$ are isomorphic.
\end{example}

\begin{definition}
\label{def:relations}
\rm
Let $V'$ denote a vector space over $\mathbb F$ with
dimension $N+1$.
Let $\mathcal B$ (resp. $\mathcal B'$) denote a Concrete Billiard Array
on $V$ (resp. $V'$).
Then $\mathcal B, \mathcal B'$ 
are called {\it similar} whenever their corresponding Billiard
Arrays are isomorphic in the sense of
Definition
\ref{def:isoBA}.
Similarity is an equivalence relation.
\end{definition}

\begin{example} \rm
Referring to Definition \ref{def:relations},
assume $N=0$ or $N=1$. Then
$\mathcal B, \mathcal B'$ are similar.
\end{example}

\noindent In each of the last four definitions we gave 
an equivalence relation for
Concrete Billiard Arrays.

\begin{lemma}
\label{lem:fourrel}
The above equivalence relations 
obey the
 following logical implications:

\begin{equation*}
\begin{CD}
\mbox{\rm relative}  @>>>
         \mbox{\rm isomorphic} 
	  \\ 
          @VVV                     @VVV \\
      \mbox{\rm associates}  @>>>
                       \mbox{\rm similar}
                   \end{CD}
\end{equation*}

\end{lemma}
\begin{proof} This is routinely verified.
\end{proof}

\noindent We have a few comments.

\begin{proposition}
\label{prop:isoAssoc}
Let $V'$ denote a vector space over $\mathbb F$ with
dimension $N+1$. 
Let $\mathcal B$ (resp.  $\mathcal B')$
denote a Concrete Billiard Array on $V$ (resp. $V'$).
Then the following are equivalent:
\begin{enumerate}
\item[\rm (i)]  $\mathcal B$ and $\mathcal B'$ are similar;
\item[\rm (ii)]  $\mathcal B$ is isomorphic to a
Concrete Billiard Array associated with
$\mathcal B'$;
\item[\rm (iii)]  $\mathcal B$ is associated with a
Concrete Billiard Array isomorphic to 
$\mathcal B'$.
\end{enumerate}
\end{proposition}
\begin{proof}
Let $B$ (resp. $B'$) denote the Billiard Array on $V$ (resp. $V'$)
that
corresponds to $\mathcal B$ (resp. $\mathcal B'$).
\\
\noindent 
${\rm (i)}\Rightarrow {\rm (ii)}$ 
By Definition
\ref{def:relations}
$B,B'$  are isomorphic. Let 
$\sigma:V\to V'$ denote an
isomorphism of Billiard Arrays from $B$ to $B'$.
Define a function
$\mathcal B'':\Delta_N \to V', 
\lambda \mapsto \sigma(\mathcal B_\lambda)$.
 By construction,
$\mathcal B''$ is a Concrete Billiard Array that is
isomorphic to $\mathcal B$ and corresponds to $B'$. 
Note that $\mathcal B',\mathcal B''$ 
are associates since they both correspond to $B'$.
\\
\noindent ${\rm (ii)}\Rightarrow {\rm (i)}$ 
Let $\mathcal B''$ denote  a Concrete Billiard Array
that is isomorphic to $\mathcal B$ and associated with $\mathcal B'$.
Since $\mathcal B',\mathcal B''$ are  associates,
$\mathcal B''$ must correspond to $B'$.
Let $\sigma$ denote an isomorphism of Concrete Billiard Arrays
from $\mathcal B$ to $\mathcal B''$. Then $\sigma$ is an
isomorphism of Billiard Arrays from
$B$ to $B'$. Therefore $B,B'$ are isomorphic so
$\mathcal B,
\mathcal B'$ are similar.
\\
${\rm (i)}\Leftrightarrow {\rm (iii)}$ 
Interchange the roles of $\mathcal B, \mathcal B'$
in the proof of
${\rm (i)}\Leftrightarrow {\rm (ii)}$.
\end{proof} 

\begin{lemma} 
\label{lem:assocIso}
Consider the map which takes a Concrete Billiard Array
to the corresponding Billiard Array. This map induces a bijection
between the following two sets:
\begin{enumerate}
\item[\rm (i)] the similarity classes of Concrete Billiard Arrays
over $\mathbb F$ that have diameter $N$;
\item[\rm (ii)] the isomorphism classes of Billiard Arrays over
$\mathbb F$ that have diameter $N$.
\end{enumerate}
\end{lemma}
\begin{proof} By Definition
\ref{def:relations} the given map induces an injective
function from set (i) to set (ii). The function is
surjective by the comments above
Definition
\ref{def:corr}. Therefore the function is a bijection.
\end{proof}

\section{Billiard Arrays and the poset $\Delta_{\leq N}$}

Throughout this section the following notation is
in effect. Fix $N \in \mathbb N$. Let $V$ denote
a vector space over $\mathbb F$ with dimension $N+1$.
Let $B$ denote a Billiard Array on $V$. 
\medskip

\noindent 
By definition,
$B$ is a function $\Delta_N \to \mathcal P_1(V)$. 
We now extend $B$ to a function 
$B: \Delta_{\leq N} \to \mathcal P(V)$.

\begin{definition}
\label{def:extB}
\rm 
For  $\mu \in \Delta_{\leq N}$ define
\begin{equation}
B_{\mu} = \sum_{
\genfrac{}{}{0pt}{}{\lambda \in \Delta_N}{\mu \leq \lambda}
}
B_{\lambda}.
\label{eq:Bmudef}
\end{equation}
\noindent Thus $B_{\mu} \in \mathcal P(V)$.
\end{definition}

\noindent Pick $\mu \in \Delta_{\leq N}$.
Let $n=N-{\rm rank}(\mu)$, so that
$\mu \in \Delta_{N-n}$.
Evaluating
(\ref{eq:Bmudef}) using
Lemma \ref{lem:upsetdesc}  we obtain
\begin{equation}
B_\mu = \sum_{\nu \in \Delta_n} B_{\mu+\nu}.
\label{eq:Bmudec}
\end{equation}
Our next goal is to show that the function
 $\Delta_n \to \mathcal P_1(B_\mu)$, $\nu \mapsto 
B_{\mu+\nu}$ is a Billiard Array on $B_\mu$. To this
end, we show that $B_{\mu}$ has dimension $n+1$.

\begin{lemma}
\label{lem:decompBmu2}
With the above notation,
pick $\xi\in \lbrace 1,2,3\rbrace$ and
let $K$ denote the $\xi$-boundary of $\Delta_n$.
Then the sum
\begin{equation}
B_\mu = \sum_{\nu \in K} B_{\mu+\nu}
\label{eq:Bmudec2}
\end{equation}
is direct. Moreover $B_{\mu}$ has dimension $n+1$.
\end{lemma}
\begin{proof}
Define $B'_\mu
 = \sum_{\nu \in K} B_{\mu+\nu}$.
The preceding sum 
 is direct,
by Definition
\ref{def:ba}(i) and since
$\lbrace \mu + \nu\rbrace_{\nu \in K}$
are collinear locations in $\Delta_N$.
We show $B_{\mu} = B'_\mu$.
We have
$B'_\mu \subseteq B_\mu$ by
(\ref{eq:Bmudec}).
To obtain
$B_\mu \subseteq B'_\mu$, for
 $\nu \in \Delta_n$ 
use Corollary
\ref{cor:usefuldep} and induction on the $\xi$-coordinate
of $\nu$ to obtain
$B_{\mu + \nu}\subseteq  
 B'_{\mu}$.
This  and
(\ref{eq:Bmudec})
yield
$B_\mu \subseteq B'_\mu$, and
therefore
$B_\mu = B'_\mu$. We have shown that the sum
(\ref{eq:Bmudec2})
is direct. The set
$K$ has cardinality $n+1$, so
$B_\mu$ has dimension $n+1$.
\end{proof}

\noindent  We emphasize one aspect of
Lemma \ref{lem:decompBmu2}.

\begin{lemma} 
\label{lem:Bmudim}
For $\mu \in \Delta_{\leq N}$ we have
$B_{\mu} \in \mathcal P_{n+1}(V)$, where $n=
N-{\rm rank}(\mu)$.
\end{lemma}

\begin{proposition}
\label{prop:newBA}
For $0 \leq n \leq N$ and $\mu \in \Delta_{N-n}$,
the function $\Delta_n \to \mathcal P_1(B_\mu)$, $\nu \mapsto 
B_{\mu+\nu}$ is a Billiard Array on $B_\mu$.
\end{proposition}
\begin{proof} We verify that
the given function satisfies the conditions
of 
Definition \ref{def:ba}.
Let $L$ denote a line in $\Delta_n$.
Then the elements $\lbrace \mu+\nu\rbrace_{\nu \in L}$
are collinear locations
in $\Delta_N$.
Consequently the given function satisfies
Definition \ref{def:ba}(i).
Let $C$ denote 
a black 3-clique in $\Delta_n$. Then 
the locations $\lbrace \mu+\nu\rbrace_{\nu \in C}$ 
 form a black 3-clique in $\Delta_N$.
Therefore the given function satisfies
Definition \ref{def:ba}(ii).
The result follows.
\end{proof}

\noindent As we work with the elements $B_\mu$ from
Definition
\ref{def:extB},
we often encounter the case in which
$\mu$ is a corner of $\Delta_{\leq N}$. We now
consider this case.

\begin{lemma}
\label{lem:3Fclar}
Pick $\eta \in \lbrace 1,2,3\rbrace$ and
$0 \leq n \leq N$.
Let $\mu$ denote the $\eta$-corner of $\Delta_{N-n}$.
Then $B_\mu$ is described as follows.
\begin{enumerate}
\item[\rm (i)]
We have $B_\mu=\sum_{\lambda} B_\lambda$,
where 
the sum is over the locations $\lambda $ in $\Delta_N$
 at distance $\leq n$
from the $\eta$-corner of $\Delta_N$.
\item[\rm (ii)]
The sum $B_\mu=\sum_{\lambda} B_\lambda$
is direct, where 
the sum is over the locations $\lambda $ in $\Delta_N$
 at distance $n$
from the $\eta$-corner of $\Delta_N$.
\item[\rm (iii)] Pick 
$\xi \in \lbrace 1,2,3\rbrace$ other than $\eta$.
Then the sum $B_\mu=\sum_{\lambda} B_\lambda$
is direct, where the sum is over
the locations $\lambda$ 
on the $\xi$-boundary of $\Delta_N$ at distance  $\leq n$
from the $\eta$-corner of $\Delta_N$.
\end{enumerate}
\end{lemma}
\begin{proof}(i) By
Lemma
\ref{lem:moredetail}(i),(iv) and
Definition \ref{def:extB}.
\\
\noindent (ii), (iii) In
Lemma \ref{lem:decompBmu2},
take $\mu$ to be the $\eta$-corner of
$\Delta_{N-n}$.
\end{proof}

\noindent At the end of Section 5 we discussed
the flags in the poset $\Delta_{\leq N}$.
In Section 6 we discussed the flags in the poset
${\mathcal P}(V)$. We now consider how the flags in these
two posets are related. 

\begin{lemma}
\label{lem:reverse}
Let $\mu,\nu$ denote elements in $\Delta_{\leq N}$ such that
$\mu\leq \nu$. Then $B_{\nu}\subseteq B_{\mu}$.
\end{lemma}
\begin{proof}
Use
Definition \ref{def:extB}.
\end{proof}

\begin{lemma}
\label{lem:ind}
Let $\lbrace \lambda_n\rbrace_{n=0}^N$
denote a flag in $\Delta_{\leq N}$. Define
$U_i = B_{\lambda_{N-i}}$ for $0 \leq i \leq N$.
Then the sequence
$\lbrace U_i\rbrace_{i=0}^N$ is a flag on 
$V$.
\end{lemma}
\begin{proof}
By Lemma
\ref{lem:reverse}
$U_{i-1} \subseteq U_i$ for $1 \leq i \leq N$.
By Lemma
\ref{lem:Bmudim} the subspace $U_i$ has dimension $i+1$
for $0 \leq i \leq N$. The result follows.
\end{proof}

\begin{definition}
\label{def:inducedflag}
\rm
Referring to Lemma
\ref{lem:ind}, the flag
$\lbrace U_i\rbrace_{i=0}^N$ 
is called the  {\it $B$-flag induced by
 $\lbrace \lambda_n\rbrace_{n=0}^N$}.
 \end{definition}

\begin{definition}
\label{lem:threeflags}
\rm
Pick $\eta \in \lbrace 1,2,3\rbrace$.
Recall from
Definition
\ref{eq:flagex} the
 flag 
$\lbrack \eta \rbrack$ in 
$\Delta_{\leq N}$.
 The $B$-flag induced by
$\lbrack \eta \rbrack$ will be called the
{\it $B$-flag 
$\lbrack \eta \rbrack$}. This is a flag on $V$.
\end{definition}

\noindent The next result is meant to clarify Definition
\ref{lem:threeflags}.

\begin{lemma}
\label{lem:clar}
Pick $\eta \in \lbrace 1,2,3\rbrace$. Let
$\lbrace U_n \rbrace_{n=0}^N$ denote the
$B$-flag $\lbrack \eta \rbrack $.
Then for $0 \leq n \leq N$, $U_n$ is equal to
the space $B_{\mu}$ from
Lemma
\ref{lem:3Fclar}.
\end{lemma}
\begin{proof}
By Lemma
\ref{lem:ind} and Definitions
\ref{def:inducedflag},
\ref{lem:threeflags}.
\end{proof}

\begin{lemma} 
\label{lem:flagext}
Pick $\eta \in \lbrace 1,2,3\rbrace$. Let
$\lbrace U_n \rbrace_{n=0}^N$ denote the
$B$-flag $\lbrack \eta \rbrack$.
Then for $1 \leq n \leq N$ the sum
$U_n = U_{n-1} + B_\lambda$ is direct, where $\lambda$ is
any location in $\Delta_N$ at distance $n$ from the $\eta$-corner.
\end{lemma}
\begin{proof}
We first show that $U_{n-1}+B_\lambda$ is independent
of the choice of $\lambda$. Let $L$ denote
the set of locations in $\Delta_N$ that are at distance $n$
from the $\eta$-corner. By Lemma
\ref{lem:linedist},
$L$ is the $\eta $-line
with cardinality $n+1$. 
By construction $\lambda \in L$.
Let $\mu \in L$ be adjacent
to $\lambda$.
By Lemma
\ref{lem:edgegivesclique} there exists a unique
$\nu \in \Delta_N$ such that $\lambda, \mu,\nu$ form
a black 3-clique.
The location $\nu$ is at distance $n-1$
from the $\eta$-corner of $\Delta_N$, so $B_{\nu} \subseteq U_{n-1}$
in view of Lemmas
\ref{lem:3Fclar}(i),
\ref{lem:clar}.
By
Lemma
\ref{lem:threedim2} we have $B_\lambda + B_\nu= B_\mu+B_\nu$.
By these comments
 $U_{n-1}+B_\lambda = U_{n-1}+B_\mu$. 
It follows that 
 $U_{n-1}+ B_\lambda $  is independent of the choice of $\lambda$.
So we may assume that $\lambda$ is on the
boundary of $\Delta_N$. Now there exists  
$\xi \in \lbrace 1,2,3\rbrace$ such that $\xi \not=\eta$
and $\lambda$ is on the $\xi$-boundary of
$\Delta_N$. By
Lemmas
\ref{lem:3Fclar}(iii),
\ref{lem:clar}
we have a direct sum
$U_{n-1} = \sum_{\zeta} B_\zeta$, where the sum
is over the locations $\zeta$ on the 
$\xi$-boundary of $\Delta_N$
 at distance
$\leq n-1$ from the $\eta$-corner. 
Similarly we have a direct sum
$U_{n} = \sum_{\zeta} B_\zeta$, where the sum
is over the locations $\zeta$ on the $\xi$-boundary of $\Delta_N$
 at distance
$\leq n$ from the $\eta $-corner. 
By assumption $\lambda$ is on the
$\xi$-boundary of $\Delta_N$ at distance $n$
from the $\eta$-corner.
By these comments the sum
$U_n = U_{n-1} +B_\lambda$ is direct.
\end{proof}

\begin{proposition}
\label{lem:dist}
Pick $\eta \in \lbrace 1,2,3\rbrace$.
Let $S$ denote a  subset of
$\Delta_N$ that is 
$\eta$-geodesic
in the sense of Definition
\ref{lem:distcoord2}.
Then the sum $\sum_{\lambda \in S} B_\lambda$
is direct.
\end{proposition}
\begin{proof}
We assume that the assertion is false,
and get a contradiction. By assumption
there exists
a counterexample $S$. Without loss,
we may assume that among all the counterexamples,
the cardinality of $S$  is minimal. This
cardinality is at least 2.
Let $\mu$ denote the location in $S$
that has minimal $\eta$-coordinate, 
and abbreviate $R=S\backslash \lbrace \mu \rbrace $.
By construction the sum
$ \sum_{\lambda \in S} B_\lambda$ is not direct, 
and the sum
$ \sum_{\lambda \in R} B_\lambda$ is direct.
By Definition 
\ref{def:ba}
$B_{\mu}$ has dimension 1.
By these comments
$B_\mu \subseteq \sum_{\lambda \in R} B_\lambda$.
Denote the $\eta$-coordinate of $\mu$ by $N-n$, and observe $n\geq 1$.
By construction $\mu$ is at distance $n$ from the $\eta$-corner of
$\Delta_N$, and  each element of $R$ is
at distance $\leq n-1$ from the 
$\eta$-corner
of $\Delta_N$.
Let $\lbrace U_i\rbrace_{i=0}^N$ denote the
$B$-flag $\lbrack \eta \rbrack$. 
By Lemma
\ref{lem:flagext} 
 the sum
$U_n = U_{n-1} + B_{\mu}$ is direct.
Using Lemmas
\ref{lem:3Fclar}(i),
\ref{lem:clar} we obtain
$B_\lambda \subseteq U_{n-1}$ for all
$\lambda \in R$.
Therefore $B_\mu \subseteq U_{n-1}$,
 for a contradiction. The result follows.
\end{proof}



\begin{corollary}
The function $B$ is injective.
\end{corollary}
\begin{proof}
Let $\lambda, \mu$ denote locations in $\Delta_N$
such that $B_\lambda = B_\mu$.
By Proposition
\ref{lem:dist} the locations $\lambda, \mu$ 
have the same $\eta$-coordinate for all 
$\eta \in \lbrace 1,2,3\rbrace$. Therefore
$\lambda = \mu$.
\end{proof}

\noindent The following definition is for notational
convenience.
\begin{definition} 
\label{def:outofbounds}
Let $B$ denote a Billiard Array on $V$.
Define $B_\lambda = 0$ for all $\lambda \in \mathbb R^3$ such that
$\lambda \notin \Delta_{\leq N}$.
\end{definition}

\section{Bases, decompositions, and flags} 

\noindent 
Throughout this section the following notation is
in effect.
Fix $N \in \mathbb N$. Let $V$ denote a vector space over
$\mathbb F$ with dimension $N+1$. 
Let $\mathcal B$ denote a Concrete Billiard
Array on $V$, and let $B$ denote the corresponding Billiard Array
on $V$.

\begin{lemma} 
\label{lem:geodesicpath}
Let $\lbrace \lambda_i\rbrace_{i=0}^n$
denote a geodesic path in $\Delta_N$, and define
$\mu=\lambda_0 \wedge\lambda_n$.  Then the following
{\rm (i)--(iii)} hold:
\begin{enumerate}
\item[\rm (i)] $\mu \leq \lambda_i$ for $0 \leq i \leq n$;
\item[\rm (ii)]
 the sequence
$\lbrace B_{\lambda_i}\rbrace_{i=0}^n$ is a decomposition
of $B_\mu$;
\item[\rm (iii)] 
the sequence $\lbrace \mathcal B_{\lambda_i}\rbrace_{i=0}^n$
is a basis for $B_\mu$.
\end{enumerate}
\end{lemma}
\begin{proof}
(i) 
The set
$\lbrace \lambda \in \Delta_N\;|\;\mu\leq \lambda\rbrace$
is geodesically closed by 
Lemma \ref{lem:gcsubset}.
By construction $\mu \leq \lambda_0$ and
$\mu\leq \lambda_n$.
Now by 
 Definition
\ref{def:GC},
$\mu \leq \lambda_i $ for $0 \leq i \leq n$.
\\
\noindent (ii)
By Definition
\ref{def:extB} and
(i) above,
 $B_{\lambda_i} \subseteq B_u$
for $0 \leq i \leq n$.
Observe that $n=\partial(\lambda_0, \lambda_n)$
by Definition
\ref{def:Geo},
so $\mu \in \Delta_{N-n}$ by Lemma
\ref{lem:findrank}. Now $B_\mu$
has dimension $n+1$ by
Lemma
\ref{lem:Bmudim}.
It remains to show that the sum
$\sum_{i=0}^n B_{\lambda_i}$ is direct.
This sum is direct by 
 Lemma
\ref{lem:etaGeo} and
Proposition
\ref{lem:dist}.
\\
\noindent (iii) By (ii) and since
$\mathcal B_\lambda$ spans $B_\lambda$ for all
$\lambda \in \Delta_N$.
\end{proof}

\begin{definition}
\label{def:basisInduced}
\rm
Referring to Lemma
\ref{lem:geodesicpath},
the decomposition
 $\lbrace  B_{\lambda_i}\rbrace_{i=0}^n$
(resp.
basis
$\lbrace \mathcal B_{\lambda_i}\rbrace_{i=0}^n$)
is said to be {\it $B$-induced}
(resp. {\it $\mathcal B$-induced})
by the
path
 $\lbrace \lambda_i\rbrace_{i=0}^n$.
\end{definition}

\begin{definition}
\label{def:Bdec}
\rm 
For distinct $\eta, \xi \in\lbrace 1,2,3\rbrace$
we define
 a decomposition of $V$ called the
{\it $B$-decomposition $\lbrack \eta, \xi\rbrack$}.
This decomposition is $B$-induced by the
path
$\lbrack \eta, \xi\rbrack$ in $\Delta_N$  from
Definition
\ref{def:sixpaths}.
Similarly, we define
a basis of $V$ called the
{\it $\mathcal B$-basis  $\lbrack \eta, \xi\rbrack$}.
This basis is $\mathcal B$-induced by the path
$\lbrack \eta, \xi\rbrack$ in $\Delta_N$. 
By construction, the $\mathcal B$-basis 
$\lbrack \eta, \xi\rbrack$ of $V$ induces the
$B$-decomposition
$\lbrack \eta, \xi\rbrack$ of $V$, in the sense of the first paragraph
of Section 6.
\end{definition}

\begin{lemma} 
\label{lem:BdecClar}
For distinct $\eta, \xi \in \lbrace 1,2,3\rbrace$
consider the 
$B$-decomposition
(resp. $\mathcal B$-basis) $\lbrack \eta, \xi\rbrack$ of
$V$. For $0 \leq i \leq N$ the
$i$-component of
this decomposition (resp. basis) is included in $B$
(resp. $\mathcal B$)
at the location $\lambda_i$ given in Lemma
\ref{lem:sixpathsinfo}.
\end{lemma}
\begin{proof}
By Lemma \ref{lem:sixpathsinfo} and
Definitions
\ref{def:basisInduced},
\ref{def:Bdec}.
\end{proof}

\begin{lemma}
\label{lem:BinvB}
For distinct
$\eta, \xi \in \lbrace 1,2,3\rbrace$ the
$B$-decomposition (resp. $\mathcal B$-basis) $\lbrack \eta, \xi\rbrack$ is
the inversion of the $B$-decomposition
(resp. $\mathcal B$-basis)
$\lbrack \xi, \eta \rbrack$.
\end{lemma}
\begin{proof} By Lemma
\ref{lem:boundarypathInv} or Lemma
\ref{lem:BdecClar}.
\end{proof}

\begin{lemma}
\label{lem:decIndFlag}
 For distinct
$\eta, \xi \in\lbrace 1,2,3\rbrace$ 
the $B$-decomposition $\lbrack \eta, \xi\rbrack$ of $V$
induces the $B$-flag $\lbrack \eta \rbrack$ on $V$.
\end{lemma}
\begin{proof}
By Lemmas
\ref{lem:3Fclar}(iii), \ref{lem:clar}.
\end{proof}

\begin{lemma}
\label{lem:mutOp}
The $B$-flags 
$\lbrack 1\rbrack$,
$\lbrack 2\rbrack$,
$\lbrack 3\rbrack$ on $V$
are mutually opposite.
\end{lemma}
\begin{proof}
Pick distinct $\eta, \xi \in \lbrace 1,2,3\rbrace$.
By Lemma
\ref{lem:decIndFlag},
the
$B$-decomposition $\lbrack \eta, \xi \rbrack$
induces 
the $B$-flag $\lbrack \eta\rbrack$,
and the 
$B$-decomposition $\lbrack \xi, \eta \rbrack$
induces 
the $B$-flag 
$\lbrack \xi\rbrack$.
Now the $B$-flags 
$\lbrack \eta\rbrack$,
$\lbrack \xi\rbrack$ are opposite
in view of Lemma
\ref{lem:BinvB}.
\end{proof}

\begin{lemma}
For distinct $\eta,\xi \in \lbrace 1,2,3\rbrace$ and
$0 \leq i \leq N$, the $i$-component of
the $B$-decomposition $\lbrack \eta,\xi \rbrack$ is equal to 
the intersection 
of the following two subspaces of $V$:
\begin{enumerate}
\item[\rm (i)] component $i$ of the $B$-flag $\lbrack \eta \rbrack$;
\item[\rm (ii)] component $N-i$ of the $B$-flag $\lbrack \xi \rbrack$.
\end{enumerate}
\end{lemma}
\begin{proof}
The
$B$-decomposition $\lbrack \eta, \xi \rbrack$
induces 
the $B$-flag $\lbrack \eta\rbrack$,
and the inversion of this decomposition
induces 
the $B$-flag $\lbrack \xi\rbrack$.
\end{proof}

\section{More on flags}

\noindent Throughout this section the following notation is in effect.
Fix $N \in \mathbb N$.
Let $V$ denote a vector space over $\mathbb F$
with dimension $N+1$. Let $B$ denote a Billiard Array on $V$.

\begin{proposition} 
\label{prop:3flagsgiveB}
Pick integers
$r,s,t$ $(0 \leq r,s,t\leq N)$
and consider the intersection of the following three sets:
\begin{enumerate}
\item[\rm (i)] component $N-r$ of the $B$-flag $\lbrack 1 \rbrack$;
\item[\rm (ii)] component $N-s$ of the $B$-flag $\lbrack 2 \rbrack$;
\item[\rm (iii)] component $N-t$ of the $B$-flag $\lbrack 3 \rbrack$.
\end{enumerate}
If $r+s+t>N$ then the intersection is zero. If $r+s+t\leq N$
then the intersection is $B_{\mu}$, where $\mu=(r,s,t)$.
\end{proposition}
\begin{proof}
Denote the sets in (i), (ii), (iii) by $S_1,S_2,S_3$ respectively.
Let $L$ denote the set of locations in $\Delta_N$ 
 at distance $N-s$ from the $2$-corner of $\Delta_N$.
Thus $L$ is the $2$-line with cardinality $N-s+1$.
The set $L$ consists of the locations
$\lbrace \lambda_i \rbrace_{i=0}^{N-s}$ where
$\lambda_i = (N-s-i, s, i)$ for $0 \leq i \leq N-s$.
By Lemmas
\ref{lem:3Fclar}(ii),
\ref{lem:clar}
the sequence
$\lbrace B_{\lambda_i}\rbrace_{i=0}^{N-s}$
is a decomposition of $S_2$.
We now describe $S_1\cap S_2$. 
By Lemma
\ref{lem:mutOp}
 the $B$-flags
$\lbrack 1 \rbrack, 
\lbrack 2 \rbrack$ are opposite.
For the moment assume $r+s>N$. Then 
$S_1\cap S_2=0$, so 
$S_1\cap S_2\cap S_3=0$ and we are done.
Next assume 
 $r+s\leq N$. Then $S_1\cap S_2$ has dimension $ N-r-s+1$ by
 Lemma
\ref{lem:OppDim}.
For $0 \leq i \leq N-s$ the location
$\lambda_i$ is at distance $s+i$ from
the $1$-corner of $\Delta_N$. 
For $0 \leq i \leq N-r-s$ this
 distance is at most $ N-r$,
 so 
 $B_{\lambda_i} \subseteq S_1$ by 
 Lemmas
\ref{lem:3Fclar}(i),
\ref{lem:clar}, so
 $B_{\lambda_i} \subseteq S_1\cap S_2$.
Therefore $\sum_{i=0}^{N-r-s} 
 B_{\lambda_i} \subseteq S_1\cap S_2$.
In this inclusion each side has dimension $N-r-s+1$,
so equality holds. In other words
\begin{equation}
\label{eq:int2}
S_1\cap S_2 = 
\sum_{i=0}^{N-r-s} 
 B_{\lambda_i}.
\end{equation}
We now describe $S_2\cap S_3$. 
By Lemma
\ref{lem:mutOp} the
$B$-flags
$\lbrack 2 \rbrack, 
\lbrack 3 \rbrack$ are opposite.
For the moment assume $s+t>N$. Then 
$S_2\cap S_3=0$, so 
$S_1\cap S_2\cap S_3=0$ and we are done.
Next assume 
 $s+t\leq N$. Then $S_2\cap S_3$ has dimension $ N-s-t+1$
by Lemma
\ref{lem:OppDim}.
For $0 \leq i \leq N-s$ the location
$\lambda_i$ is at distance $N-i$ from
the $3$-corner of $\Delta_N$. 
For $t \leq i \leq N-s$ this
 distance is at most $ N-t$, so
 $B_{\lambda_i} \subseteq S_3$ by 
 Lemmas
\ref{lem:3Fclar}(i), 
\ref{lem:clar}, so
 $B_{\lambda_i} \subseteq S_2\cap S_3$.
Therefore
$\sum_{i=t}^{N-s} 
 B_{\lambda_i} \subseteq S_2\cap S_3$.
In this inclusion each side has dimension $N-s-t+1$,
so equality holds. In other words
\begin{equation}
\label{eq:int3}
S_2\cap S_3 = 
\sum_{i=t}^{N-s} 
 B_{\lambda_i}.
\end{equation}
Combining 
(\ref{eq:int2}),
(\ref{eq:int3}) we obtain
\begin{equation}
S_1\cap S_2\cap S_3 = 
\sum_{i=t}^{N-r-s} 
 B_{\lambda_i}.
\label{eq:finalint}
\end{equation}
Consider the sum on the right in
(\ref{eq:finalint}).
First assume $r+s+t>N$. Then the sum
 is empty, so
$S_1\cap S_2\cap S_3 =0$. Next assume
$r+s+t\leq N$, and define $n=N-r-s-t$. Then
$0 \leq n \leq N$ and 
$\mu=(r,s,t)$  is contained in
$\Delta_{N-n}$.
Let $K$ denote the $2$-boundary of $\Delta_n$,
and consider the set $K+\mu$. This set
consists of the locations
$\lbrace \lambda_i \rbrace_{i=t}^{N-r-s}$.
Now by Lemma
\ref{lem:decompBmu2} we obtain
$B_\mu=\sum_{i=t}^{N-r-s} B_{\lambda_i}$. By this and 
(\ref{eq:finalint}) we 
obtain $S_1\cap S_2\cap S_3 = B_\mu$.
\end{proof}

\begin{corollary}
\label{cor:3flagsgiveB}
Pick a location $\lambda =(r,s,t)$ in $\Delta_N$.
Then $B_\lambda$ is equal to the intersection of the following three sets:
\begin{enumerate}
\item[\rm (i)] component $N-r$ of the $B$-flag $\lbrack 1 \rbrack$;
\item[\rm (ii)] component $N-s$ of the $B$-flag $\lbrack 2 \rbrack$;
\item[\rm (iii)] component $N-t$ of the $B$-flag $\lbrack 3 \rbrack$.
\end{enumerate}
\end{corollary}
\begin{proof} 
In Proposition
\ref{prop:3flagsgiveB}, assume $r+s+t=N$.
\end{proof}

\noindent By Corollary 
\ref{cor:3flagsgiveB} the Billiard Array $B$ is uniquely
determined by the $B$-flags
$\lbrack 1\rbrack,
\lbrack 2\rbrack,
\lbrack 3\rbrack$. This idea will be pursued further in Section 12.
In the meantime we have a few comments.

\begin{proposition}
\label{prop:meet}
Let $\mu, \nu$ denote elements in
$\Delta_{\leq N}$.
\begin{enumerate}
\item[\rm (i)]
Assume that $\mu\vee \nu$ exists in $\Delta_{\leq N}$.
Then $B_\mu \cap B_\nu = B_{\mu\vee \nu}$.
\item[\rm (ii)]
Assume that $\mu \vee\nu$ does not exist in $\Delta_{\leq N}$.
Then $B_\mu \cap B_\nu =0$.
\end{enumerate}
\end{proposition}
\begin{proof} This is a routine application
of Proposition
\ref{prop:3flagsgiveB}.
\end{proof}

\begin{proposition} Let $\mu, \nu$ denote elements in
$\Delta_{\leq N}$. Then $\mu\leq \nu$ if and only if
$B_{\nu}\subseteq B_{\mu}$.
\end{proposition}
\begin{proof}
First assume that $\mu \leq \nu$. Then
$B_{\nu}\subseteq B_{\mu}$ by Lemma
\ref{lem:reverse}.  Next assume that 
$B_{\nu}\subseteq B_{\mu}$.
Then $B_{\nu} = B_{\mu} \cap B_{\nu}$, so
$B_{\mu} \cap B_{\nu}\not=0$. 
Now by
Proposition
\ref{prop:meet} we see that $\mu\vee \nu$ exists in $\Delta_{\leq N}$,
and 
 $ B_{\mu}\cap B_{\nu} =
 B_{\mu\vee \nu}$. Therefore
 $B_{\nu}=B_{\mu\vee \nu}$.
By this and
Lemma
\ref{lem:Bmudim},
$\nu$ and $\mu\vee \nu$ have the same rank.
By construction $\nu \leq \mu\vee \nu$.
By these comments
 $\nu=\mu\vee \nu$, so
$\mu\leq \nu$.
\end{proof}

\section{A characterization of Billiard Arrays using totally opposite flags}

\noindent Throughout this section the following notation is in effect.
Fix $N \in \mathbb N$. Let $V$ denote a vector space over
$\mathbb F$ with dimension $N+1$.
\medskip

\noindent In Lemma
\ref{lem:mutOp}
we encountered three flags on $V$ that are mutually opposite.
 We now
introduce a stronger condition involving three flags.

\begin{definition}
\label{def:Topp}
\rm
Suppose we are given three flags on $V$, 
denoted 
$\lbrace U_i\rbrace_{i=0}^N$,
$\lbrace U'_i\rbrace_{i=0}^N$,
$\lbrace U''_i\rbrace_{i=0}^N$.
These flags are said to be {\it totally opposite}
whenever $U_{N-r} \cap U'_{N-s} \cap U''_{N-t} =0 $ 
for all $r,s,t$ $(0 \leq r,s,t\leq N)$ such that
$r+s+t>N$.
\end{definition}

\begin{lemma} 
\label{lem:ToppisMutopp}
Referring to
Definition \ref{def:Topp}, assume that 
$\lbrace U_i\rbrace_{i=0}^N$,
$\lbrace U'_i\rbrace_{i=0}^N$,
$\lbrace U''_i\rbrace_{i=0}^N$ are totally opposite.
Then they are mutually opposite.
\end{lemma}
\begin{proof}  Setting $r=0$ (resp. $s=0$) (resp. $t=0$) in
 Definition 
\ref{def:Topp}, we find that the flags
$\lbrace U'_i\rbrace_{i=0}^N$ and
$\lbrace U''_i\rbrace_{i=0}^N$
(resp. $\lbrace U''_i\rbrace_{i=0}^N$ and
$\lbrace U_i\rbrace_{i=0}^N$)
(resp. $\lbrace U_i\rbrace_{i=0}^N$ and
$\lbrace U'_i\rbrace_{i=0}^N$) are opposite.
\end{proof}

\begin{lemma}
\label{lem:ToppMeaning}
Referring to Definition 
\ref{def:Topp}, the following are equivalent:
\begin{enumerate}
\item[\rm (i)] the flags
$\lbrace U_i\rbrace_{i=0}^N$,
$\lbrace U'_i\rbrace_{i=0}^N$,
$\lbrace U''_i\rbrace_{i=0}^N$ are totally opposite;
\item[\rm (ii)] for $0 \leq n \leq N$ the sequences
\begin{equation}
\label{eq:3flags}
\lbrace U_{i}\rbrace_{i=0}^{N-n}, \qquad  
\lbrace U_{N-n} \cap U'_{n+i}\rbrace_{i=0}^{N-n},
\qquad 
\lbrace U_{N-n} \cap U''_{n+i}\rbrace_{i=0}^{N-n}
\end{equation}
are mutually opposite flags on $U_{N-n}$;
\item[\rm (iii)] for $0 \leq n \leq N$ the sequences
\begin{equation*}
\lbrace U'_{i}\rbrace_{i=0}^{N-n},\qquad
\lbrace U'_{N-n} \cap U''_{n+i}\rbrace_{i=0}^{N-n},\qquad
\lbrace U'_{N-n} \cap U_{n+i}\rbrace_{i=0}^{N-n}
\end{equation*}
are mutually opposite flags on $U'_{N-n}$;
\item[\rm (iv)] for $0 \leq n \leq N$ the sequences
\begin{equation*}
\lbrace U''_{i}\rbrace_{i=0}^{N-n},\qquad
\lbrace U''_{N-n} \cap U_{n+i}\rbrace_{i=0}^{N-n},\qquad
\lbrace U''_{N-n} \cap U'_{n+i}\rbrace_{i=0}^{N-n} 
\end{equation*}
are mutually opposite flags on $U''_{N-n}$.
\end{enumerate}
\end{lemma}
\begin{proof}  ${\rm (i)}\Rightarrow {\rm (ii)}$ 
We first show that each sequence in
(\ref{eq:3flags}) is a flag on
$U_{N-n}$.
For the first sequence this is clear, so consider
the second sequence.
By construction, the subspaces
$\lbrace U_{N-n} \cap U'_{n+i}\rbrace_{i=0}^{N-n}$ 
are nested and contained in $U_{N-n}$.
For $0 \leq i \leq N-n$ the subspace
$ U_{N-n} \cap U'_{n+i}$ 
has dimension
$i+1$ by Lemmas
\ref{lem:OppDim},
\ref{lem:ToppisMutopp}.
 Therefore the sequence 
$\lbrace U_{N-n} \cap U'_{n+i}\rbrace_{i=0}^{N-n}$ 
is a flag on $U_{N-n}$.
Similarly the sequence
$\lbrace U_{N-n} \cap U''_{n+i}\rbrace_{i=0}^{N-n}$ 
is a flag on $U_{N-n}$.
We have shown that each sequence in
(\ref{eq:3flags}) is a flag on
$U_{N-n}$. We check that these
three flags are mutually opposite.
We show that the flags
$\lbrace U_i \rbrace_{i=0}^{N-n}$ and
$\lbrace U_{N-n} \cap U'_{n+i}\rbrace_{i=0}^{N-n}$ 
are opposite.
For $0 \leq i,j\leq N-n$ such that $i+j<N-n$,
the intersection of
$ U_{i}$
and 
$U_{N-n} \cap U'_{n+j}$ 
is equal to $U_i \cap U'_{n+j}$,
which is zero by Lemma
\ref{lem:ToppisMutopp}.
Therefore the flags
$\lbrace U_i \rbrace_{i=0}^{N-n}$ and
$\lbrace U_{N-n} \cap U'_{n+i}\rbrace_{i=0}^{N-n}$ 
are opposite.
Similarly the flags
$\lbrace U_i \rbrace_{i=0}^{N-n}$ and
$\lbrace U_{N-n} \cap U''_{n+i}\rbrace_{i=0}^{N-n}$ 
are opposite.
We now show that the flags
$\lbrace U_{N-n} \cap U'_{n+i}\rbrace_{i=0}^{N-n}$ 
and
$\lbrace U_{N-n} \cap U''_{n+i}\rbrace_{i=0}^{N-n}$ 
are opposite.
For $0 \leq i,j\leq N-n$ such that $i+j<N-n$,
the intersection of
$ U_{N-n} \cap U'_{n+i}$ 
and 
$U_{N-n} \cap U''_{n+j}$ 
is equal to $U_{N-n} \cap U'_{n+i} \cap U''_{n+j}$,
which is zero by Definition
\ref{def:Topp}.
Therefore the flags
$\lbrace U_{N-n} \cap U'_{n+i}\rbrace_{i=0}^{N-n}$ 
and
$\lbrace U_{N-n} \cap U''_{n+i}\rbrace_{i=0}^{N-n}$ 
are opposite. By these comments the three flags in
(\ref{eq:3flags}) are mutually opposite.
\\
${\rm (ii)}\Rightarrow {\rm (i)}$ 
Setting $n=0$ in
(\ref{eq:3flags}) we see that 
the flags
$\lbrace U_i\rbrace_{i=0}^N$,
$\lbrace U'_i\rbrace_{i=0}^N$,
$\lbrace U''_i\rbrace_{i=0}^N$ are mutually opposite.
Pick integers $r,s,t$ $(0 \leq r,s,t\leq N)$  such that $r+s+t>N$.
We show that $U_{N-r}\cap U'_{N-s}\cap U''_{N-t}=0$.
We may assume that 
$r+s\leq N$
(resp. $s+t\leq N$) (resp. $t+r\leq N$);
otherwise
$U_{N-r}\cap U'_{N-s}=0$
(resp. $U'_{N-s}\cap U''_{N-t}=0$)
(resp. $U''_{N-t}\cap U_{N-r}=0$) and we are done.
Define 
 $i=N-r-s$ and  $j=N-r-t$. By construction
 $0 \leq i,j\leq N-r$ and $i+j<N-r$.
Observe that $U'_{N-s} = U'_{r+i}$
and 
 $U''_{N-t} = U''_{r+j}$.
Therefore $U_{N-r}\cap U'_{N-s}\cap U''_{N-t}$
is the intersection of
 $U_{N-r}\cap U'_{r+i} $ and $U_{N-r}\cap U''_{r+j}$,
which is zero by assumption.
\\
The implications ${\rm (i)}\Leftrightarrow {\rm (iii)}$  and
${\rm (i)}\Leftrightarrow {\rm (iv)}$ are similarly obtained.
\end{proof}

\begin{theorem} 
\label{thm:forward}
Let $B$ denote a Billiard Array on $V$, and recall
 the $B$-flags
$\lbrack 1 \rbrack$,
$\lbrack 2 \rbrack$,
$\lbrack 3 \rbrack$ on $V$
from Definition
\ref{lem:threeflags}. These flags are totally opposite.
\end{theorem}
\begin{proof} 
By Proposition
\ref{prop:3flagsgiveB}
and 
Definition
\ref{def:Topp}.
\end{proof}

\begin{lemma}
\label{lem:begin}
\rm
Suppose we are given three totally
opposite flags on $V$, denoted
$\lbrace U_i\rbrace_{i=0}^N$,
$\lbrace U'_i\rbrace_{i=0}^N$,
$\lbrace U''_i\rbrace_{i=0}^N$.
Then $U_{N-r} \cap U'_{N-s} \cap U''_{N-t} $ 
has dimension $N-r-s-t+1$ for all $(r,s,t) \in \Delta_{\leq N}$.
\end{lemma}
\begin{proof} The given flags
satisfy 
Lemma
\ref{lem:ToppMeaning}(i), so they satisfy
Lemma
\ref{lem:ToppMeaning}(ii).
Therefore
the sequences
$\lbrace U_{N-r} \cap U'_{r+i}\rbrace_{i=0}^{N-r}$
and
$\lbrace U_{N-r} \cap U''_{r+i}\rbrace_{i=0}^{N-r}$
are opposite flags on $U_{N-r}$.
Define $i=N-r-s$ and
$j=N-r-t$.
By construction $0 \leq i,j\leq N-r$ 
and $i+j\geq N-r$. 
Observe that
$U'_{N-s} =
U'_{r+i} $
and
$U''_{N-t}=U''_{r+j}$.
Therefore
$U_{N-r} \cap U'_{N-s} \cap U''_{N-t}$
is the intersection of
$U_{N-r} \cap U'_{r+i}$
and 
$U_{N-r} \cap U''_{r+j}$.
By 
Lemma \ref{lem:OppDim}
this intersection 
has dimension $i+j+r-N+1$, which is equal to
$N-r-s-t+1$.
\end{proof}

\begin{corollary}
\label{lem:begin2}
\rm
Suppose we are given three totally
opposite flags on $V$, denoted
$\lbrace U_i\rbrace_{i=0}^N$,
$\lbrace U'_i\rbrace_{i=0}^N$,
$\lbrace U''_i\rbrace_{i=0}^N$.
Then $U_{N-r} \cap U'_{N-s} \cap U''_{N-t} $ 
has dimension one for all $(r,s,t) \in \Delta_{N}$.
\end{corollary}
\begin{proof} 
In Lemma
\ref{lem:begin}, assume $r+s+t=N$.
\end{proof}

\begin{theorem}
\label{thm:backward}
Suppose we are given three totally opposite flags on
$V$, 
denoted 
$\lbrace U_i\rbrace_{i=0}^N$,
$\lbrace U'_i\rbrace_{i=0}^N$,
$\lbrace U''_i\rbrace_{i=0}^N$.
 For each location
$\lambda=(r,s,t)$ in $\Delta_N$ define
\begin{equation}
B_\lambda = 
U_{N-r} \cap U'_{N-s} \cap U''_{N-t}.
\label{eq:defBLam}
\end{equation}
Then the map $B:\Delta_N \to \mathcal P_1(V)$,
$\lambda \mapsto B_\lambda$ is a Billiard
Array on $V$.
\end{theorem}
\begin{proof}
The map $B$ is well defined by Corollary \ref{lem:begin2}.
We check that $B$ satisfies the conditions
of Definition
\ref{def:ba}. 
We show that $B$ satisfies 
Definition
\ref{def:ba}(i).
Let $L$ denote a line in $\Delta_N$.
We show that the sum
$\sum_{\lambda \in L}B_\lambda$ 
is direct.
For notational convenience,  assume that
$L$ is a $1$-line. Denote the cardinality
of $L$ by $N-n+1$. 
The sum $\sum_{\lambda \in L}B_\lambda$ 
is direct, since the summands make up the
decomposition of $U_{N-n}$
associated with the 
last two flags in
(\ref{eq:3flags}).
We have shown that $B$ satisfies
 Definition
\ref{def:ba}(i). 
Next we show that $B$ satisfies
 Definition
\ref{def:ba}(ii). Assume $N\geq 1$;
otherwise we are done.
Pick $(r,s,t) \in
\Delta_{N-1}$ and let $C$ denote the corresponding
black 3-clique in $\Delta_N$ from Lemma
\ref{lem:white}.
By construction, for all $\lambda \in C$
the subspace $B_\lambda$ is contained in
$U_{N-r} \cap U'_{N-s} \cap U''_{N-t}$. This 
space has dimension 2 by 
Lemma \ref{lem:begin}. Therefore
the 
sum $\sum_{\lambda \in C} B_\lambda$
is not direct. 
 We have shown that $B$ satisfies Definition
\ref{def:ba}(ii).
By the above comments
$B$ is
a Billiard Array on $V$. 
\end{proof}

\begin{proposition}
\label{lem:extra} 
Referring to Theorem
\ref{thm:backward}, the 
 $B$-flags 
$\lbrack 1 \rbrack$, $\lbrack 2 \rbrack$,
$\lbrack 3 \rbrack$ 
coincide with 
$\lbrace U_i\rbrace_{i=0}^N$,
$\lbrace U'_i\rbrace_{i=0}^N$,
$\lbrace U''_i\rbrace_{i=0}^N$ respectively.
\end{proposition}
\begin{proof}
We show that the 
$B$-flag $\lbrack 1 \rbrack$ coincides with
$\lbrace U_i\rbrace_{i=0}^N$.
Denote the $B$-flag $\lbrack 1 \rbrack$ by
$\lbrace {\mathcal U}_i\rbrace_{i=0}^N$.
We show that ${\mathcal U}_i = U_i$ for $0 \leq i \leq N$.
Let $i$ be given. By 
Lemmas
\ref{lem:3Fclar}(i),
\ref{lem:clar}
the subspace 
 ${\mathcal U}_i$ is spanned by the subspaces 
 $B_\lambda$ such that
$\lambda$ is a location in $\Delta_N$ at distance $\leq i$ from
the $1$-corner of $\Delta_N$.  Each such $B_\lambda$ is
contained in $U_i$ by
(\ref{eq:defBLam}),
 so $\mathcal U_i \subseteq U_i$.
Each of $\mathcal U_i, U_i$ has dimension $i+1$ by
the definition of a flag.
Therefore
$\mathcal U_i = U_i$.
We have shown that the $B$-flag $\lbrack 1 \rbrack$ coincides with
$\lbrace U_i\rbrace_{i=0}^N$. The remaining assertions are
similarly shown.
\end{proof}

\noindent Consider the following two sets:
\begin{enumerate}
\item[\rm (i)] the Billiard Arrays on $V$;
\item[\rm (ii)] the 3-tuples of totally opposite
flags on $V$.
\end{enumerate}
 Theorem
\ref{thm:forward} describes a function from (i) to (ii).
Theorem 
\ref{thm:backward} describes a function from (ii) to  (i).
\begin{theorem} The above functions are inverses. Moreover they
are bijections.
\end{theorem}
\begin{proof} The functions are inverses by
Corollary
\ref{cor:3flagsgiveB}
and
Proposition \ref{lem:extra}.
It follows that they are bijections.
\end{proof}

\section{The braces for a Billiard Array }

Our next general goal is to classify the Billiard Arrays up
to isomorphism. This goal will be completed in Secton 19.
Throughout the present section the following notation is in effect.
Fix $N \in \mathbb N$. Let $V$ denote a vector space over $\mathbb F$
with dimension $N+1$. Let $B$ denote a Billiard Array on $V$.

\begin{definition}
\label{def:brace}
\rm
Let $\lambda,\mu, \nu$ denote locations in $\Delta_N$ that form
a black 3-clique. By an {\it affine brace} (or {\it abrace}) for this clique,
we mean a set of vectors
\begin{equation}
u \in B_\lambda,
\qquad \quad
v \in B_\mu,
\qquad \quad
w \in B_\nu
\label{eq:location}
\end{equation}
that are not all zero, and
$u+v+w=0$.
\end{definition}

\begin{example}
\label{ex:white}
\rm 
Let $\lambda,\mu, \nu$ denote locations in $\Delta_N$ that form
a black 3-clique. Pick any nonzero vectors
\begin{equation*}
u \in B_\lambda,
\qquad  \quad 
v \in B_\mu,
\qquad  \quad 
w \in B_\nu.
\end{equation*}
The vectors $u,v,w$ are linearly dependent by Definition
\ref{def:ba}(ii).
So there
exist scalars $a,b,c$ in $\mathbb F$ that are not all zero
and $au+bv+cw=0$. The vectors $au, bv, cw$ form an abrace for the clique.
\end{example}

\noindent We have a few comments about abraces.

\begin{lemma}
\label{lem:brace}
Referring to Definition
\ref{def:brace}, assume that $u,v,w$ form an abrace. Then
each of $u,v,w$ is nonzero. Moreover 
any two of $u,v,w$ are linearly independent. 
\end{lemma}
\begin{proof} If at least one of $u,v,w$ is zero then
the other two are opposite and nonzero, contradicting the last assertion of
Lemma
\ref{lem:threedim2}. Therefore each of $u,v,w$ is nonzero.
By this and the last assertion of Lemma
\ref{lem:threedim2}, we find that
any two of $u,v,w$ are linearly independent.
\end{proof}

\begin{lemma}
\label{lem:abc}
Referring to Definition
\ref{def:brace}, assume that $u,v,w$ form an abrace. 
Let $a,b,c$ denote scalars in $\F$ such that
$au+bv+cw=0$. Then $a=b=c$.
\end{lemma}
\begin{proof} 
Using $u+v+w=0$ we obtain $(b-a)v+(c-a)w=0$. The vectors
$v,w$ are linearly independent so $b-a$ and $c-a$ are zero.
\end{proof}

\begin{lemma}
\label{lem:braceunique}
Referring to Definition
\ref{def:brace}, assume that $u,v,w$ form an abrace. 
 Then for
 $0 \not=\delta \in \mathbb F$ the vectors
$\delta u, \delta v, \delta w$ form an abrace. 
\end{lemma}
\begin{proof} By Definition
\ref{def:brace}.
\end{proof}

\begin{lemma}
\label{lem:braceunique2}
Referring to Definition
\ref{def:brace}, assume that $u,v,w$ form an abrace. 
Then for any abrace
\begin{equation*}
u' \in B_\lambda,
\qquad \quad
v' \in B_\mu,
\qquad \quad
w' \in B_\nu
\end{equation*}
 there exists  $0 \not=\delta \in \mathbb F$ such that
\begin{eqnarray*}
u' = \delta u, \qquad \quad
v' = \delta v, \qquad \quad
w' = \delta w.
\end{eqnarray*}
\end{lemma}
\begin{proof}
By 
(\ref{eq:location}) and since
$u,v,w$ are nonzero, we see that
$u,v,w$ are bases for
$B_\lambda$,
$B_\mu$,
$B_\nu$ respectively. Therefore there exist
scalars $a,b,c$ in $\mathbb F$ such that
\begin{eqnarray*}
u'=au, \qquad \quad 
v'=bv, \qquad \quad 
w'=cw. 
\end{eqnarray*}
Now $au+bv+cw=u'+v'+w'=0$ so $a=b=c$ by Lemma 
\ref{lem:abc}. Let $\delta$ denote this common value
and observe that $\delta$ satisfies the requirements of the
lemma.
\end{proof}

\begin{lemma}
\label{lem:brace2}
Let $\lambda,\mu, \nu$ denote locations in $\Delta_N$ that form
a black 3-clique. 
Then each nonzero $u \in B_\lambda$ is contained in a unique abrace
for this clique.
\end{lemma}
\begin{proof}
Concerning existence, pick nonzero vectors
$v \in B_\mu$ and
$w \in B_\nu$. By Example
\ref{ex:white} there exist scalars $a,b,c$ in $\mathbb F$
 such that $au,bv,cw$ is an abrace. The scalars
$a,b,c$ are nonzero by Lemma
\ref{lem:brace}. Now by construction the vectors $u, ba^{-1}v, ca^{-1}w$
form an abrace that contains $u$. The abrace containing
$u$ is unique by Lemma
\ref{lem:braceunique2}.
\end{proof}

\begin{definition}\rm Let $\lambda, \mu$ denote
 locations in $\Delta_N$ that form an edge. By Lemma
\ref{lem:edgegivesclique}
 there exists
a unique location $\nu \in \Delta_N$ such that
$\lambda, \mu, \nu$ form a black 3-clique. We call
$\nu$ the {\it completion} of the edge.
\end{definition}

\begin{definition}
\label{def:shortbrace}
\rm
 Let $\lambda, \mu$ denote
locations in $\Delta_N$ that form an edge.
By a {\it brace} for this edge, we mean a set of nonzero vectors
\begin{equation*}
u \in B_\lambda, \qquad \quad v \in B_\mu
\end{equation*}
such that $u+v \in B_\nu$. Here $\nu$ denotes the
completion of the edge.
\end{definition}

\begin{lemma}
\label{lem:abrace2brace}
Referring to Definition
\ref{def:brace}, assume that $u,v,w$ form an abrace. Then
$u,v$ form a brace for the edge  $\lambda, \mu$.
\end{lemma}
\begin{proof} The vectors $u,v$ are nonzero by
Lemma
\ref{lem:brace}. Moreover $u+v=-w \in B_{\nu}$.  
The result follows in view of Definition
\ref{def:shortbrace}.
\end{proof}

\begin{lemma}
\label{lem:brace2abrace}
Referring to Definition
\ref{def:shortbrace}, assume that $u,v$ form a brace,
and define $w=-u-v$. Then $u,v,w$ form an abrace for the
black 3-clique $\lambda, \mu, \nu$.
\end{lemma}
\begin{proof} By Definitions
\ref{def:brace} and
\ref{def:shortbrace}.
\end{proof}

\begin{lemma}
\label{lem:uniquebrace}
 Let $\lambda, \mu$ denote
locations in $\Delta_N$ that form an edge.
Each nonzero $u \in B_\lambda$ is contained in a unique
brace for this edge.
\end{lemma}
\begin{proof} 
By Lemmas
\ref{lem:abrace2brace},
\ref{lem:brace2abrace} together with
Lemma
\ref{lem:brace2}.
\end{proof}


\section{The maps $\tilde B_{\lambda,\mu}$ and the 
value function $\hat B$}

\noindent Throughout this section the following notation is in
effect. Fix $N \in \mathbb N$. Let $V$ denote a vector space over
$\mathbb F$ with dimension $N+1$. Let $B$ denote a Billiard
Array on $V$.

\begin{definition} 
\label{def:value}
\rm Let $\lambda, \mu$ denote
adjacent locations in $\Delta_N$.
We define an $\mathbb F$-linear map
$\tilde B_{\lambda,\mu}: B_\lambda \to V$
as follows. For each brace
\begin{equation*}
u \in B_\lambda, \qquad \quad 
v \in B_\mu
\end{equation*}
the map $\tilde B_{\lambda,\mu}$ sends $u\mapsto v$. The map
$\tilde B_{\lambda,\mu}$ is well defined by Lemma
\ref{lem:uniquebrace}.
\end{definition}

\begin{note}
\label{note:maps}
\rm
Let $\lambda, \mu$ denote
adjacent locations in $\Delta_N$.
By Definition
\ref{def:value},
 $B_\mu$
is the image of
$B_\lambda$ under $\tilde B_{\lambda,\mu}$.
Consequently we sometimes view $\tilde B_{\lambda,\mu}$
as a bijection 
$\tilde B_{\lambda,\mu}: B_\lambda \to B_\mu$.
\end{note}

\noindent The following definition is for notational convenience.
\begin{definition}\rm
Define $\tilde B_{\lambda,\mu}=0$ for all
$\lambda,\mu \in \mathbb R^3$ such that
$\lambda\notin \Delta_N$ or
$\mu\notin \Delta_N$.
\end{definition}

\begin{lemma} 
\label{lem:exvalue}
\rm Let $\lambda, \mu$ denote
adjacent locations in $\Delta_N$.
For vectors
$u \in B_\lambda $
and $v \in B_\mu$ the following are equivalent:
\begin{enumerate}
\item[\rm (i)] the vectors $u,v$ form a brace;
\item[\rm (ii)] 
the vector $u \not=0$ and the map $\tilde B_{\lambda,\mu}$ sends $u\mapsto v$;
\item[\rm (iii)] 
the vector $v \not=0$ and the map $\tilde B_{\mu,\lambda}$ sends $v\mapsto u$.
\end{enumerate}
\end{lemma}
\begin{proof} By Definition
\ref{def:value}.
\end{proof}

\begin{lemma}
\label{lem:Binverses}
Let $\lambda, \mu$ denote
adjacent locations in $\Delta_N$.
Then the maps
$\tilde B_{\lambda,\mu}:B_\lambda \to B_\mu$ and
$\tilde B_{\mu,\lambda}:B_\mu \to B_\lambda $
are inverses. 
\end{lemma}
\begin{proof} Compare parts (ii), (iii) of Lemma
\ref{lem:exvalue}.
\end{proof}

\begin{lemma} 
\label{lem:mapsforabrace}
Let $\lambda, \mu, \nu$ denote
locations in $\Delta_N$ that form a black 3-clique.
Pick an abrace 
\begin{equation*}
u \in B_\lambda,
\qquad \quad
v \in B_\mu,
\qquad \quad
w \in B_\nu.
\end{equation*}
Then
\bigskip

\centerline{
\begin{tabular}[t]{c|cc|cc|cc}
{\rm the map}  & 
  $\tilde B_{\lambda,\mu} $  
  & 
  $\tilde B_{\mu,\lambda} $  
  & 
  $\tilde B_{\mu,\nu} $  
  & 
  $\tilde B_{\nu,\mu} $  
  & 
  $\tilde B_{\nu,\lambda} $  
  & 
  $\tilde B_{\lambda,\nu} $  
  \\ \hline
{\rm sends}  & 
   $u\mapsto v$ &
   $v\mapsto u$ &
   $v\mapsto w$ &
   $w\mapsto v$ &
   $w\mapsto u$ &
   $u\mapsto w$ 
     \end{tabular}}
     \bigskip

\end{lemma}
\begin{proof}
By Lemma
\ref{lem:abrace2brace}
and
Lemma
\ref{lem:exvalue}.
\end{proof}

\begin{lemma}
\label{lem:threetermv}
Let $\lambda, \mu, \nu$ denote
locations in $\Delta_N$ that form a black 3-clique.
Then on 
 $B_\lambda$,
\begin{equation}
I + \tilde B_{\lambda,\mu}+ \tilde B_{\lambda, \nu} = 0.
\label{eq:sumit}
\end{equation}
\end{lemma}
\begin{proof}
Use Definition
\ref{def:brace}
and
 Lemma
\ref{lem:mapsforabrace}.
\end{proof}

\begin{lemma} 
\label{lem:whitevalue}
Let $\lambda, \mu, \nu$ denote
locations in $\Delta_N$ that form a black 3-clique.
Then the composition
\begin{equation*}
\begin{CD} 
B_\lambda  @>>  \tilde B_{\lambda,\mu} >  
 B_\mu  @>> \tilde B_{\mu,\nu} > B_\nu 
               @>>\tilde B_{\nu,\lambda} > B_\lambda
                  \end{CD}
\end{equation*}
is equal to the identity map on $B_\lambda$.
\end{lemma}
\begin{proof}
Use Lemma
\ref{lem:mapsforabrace}.
\end{proof}

\begin{definition} 
\label{def:orientedBval}
\rm
Let $\lambda, \mu, \nu$ denote locations in $\Delta_N$
that form a white 3-clique. Then the composition
\begin{equation*}
\begin{CD} 
B_\lambda  @>>  \tilde B_{\lambda,\mu} >  
 B_\mu  @>> \tilde B_{\mu,\nu} > B_\nu 
               @>>\tilde B_{\nu,\lambda} > B_\lambda
                  \end{CD}
\end{equation*}
is a nonzero scalar multiple of the identity map
on $B_\lambda$.
The scalar is called the {\it clockwise $B$-value} 
(resp. {\it counterclockwise $B$-value})
of the clique whenever
the sequence 
 $\lambda, \mu, \nu$ runs clockwise (resp. counterclockwise)
 around the clique.
\end{definition}

\begin{lemma} For each white 3-clique in $\Delta_N$, 
its clockwise $B$-value and counterclockwise $B$-value are reciprocal.
\end{lemma}
\begin{proof}
By Lemma
\ref{lem:Binverses}
and
Definition
\ref{def:orientedBval}.
\end{proof}

\begin{definition}
\label{def:BvalStraight}
\rm For each white 3-clique in $\Delta_N$,
by its {\it $B$-value}  we mean its clockwise $B$-value.
\end{definition}

\begin{definition}
\label{def:vfunction}
\rm By a {\it value function} on 
$\Delta_N$ we mean a function $\psi: \Delta_N \to
{\mathbb F}\backslash \lbrace 0 \rbrace$.
\end{definition}

\begin{definition}
\label{def:Bvaluefunction}
\rm Assume $N\geq 2$. We define a function
$\hat B:\Delta_{N-2} \to \mathbb F$ as follows.
Pick $(r,s,t) \in \Delta_{N-2}$. To describe
the image of $(r,s,t)$ under $\hat B$,
consider the corresponding white 3-clique in $\Delta_N$
from Lemma
\ref{lem:black}. The  $B$-value of this 3-clique
is the image of $(r,s,t)$ under
$\hat B$.
Observe that 
$\hat B$ is a value function on $\Delta_{N-2}$,
in the sense of
Definition \ref{def:vfunction}.
We call $\hat B$ the {\it value function} for $B$.
\end{definition}

\noindent We are going to show that for $N\geq 2$
the map $B\mapsto \hat B$ induces a bijection from
the set of isomorphism classes of Billiard Arrays
over $\mathbb F$ that have diameter $N$, to 
the set of value functions on $\Delta_{N-2}$.
The proof of this result will be completed in Section 19.

\section{The scalars
$ \tilde {\mathcal B}_{\lambda,\mu}$} 

Throughout this section the following notation is in effect.
Fix $N \in \mathbb N$. Let $V$ denote a vector space over $\mathbb F$
with dimension $N+1$.
Let $\mathcal B$ denote a Concrete Billiard Array on $V$.
Let
$B$ denote the corresponding Billiard Array on $V$, from
Definition
\ref{def:corr}.

\begin{definition}
\label{def:calT}
\rm
Let $\lambda, \mu$ denote adjacent locations in $\Delta_N$.
Recall the bijection
$\tilde B_{\lambda,\mu}: B_\lambda \to B_\mu$.
Recall that $\mathcal B_\lambda $ is a basis for $B_\lambda$
and 
 $\mathcal B_\mu $ is a basis for $B_\mu$.
Define a scalar
$\tilde {\mathcal B}_{\lambda,\mu} \in \mathbb F$ such that
$\tilde B_{\lambda,\mu}$ sends
	  $
	  \mathcal B_\lambda \mapsto  
      \tilde {\mathcal B}_{\lambda,\mu} \mathcal B_\mu
$.
Note that 
$\tilde {\mathcal B}_{\lambda,\mu} \not=0$.
\end{definition}

\begin{lemma}
\label{lem:kassociates}
For all $\lambda \in \Delta_N$ let $\kappa_\lambda$
denote a nonzero scalar in $\mathbb F$. Consider the
Concrete Billiard Array
$\mathcal B':\Delta_N \to V$,
$\lambda \mapsto \kappa_\lambda \mathcal B_\lambda$.
Then for all adjacent
$\lambda, \mu$ in $\Delta_N$,
\begin{equation*}
\tilde {\mathcal B}_{\lambda,\mu} \kappa_\lambda = 
\tilde {\mathcal B}'_{\lambda,\mu} \kappa_\mu.
\end{equation*}
\end{lemma}
\begin{proof}
By Definition 
\ref{def:calT} 
the map 
$\tilde B_{\lambda,\mu}$ sends
	  $
	  \mathcal B_\lambda \mapsto  
      \tilde {\mathcal B}_{\lambda,\mu} \mathcal B_\mu
$ and
	  $
	  {\mathcal B}'_\lambda \mapsto  
    \tilde {\mathcal B}'_{\lambda,\mu} {\mathcal B}'_\mu
$. Compare these using
${\mathcal B}'_\lambda=\kappa_\lambda \mathcal B_\lambda$
and
${\mathcal B}'_\mu=\kappa_\mu \mathcal B_\mu$.
\end{proof}

\noindent We mention a special case of Lemma
\ref{lem:kassociates}.
\begin{lemma}
Pick $0 \not=\kappa \in \mathbb F$ and consider
the Concrete Billiard Array
$\mathcal B':\Delta_N\to V$, $ \lambda \mapsto\kappa \mathcal B_\lambda$. 
Then
$\tilde {\mathcal B}_{\lambda,\mu} =
\tilde {\mathcal B}'_{\lambda,\mu}$ 
for all adjacent $\lambda, \mu$ in $\Delta_N$.
\end{lemma}

\begin{lemma}
\label{lem:CBAreciprocal}
Let $\lambda, \mu$ denote adjacent
locations in $\Delta_N$.
Then the scalars
$\tilde {\mathcal B}_{\lambda,\mu}$ and
$\tilde {\mathcal B}_{\mu,\lambda}$ are reciprocal.
\end{lemma}
\begin{proof}
By Lemma
\ref{lem:Binverses}
and Definition 
\ref{def:calT}. 
\end{proof}

\begin{lemma} 
\label{lem:concvalue}
\rm Let $\lambda, \mu$ denote
adjacent locations in $\Delta_N$.
Then the following are equivalent:
\begin{enumerate}
\item[\rm (i)] the vectors
$\mathcal B_\lambda, \mathcal B_\mu$ form a brace for $B$;
\item[\rm (ii)] 
 $\tilde {\mathcal B}_{\lambda,\mu}=1$;
\item[\rm (iii)] 
$\tilde {\mathcal B}_{\mu,\lambda}=1$.
\end{enumerate}
\end{lemma}
\begin{proof}
By Lemma
\ref{lem:exvalue}
and Definition 
\ref{def:calT}. 
\end{proof}

\begin{lemma} 
\label{lem:dep}
Let $\lambda, \mu, \nu$ denote
locations in $\Delta_N$ that form a black 3-clique.
Then
\begin{equation}
\label{eq:ld}
\mathcal B_\lambda
+ 
\tilde {\mathcal B}_{\lambda, \mu}\mathcal B_\mu
+
\tilde {\mathcal B}_{\lambda, \nu}\mathcal B_\nu = 0.
\end{equation}
\end{lemma}
\begin{proof}
Use Lemma
\ref{lem:threetermv}
and 
 Definition 
\ref{def:calT}. 
\end{proof}

\noindent 
We now consider the linear dependency
(\ref{eq:ld})
from a slightly more general point of view.
Let $\lambda, \mu, \nu$ denote locations in $\Delta_N$
that form a black 3-clique.
By Definition \ref{def:cba}(ii) the vectors
$\mathcal B_\lambda $,
$\mathcal B_\mu $,
$\mathcal B_\nu $ are linearly dependent. Therefore
there exist
scalars $a,b,c$ in $\mathbb F$ that are not all zero, and
\begin{equation}
\label{eq:abc}
a\mathcal B_\lambda +
b\mathcal B_\mu +
c\mathcal B_\nu =0.
\end{equation}
By Definition
\ref{def:brace}
the vectors 
$a\mathcal B_\lambda,
b\mathcal B_\mu,
c\mathcal B_\nu$ form an abrace for the given black 3-clique.
Each of $a,b,c$ is nonzero
by 
Lemma
\ref{lem:brace}.

\begin{lemma} 
\label{lem:3frac}
With the above
notation,
\begin{eqnarray*}
&&
\tilde {\mathcal B}_{\lambda,\mu}  =  b/a,
\qquad
\tilde {\mathcal B}_{\mu,\nu}  =  c/b,
\qquad
\tilde {\mathcal B}_{\nu,\lambda}  =  a/c,
\\
&&
\tilde {\mathcal B}_{\mu,\lambda}  =  a/b,
\qquad
\tilde {\mathcal B}_{\nu,\mu}  =  b/c,
\qquad
\tilde {\mathcal B}_{\lambda,\nu}  =  c/a.
\end{eqnarray*}
\end{lemma}
\begin{proof} 
To obtain the first and last equation,
compare
(\ref{eq:ld}),
(\ref{eq:abc}) in light of Lemma
\ref{lem:threedim2CBA}.
The remaining equations are obtained from these
by
cyclically permuting the roles of $\lambda,\mu,\nu$.
\end{proof}


\begin{lemma} 
\label{lem:whiteprod}
Let $\lambda, \mu, \nu$ denote
locations in $\Delta_N$ that form a black 3-clique.
Then
\begin{equation*}
          \tilde {\mathcal B}_{\lambda,\mu}
          \tilde {\mathcal B}_{\mu,\nu}
          \tilde {\mathcal B}_{\nu,\lambda}
	  = 1.
\end{equation*}
\end{lemma}
\begin{proof} Use
Lemma \ref{lem:whitevalue}
and
 Definition 
\ref{def:calT}, or use
Lemma \ref{lem:3frac}.
\end{proof}


\begin{lemma}
\label{lem:bcval}
Let $\lambda, \mu, \nu$ denote locations in $\Delta_N$
that form a white 3-clique. Then the
clockwise (resp. counterclockwise)
$ B$-value of the clique is equal to
\begin{equation*}
	 \tilde{ \mathcal B}_{\lambda,\mu}
	 \tilde{ \mathcal B}_{\mu,\nu}
	 \tilde{ \mathcal B}_{\nu,\lambda}
\end{equation*}
whenever the sequence 
$\lambda, \mu, \nu$ runs clockwise
(resp. counterclockwise)
around the clique.
\end{lemma}
\begin{proof} Use
Definitions 
\ref{def:orientedBval},
\ref{def:calT}.
\end{proof}

%
\section{Edge-labellings of $\Delta_N$}

\noindent 
Throughout this section fix
 $N\in \mathbb N$.

\begin{definition}
\label{def:EL}
\rm By an {\it edge-labelling of $\Delta_N$}
we mean a function $\beta$
that assigns to each ordered pair 
$\lambda, \mu$ of adjacent locations in $\Delta_N$
a scalar $\beta_{\lambda, \mu} \in  \mathbb F$  such that:
\begin{enumerate}
\item[\rm (i)] for adjacent locations $\lambda, \mu$ in $\Delta_N$,
\begin{equation*}
\beta_{\lambda, \mu}\beta_{\mu,\lambda} = 1;
\end{equation*}
\item[\rm (ii)] for locations $\lambda, \mu, \nu$  in $\Delta_N$
that form a black 
3-clique,
\begin{equation*}
\beta_{\lambda, \mu}
\beta_{\mu, \nu}
\beta_{\nu, \lambda}
= 1.
\end{equation*}
\end{enumerate}
\end{definition}

\begin{lemma}
\label{prop:CBAtoEL}
Let $\mathcal B$ denote a Concrete Billiard
Array over $\mathbb F$ that has diameter $N$.
Define a function
$\tilde {\mathcal B}$
that assigns to each ordered pair $\lambda, \mu$ of 
adjacent locations in $\Delta_N$ the scalar 
$\tilde {\mathcal B}_{\lambda,\mu}$ from Definition
\ref{def:calT}.
Then $\tilde {\mathcal B}$
is an edge-labelling of $\Delta_N$.
\end{lemma}
\begin{proof}
The function $\tilde {\mathcal B}$
satisfies the conditions (i), (ii)  of
Definition
\ref{def:EL} by
 Lemmas
\ref{lem:CBAreciprocal},
\ref{lem:whiteprod}, respectively.
\end{proof}

\noindent For the rest of this section the following notation
is in effect.
Let $\beta$ denote an edge-labelling of $\Delta_N$. Let 
$V$ denote a vector space over $\mathbb F$ that has dimension 
$N+1$.

\medskip
\noindent 
Our next goal is to describe the Concrete Billiard Arrays 
$\mathcal B$ on $V$ such that
$\beta = \tilde {\mathcal B}$.

%

\begin{lemma}
\label{lem:ELnonzero}
For adjacent locations $\lambda, \mu$ in $\Delta_N$ we have
$\beta_{\lambda,\mu}\not=0$.
\end{lemma}
\begin{proof} By Definition \ref{def:EL}(i).
\end{proof}

\begin{definition}
\label{def:betaFeas}
\rm 
A function
$\mathcal B:\Delta_N \to V$, $\lambda \mapsto \mathcal B_\lambda$
is said to be {\it $\beta$-dependent} whenever 
\begin{equation}
\label{eq:betaFeas}
{\mathcal B}_\lambda + 
\beta_{\lambda, \mu}{\mathcal B}_\mu + 
\beta_{\lambda, \nu}{\mathcal B}_\nu 
= 0
\end{equation}
 for  all locations 
$\lambda, \mu, \nu$ in $\Delta_N$
that form a black 3-clique.
\end{definition}

\begin{lemma} 
\label{lem:smallpt2}
Let $\mathcal B$ denote a
Concrete Billiard Array on $V$. Then $\mathcal B$ is
$\beta$-dependent if and only if $\beta = \tilde {\mathcal B}$.
\end{lemma}
\begin{proof}
Compare
(\ref{eq:ld}),
(\ref{eq:betaFeas}) in light of
Lemma
\ref{lem:threedim2CBA}.
\end{proof}

\noindent 
Let 
$\mathcal B:\Delta_N \to V$, $\lambda \mapsto \mathcal B_\lambda$
denote a $\beta$-dependent function.
Let 
$\lambda, \mu, \nu$ denote locations in $\Delta_N$
that form a black 3-clique. Then the vectors
$\mathcal B_\lambda,
\mathcal B_\mu,
\mathcal B_\nu$ satisfy 
(\ref{eq:betaFeas}). Cyclically permuting the roles of
$\lambda, \mu, \nu$ we obtain
\begin{eqnarray}
&&
{\mathcal B}_\lambda + 
\beta_{\lambda, \mu}{\mathcal B}_\mu + 
\beta_{\lambda, \nu}{\mathcal B}_\nu 
= 0,
\label{eq:betaFeas1}
\\
&&
\beta_{\mu, \lambda}{\mathcal B}_\lambda 
+
{\mathcal B}_\mu + 
\beta_{\mu, \nu}{\mathcal B}_\nu  
= 0,
\label{eq:betaFeas2}
\\
&&
\beta_{\nu, \lambda}{\mathcal B}_\lambda + 
\beta_{\nu, \mu}{\mathcal B}_\mu +
{\mathcal B}_\nu  
= 0.
\label{eq:betaFeas3}
\end{eqnarray}
\noindent  The linear dependencies 
(\ref{eq:betaFeas1})--(\ref{eq:betaFeas3}) are essentially
the same. To see this, consider the coefficient matrix
\begin{equation}
\label{eq:cm}
\left(
\begin{array}{ccc}
1 & \beta_{\lambda,\mu} & \beta_{\lambda, \nu} 
\\
\beta_{\mu,\lambda} & 1 & \beta_{\mu,\nu} 
\\
\beta_{\nu,\lambda} & \beta_{\nu,\mu} & 1 
\end{array}
\right).
\end{equation}
Using Definition
\ref{def:EL} one checks that this matrix has rank 1.

\begin{lemma}
\label{lem:betafeas2}
Assume $N\geq 1$. For a function
$\mathcal B:\Delta_N \to V$, $\lambda \mapsto \mathcal B_\lambda$
the following {\rm (i), (ii)} are equivalent:
\begin{enumerate}
\item[\rm (i)] $\mathcal B$ is $\beta$-dependent;
\item[\rm (ii)] for all $(r,s,t) \in \Delta_{N-1}$ we have
\begin{equation}
{\mathcal B}_\lambda + 
\beta_{\lambda, \mu}{\mathcal B}_\mu + 
\beta_{\lambda, \nu}{\mathcal B}_\nu 
= 0,
\label{eq:BF}
\end{equation}
where
\begin{equation*}
\lambda = (r+1,s,t), \qquad
\mu = (r,s+1,t), \qquad
\nu = (r,s,t+1).
\end{equation*}
\end{enumerate}
\end{lemma}
\begin{proof}
By Lemma \ref{lem:white} and
Definition
\ref{def:betaFeas}, along with the
discussion around
(\ref{eq:cm}).
\end{proof}

\noindent Given a $\beta$-dependent function
$\mathcal B:\Delta_N \to V, \lambda \to \mathcal B_\lambda$,
using the dependencies
(\ref{eq:BF}) we can solve for $\lbrace {\mathcal B}_\lambda
\rbrace_{\lambda \in \Delta_N}$ in terms of 
$\lbrace {\mathcal B}_{\lambda}\rbrace_{\lambda \in L}$,
where $L$ denotes the 
$1$-boundary of $\Delta_N$.
We give the details after a few definitions.

\begin{definition}
\label{def:tildecalB}
\rm
For each walk $\omega=\lbrace \lambda_i \rbrace_{i=0}^n$
in $\Delta_N$ we define the scalar
$\beta_\omega = \prod_{i=1}^n
\beta_{\lambda_{i-1},\lambda_i}$ called 
the {\it $\beta$-value of $\omega$}. 
Note that $\beta_\omega \not=0$ by
Lemma
\ref{lem:ELnonzero}.
We interpret $\beta_\omega=1$ if $n=0$.
\end{definition}

\begin{lemma}
\label{lem:mathcalrec}
For any walk $\omega$  in $\Delta_N$
the following scalars are reciprocal:
\begin{enumerate}
\item[\rm (i)] the $\beta$-value for $\omega$;
\item[\rm (ii)] the $\beta$-value for the inversion of $\omega$.
\end{enumerate}
\end{lemma}
\begin{proof}
By Definition
\ref{def:EL}(i) and
Definition
\ref{def:tildecalB}.
\end{proof}

\begin{definition}
\label{def:betaxy}
\rm
For $\lambda, \mu \in \Delta_N$ we define the scalar
\begin{equation}
\beta_{\lambda, \mu}
= \sum_{\omega} \beta_\omega,
\label{eq:betaSum}
\end{equation}
where the
sum is over all geodesic paths $\omega$ in $\Delta_N$ from
$\lambda$ to $\mu$. 
\end{definition}

\begin{note}\rm
\label{lem:smallpt}
Let $\lambda, \mu$ denote collinear locations
in $\Delta_N$. By
 Lemma
\ref{lem:collinGeo} there exists a unique
geodesic path $\omega$ in $\Delta_N$
from $\lambda$ to $\mu$.
 Now
$\beta_{\lambda, \mu} = \beta_\omega$ in view of
  Definition \ref{def:betaxy}.
In this case
$\beta_{\lambda, \mu} \not=0$.
\end{note}

\noindent 
Recall the path 
 $\lbrack 2,3\rbrack$ in $\Delta_N$, from Definition
\ref{def:sixpaths} and Lemma
\ref{lem:sixpathsinfo}. This path runs along the $1$-boundary of
$\Delta_N$, from the $2$-corner to the $3$-corner.

\begin{proposition} 
\label{prop:vigives}
Let $\lbrace \lambda_i \rbrace_{i=0}^N$ denote the
path $\lbrack 2,3\rbrack$ in $\Delta_N$.
Let $\lbrace v_i \rbrace_{i=0}^N$ denote
arbitrary vectors in $V$.
Then there exists a unique $\beta$-dependent function
 $\mathcal B:\Delta_N\to V, \lambda \mapsto \mathcal B_\lambda$
 that sends $\lambda_i \mapsto v_i$ for $0 \leq i \leq N$. 
For $\lambda = (r,s,t)$ in $\Delta_N$ we have
\begin{equation}
\mathcal B_\lambda = (-1)^r\sum_{i=t}^{r+t} \beta_{\lambda,\lambda_i}v_i.
\label{eq:Bformula}
\end{equation}
\end{proposition}
\begin{proof}
We first show  that $\mathcal B$ exists.
We define
the vectors $\lbrace \mathcal B_\lambda\rbrace_{\lambda \in \Delta_N}$ 
in the following
recursive way.
For $\xi=0,1,\ldots, N$ we define 
$\mathcal B_\lambda$ for those locations $\lambda \in \Delta_N$ 
that have $1$-coordinate $\xi$.
Let $\xi$ and $\lambda$ be given.
First assume  $\xi=0$, so that  $\lambda$ is on the
$1$-boundary of $\Delta_N$. Define
${\mathcal B}_{\lambda} = v_i$ where 
$\lambda = \lambda_i$. Next assume $1 \leq \xi \leq N$.
For notational convenience define $r=\xi-1$
and write $\lambda = (r+1,s,t)$.
Consider the locations
$\mu= (r,s+1,t)$ and $\nu=(r,s,t+1)$ in $\Delta_N$.
The locations $\lambda, \mu,\nu$ form a black 3-clique.
Each of $\mu, \nu$  
has $1$-coordinate $r$, so 
the vectors 
${\mathcal B}_{\mu}$ and ${\mathcal B}_\nu$  were defined
earlier in the recursion. Define  ${\mathcal B}_\lambda$
so that
(\ref{eq:BF}) holds.
The vectors $\lbrace {\mathcal B}_\lambda \rbrace_{\lambda \in \Delta_N}$
are now defined.
Consider the function
$\mathcal B:\Delta_N\to V$, $\lambda \mapsto 
\mathcal B_\lambda$. 
By construction 
$\mathcal B$ satisfies condition (ii) of 
 Lemma
\ref{lem:betafeas2}, so
by that lemma
$\mathcal B$  is $\beta$-dependent.
By construction $\mathcal B$ sends $\lambda_i \mapsto v_i$
for $0 \leq i \leq N$.
We have shown that $\mathcal B$ exists.
Now let $\mathcal B:\Delta_N \to V, \lambda\mapsto \mathcal B_\lambda$
denote any $\beta$-dependent function that sends $\lambda_i \mapsto v_i$
for $0 \leq i \leq N$. Then $\mathcal B$ satisfies
the equivalent conditions  (i), (ii) of
Lemma
\ref{lem:betafeas2}.
In condition (ii) there is a parameter
$r$; using condition (ii) and induction on $r$ one routinely
obtains
(\ref{eq:Bformula}). It follows from
(\ref{eq:Bformula}) that $\mathcal B$ is unique.
\end{proof}

\begin{proposition}
\label{prop:addmore}
With reference to Proposition
\ref{prop:vigives}, the following are equivalent:
\begin{enumerate}
\item[\rm (i)] $\mathcal B$ is a Concrete Billiard Array;
\item[\rm (ii)] the vectors 
$\lbrace v_i\rbrace_{i=0}^N$ are linearly independent.
\end{enumerate}
\end{proposition}
\begin{proof}
\noindent ${\rm (i)}\Rightarrow {\rm (ii)}$ By 
 Definition
\ref{def:cba}(i) and since $\lbrace \lambda_i\rbrace_{i=0}^N$
form a line in $\Delta_N$.
\\
\noindent ${\rm (ii)}\Rightarrow {\rm (i)}$ 
We show that $\mathcal B$ satisfies the conditions 
of Definition \ref{def:cba}.
Concerning 
Definition \ref{def:cba}(i),
let $L$ denote a line in $\Delta_N$.
Pick a location $\lambda = (r,s,t)$ in $L$
and consider the sum in line
(\ref{eq:Bformula}). In that sum,
by 
Note
\ref{lem:smallpt}
the coefficient
of $v_i$ is nonzero for $i=t$ and $i=r+t$.
Therefore the vectors $\lbrace {\mathcal B}_\lambda \rbrace_{\lambda \in L}$ 
are linearly independent. 
We have shown that $\mathcal B$ satisfies
 Definition \ref{def:cba}(i).
The function $\mathcal B$ satisfies
Definition \ref{def:cba}(ii), since
for each black 3-clique $C$ in $\Delta_N$
the vectors 
$\lbrace {\mathcal B}_\lambda \rbrace_{\lambda \in C}$ 
are linearly dependent by
(\ref{eq:betaFeas}).
We have shown that $\mathcal B$ satisfies
the conditions 
of Definition \ref{def:cba}, so
$\mathcal B$ is a 
 Concrete Billiard Array.
\end{proof}

\noindent We give two variations on
Propositions
\ref{prop:vigives},
\ref{prop:addmore}.

\begin{proposition}
\label{lem:prelim}
Let $\lbrace \lambda_i \rbrace_{i=0}^N$ denote the
path $\lbrack 2,3\rbrack$ in $\Delta_N$.
Then for any function
$\mathcal B:\Delta_N \to V$,
$\lambda \mapsto \mathcal B_\lambda$
the following {\rm (i), (ii)} are
equivalent:
\begin{enumerate}
\item[\rm (i)] $\mathcal B$ is a Concrete Billiard Array
 and $\beta = \tilde {\mathcal B}$;
\item[\rm (ii)] $\mathcal B$ is $\beta$-dependent
and $\lbrace \mathcal B_{\lambda_i}\rbrace_{i=0}^N$
are linearly independent.
\end{enumerate}
\end{proposition}
\begin{proof}
\noindent ${\rm (i)}\Rightarrow {\rm (ii)}$
The function $\mathcal B$ is $\beta$-dependent
by 
Lemma \ref{lem:smallpt2}.
 The vectors
$\lbrace {\mathcal B}_{\lambda_i}\rbrace_{i=0}^N$
are linearly independent by Definition
\ref{def:cba}(i).
\\
\noindent ${\rm (ii)}\Rightarrow {\rm (i)}$ 
Define $v_i  
 = {\mathcal B}_{\lambda_i}$
for $0 \leq i \leq N$.
By assumption,
the function $\mathcal B$ is $\beta$-dependent
and sends $\lambda_i \mapsto v_i$ for
$0 \leq i \leq N$. Therefore $\mathcal B$
meets the requirements of Proposition
\ref{prop:vigives}.
By assumption, the vectors 
$\lbrace v_i\rbrace_{i=0}^N$
are linearly independent.
Now  
$\mathcal B$
is a Concrete Billiard Array in view of
Proposition
\ref{prop:addmore}.
We have
 $\beta=\tilde {\mathcal B}$ by Lemma \ref{lem:smallpt2}.
\end{proof}

\begin{proposition}
\label{lem:ELtoCBA}
Let $\lbrace v_i \rbrace_{i=0}^N$ denote
a basis for $V$.
Then there exists a unique Concrete Billiard
Array $\mathcal B$ on $V$ such that $\beta=\tilde {\mathcal B}$
and 
 $\lbrace v_i \rbrace_{i=0}^N$ is the 
 $\mathcal B$-basis $\lbrack 2,3\rbrack$.
\end{proposition}
\begin{proof} 
Let $\lbrace \lambda_i\rbrace_{i=0}^N$ denote the
path $\lbrack 2,3\rbrack$ in $\Delta_N$.
We first show that $\mathcal B$ exists.
By Proposition
\ref{prop:vigives} there exists
a $\beta$-dependent function $\mathcal B:\Delta_N \to V$,
 $\lambda \mapsto \mathcal B_\lambda$ that sends
$\lambda_i \mapsto v_i$ for $0 \leq i \leq N$.
By construction
$\mathcal B$ satisfies condition (ii) of
Proposition
\ref{lem:prelim}.
By that proposition
$\mathcal B$ is a Concrete Billiard Array and $\beta = \tilde {\mathcal B}$.
By construction 
 $\lbrace v_i \rbrace_{i=0}^N$ is the 
 $\mathcal B$-basis $\lbrack 2,3\rbrack$.
We have shown that $\mathcal B$ exists.
Concerning 
 uniqueness, let
$\mathcal B$ denote any Concrete Billiard Array on $V$
such that
$\beta=\tilde {\mathcal B}$
and 
 $\lbrace v_i \rbrace_{i=0}^N$ is the 
 $\mathcal B$-basis $\lbrack 2,3\rbrack$.
Then $\mathcal B$ is $\beta$-dependent by
Proposition
\ref{lem:prelim}. By construction
the function $\mathcal B$ sends $\lambda_i\mapsto v_i$
for $0 \leq i \leq N$.
Therefore $\mathcal B$ meets the requirements  of
 Proposition
\ref{prop:vigives}. By that proposition
$\mathcal B$  is unique.
\end{proof}

\begin{corollary} Let $\mathcal B:\Delta_N \to V, \lambda \mapsto
\mathcal B_\lambda$ denote a $\beta$-dependent function.
Then $\mathcal B$ is a Concrete Billiard Array if and only
if the vectors 
$\lbrace \mathcal B_\lambda\rbrace_{\lambda \in \Delta_N}$
span $V$.
\end{corollary}
\begin{proof} First assume that
$\mathcal B$ is a Concrete Billiard Array.
Then the vectors
$\lbrace \mathcal B_\lambda\rbrace_{\lambda \in \Delta_N}$
span $V$ by
Corollary \ref{lem:BspanV2CBA}.
Conversely, assume that
the vectors
$\lbrace \mathcal B_\lambda\rbrace_{\lambda \in \Delta_N}$
span $V$. Let $\lbrace \lambda_i \rbrace_{i=0}^N$ denote
the path $\lbrack 2,3\rbrack$ in $\Delta_N$
and define $v_i = \mathcal B_{\lambda_i}$ for
$0 \leq i \leq N$. By Proposition
\ref{prop:vigives}
the vectors 
$\lbrace \mathcal B_\lambda\rbrace_{\lambda \in \Delta_N}$
are contained in the span of
 $\lbrace v_i\rbrace_{i=0}^N$.
Therefore the vectors
 $\lbrace v_i\rbrace_{i=0}^N$ span $V$.
The dimenension of $V$ is $N+1$ so
 $\lbrace v_i\rbrace_{i=0}^N$ are linearly independent.
Now $\mathcal B$ is a Concrete Billiard Array by Proposition
\ref{prop:addmore}.
\end{proof}

\section{Concrete Billiard Arrays and edge-labellings}

\noindent Throughout this section fix $N \in \mathbb N$.

\medskip
\noindent Let $\mathcal B$ denote a Concrete Billiard Array
over $\mathbb F$ that has diameter $N$. Recall from
Lemma
\ref{prop:CBAtoEL} the edge-labelling
$\tilde {\mathcal B}$
 of $\Delta_N$.

\begin{proposition}
\label{prop:CBAisoEL}
With the above notation,
the map $\mathcal B\mapsto \tilde {\mathcal B}$
induces a bijection between the following two sets:
\begin{enumerate}
\item[\rm (i)] the isomorphism classes of Concrete Billiard Arrays
over $\mathbb F$ that have diameter $N$;
\item[\rm (ii)] the edge-labellings of $\Delta_N$.
\end{enumerate}
\end{proposition}
\begin{proof} The map $\mathcal B \mapsto \tilde {\mathcal B}$
induces a function from set (i) to set (ii), and this function
is surjective by Proposition
\ref{lem:ELtoCBA}. We now show that the function is injective.
Let $\mathcal B, \mathcal B'$ denote Concrete Billiard
Arrays over $\mathbb F$ that have diameter $N$ and
$\tilde {\mathcal B} = \tilde {\mathcal B}'$. We 
show that 
$\mathcal B, \mathcal B'$ are isomorphic.
Let $V$ (resp. $V'$) denote the vector space underlying
$\mathcal B$ (resp. $\mathcal B'$).
Let $\lbrace v_i\rbrace_{i=0}^N$ denote the $\mathcal B$-basis
$\lbrack 2,3\rbrack$ for $V$, and 
let $\lbrace v'_i\rbrace_{i=0}^N$ denote the $\mathcal B'$-basis
$\lbrack 2,3\rbrack$ for $V'$.
Consider the $\mathbb F$-linear
bijection $\sigma :V\to V'$ that sends $v_i \mapsto v'_i$
for $0 \leq i \leq N$.
For $\lambda =(r,s,t)$ in $\Delta_N$ the vector $\mathcal B_\lambda$
satisfies 
(\ref{eq:Bformula}), where
$\beta = 
\tilde {\mathcal B} = \tilde {\mathcal B}'$.  Applying
$\sigma $ to each side of
(\ref{eq:Bformula}) we see
 that $\sigma$ sends $\mathcal B_\lambda \mapsto \mathcal B'_\lambda$.
So by Definition
\ref{def:cbaIso},
$\sigma$ is an isomorphism of Concrete Billiard Arrays
from 
$\mathcal B$ to $\mathcal B'$. Consequently
$\mathcal B, \mathcal B'$ are isomorphic, as desired.
The result follows.
\end{proof}

\begin{lemma} Let $\beta$ denote an edge-labelling of
$\Delta_N$. Let $\lbrace \kappa_\lambda\rbrace_{\lambda\in \Delta_N}$
denote nonzero scalars in $\mathbb F$. Then there exists
an edge-labelling $\beta'$ of $\Delta_N$ such that
$\beta_{\lambda, \mu}\kappa_\lambda 
=
\beta'_{\lambda, \mu} \kappa_\mu 
$
for all adjacent $\lambda, \mu $ in $\Delta_N$.
\end{lemma}
\begin{proof}  For all adjacent $\lambda, \mu \in
\Delta_N$ define
$\beta'_{\lambda,\mu} = \beta_{\lambda,\mu}\kappa_\lambda/\kappa_\mu$.
One checks that $\beta'$ is an edge-labelling of $\Delta_N$
with the required features.
\end{proof}

\begin{definition}
\label{def:ELsim}
\rm Edge-labellings $\beta, \beta'$ of 
$\Delta_N$ are called {\it similar} whenever
there exist nonzero scalars
$\lbrace \kappa_\lambda\rbrace_{\lambda\in \Delta_N}$ in $\mathbb F$
such that
$\beta_{\lambda,\mu} \kappa_\lambda = 
\beta'_{\lambda,\mu} \kappa_\mu $ for all
adjacent $\lambda, \mu$ in $\Delta_N$.
This similarity relation is an equivalence relation.
\end{definition}

\begin{example}
\label{ex:ELsmallN}
\rm
Assume $N=1$. Then any two edge-labellings of $\Delta_N$
are similar.
\end{example}

\begin{lemma}
\label{lem:simsim}
Let $\mathcal B$ and $\mathcal B'$ denote
Concrete Billiard Arrays over $\mathbb F$ that have diameter
$N$. Then the following are equivalent:
\begin{enumerate}
\item[\rm (i)] $\mathcal B$ and $\mathcal B'$ are similar in
the sense of Definition
\ref{def:relations};
\item[\rm (ii)] the edge-labellings
$\tilde {\mathcal B}$ and $\tilde {\mathcal B}'$ are similar in
the sense of Definition
\ref{def:ELsim}.
\end{enumerate}
\end{lemma}
\begin{proof} ${\rm (i)}\Rightarrow {\rm (ii)}$ 
By Proposition
\ref{prop:isoAssoc} there exists a Concrete Billiard
Array $\mathcal B''$ over $\mathbb F$ that is associated with 
$\mathcal B$ and isomorphic to $\mathcal B'$.
By Definition
\ref{def:assoc} there exist nonzero scalars
$\lbrace \kappa_\lambda\rbrace_{\lambda \in \Delta_N}$
in $\mathbb F$ such that 
$\mathcal B''_\lambda = \kappa_\lambda \mathcal B_\lambda $
for all $\lambda \in \Delta_N$.
For all adjacent $\lambda,\mu\in 
\Delta_N$ we have
$\tilde {\mathcal B}_{\lambda, \mu}\kappa_\lambda
=
\tilde {\mathcal B}''_{\lambda, \mu}\kappa_\mu$
by Lemma
\ref{lem:kassociates}, 
and
$\tilde {\mathcal B}''_{\lambda, \mu}
=
\tilde {\mathcal B}'_{\lambda, \mu}$
 by 
 Proposition
\ref{prop:CBAisoEL},
so
$\tilde {\mathcal B}_{\lambda, \mu}\kappa_\lambda
=
\tilde {\mathcal B}'_{\lambda, \mu}\kappa_\mu$.
Now 
$\tilde {\mathcal B}$ and
$\tilde {\mathcal B}'$ are similar in the sense of
Definition \ref{def:ELsim}.
\\
\noindent ${\rm (ii) \Rightarrow (i)}$
By Definition
\ref{def:ELsim} there exist nonzero 
scalars
$\lbrace \kappa_\lambda \rbrace_{\lambda \in \Delta_N}$
in $\mathbb F$ such that 
$\tilde {\mathcal B}_{\lambda,\mu} \kappa_\lambda = 
\tilde {\mathcal B}'_{\lambda,\mu} \kappa_\mu $ for all
adjacent $\lambda, \mu$ in $\Delta_N$.
Let $\mathcal B''$ denote the 
Concrete Billiard Array over $\mathbb F$ that sends
$\lambda \mapsto \kappa_\lambda \mathcal B_\lambda$
for all $\lambda \in \Delta_N$.
Then $\mathcal B''$ is associated with $\mathcal B$.
For all adjacent $\lambda, \mu \in \Delta_N$
we have
$\tilde {\mathcal B}_{\lambda, \mu} 
\kappa_\lambda =
\tilde {\mathcal B}''_{\lambda, \mu} 
\kappa_\mu$
by Lemma
\ref{lem:kassociates},
so
$\tilde {\mathcal B}'_{\lambda, \mu} =
\tilde {\mathcal B}''_{\lambda, \mu}$.
Therefore
$\mathcal B', \mathcal B''$ are isomorphic
by Proposition
\ref{prop:CBAisoEL}.
The Concrete Billiard Array $\mathcal B''$ is
associated with $\mathcal B$ and isomorphic to
$\mathcal B'$. Now by 
Proposition
\ref{prop:isoAssoc}, 
 $\mathcal B$ and $\mathcal B'$
are similar in the sense of
Definition
\ref{def:relations}.
\end{proof}

\noindent The following result is a variation on Proposition
\ref{prop:CBAisoEL}.

\begin{proposition}
\label{prop:simSim}
The map $\mathcal B \mapsto \tilde {\mathcal B}$ induces a bijection
between the following two sets:
\begin{enumerate}
\item[\rm (i)] the similarity classes of Concrete Billiard Arrays
over $\mathbb F$ that have diameter $N$;
\item[\rm (ii)] the similarity classes of edge-labellings for
$\Delta_N$.
\end{enumerate}
\end{proposition}
\begin{proof} By Proposition
\ref{prop:CBAisoEL}
and
Lemma \ref{lem:simsim}.
\end{proof}

\section{Edge-labellings and value functions}

\noindent Throughout this section fix $N \in \mathbb N$.

\medskip
\noindent For this section our goal is to
describe the similarity classes of 
edge-labellings for $\Delta_N$. Note that if $N=0$ then
$\Delta_N$ has no edges.
If $N=1$ then by Example
\ref{ex:ELsmallN},
any two edge-labellings of $\Delta_N$ are similar.
For $N\geq 2$ we will display a bijection
between the following two sets:
(i) 
the similarity classes of edge-labellings for $\Delta_N$;
(ii) 
the value functions on $\Delta_{N-2}$.

\begin{definition} 
\label{def:orientedbetaval}
\rm
Let $\beta$ denote an edge-labelling of $\Delta_N$.
Let $\lambda, \mu, \nu$ denote locations in $\Delta_N$
that form a white 3-clique. Then the scalar
\begin{equation*}
\beta_{\lambda,\mu}
\beta_{\mu,\nu}
\beta_{\nu,\lambda}
\end{equation*}
is called the {\it clockwise $\beta$-value} 
(resp. {\it counterclockwise $\beta$-value})
of the clique whenever
the sequence 
 $\lambda, \mu, \nu$ runs clockwise (resp. counterclockwise)
 around the clique.
\end{definition}

\begin{lemma} 
Let $\beta$ denote an edge-labelling of $\Delta_N$.
For each white 3-clique in $\Delta_N$, 
its clockwise $\beta$-value and counterclockwise $\beta$-value are reciprocal.
\end{lemma}
\begin{proof}
By Definition \ref{def:EL}(i) and
Definition
\ref{def:orientedbetaval}.
\end{proof}

\begin{definition}
\label{def:betavalStraight}
\rm
Let $\beta$ denote an edge-labelling of $\Delta_N$.
 For each white 3-clique in $\Delta_N$,
by its {\it $\beta$-value}  we mean its clockwise $\beta$-value.
\end{definition}


\begin{definition}
\label{def:el}
\rm Assume $N\geq 2$, and let
$\beta$ denote an edge-labelling of $\Delta_N$.
We define a function
$\hat \beta:\Delta_{N-2} \to \mathbb F$ as follows.
Pick $(r,s,t) \in \Delta_{N-2}$. To describe
the image of $(r,s,t)$ under $\hat \beta$,
consider the corresponding white 3-clique in $\Delta_N$
from Lemma
\ref{lem:black}. The  $\beta$-value of this 3-clique
is the image of $(r,s,t)$ under
$\hat \beta$.
Note that $\hat \beta$ is a value function on
$\Delta_{N-2}$, in the sense of Definition
\ref{def:vfunction}.
We call $\hat \beta$ the {\it value function} for $\beta$.
\end{definition}

\begin{lemma}
\label{lem:diagcom}
Let $\mathcal B$ denote a Concrete Billiard
Array over $\mathbb F$ that has diameter $N$. Let $B$
denote the corresponding Billiard Array.
 Then for each white 3-clique in 
$\Delta_N$ the $B$-value is equal to the $\tilde {\mathcal B}$-value.
In other words, $B$ and $\tilde {\mathcal B}$ have the same
value function.
\end{lemma}
\begin{proof}
Compare
 Lemma
\ref{lem:bcval}
and 
Definition
\ref{def:orientedbetaval}.
\end{proof}

\begin{proposition}
\label{thm:cyclevalue}
Let $\beta$ denote an edge-labelling of $\Delta_N$.
Given a cycle
in $\Delta_N$
that runs clockwise (resp. counterclockwise), consider
the white 3-cliques that are surrounded by the cycle.
Then the following scalars coincide:
\begin{enumerate}
\item[\rm (i)] 
the $\beta$-value of the cycle, in the sense of Definition
\ref{def:tildecalB};
\item[\rm (ii)]
the product of the
clockwise (resp. counterclockwise)
$\beta $-values for these
white 3-cliques.
\end{enumerate}
\end{proposition}
\begin{proof}
Use Definitions
\ref{def:EL},
\ref{def:tildecalB},
\ref{def:orientedbetaval}.
\end{proof}

\noindent
Let $\beta$ denote an edge-labelling of $\Delta_N$.
Let $T$ denote a spanning tree of $\Delta_N$, as in
Definition
\ref{def:spanningT}.
Consider the scalars
\begin{equation}
\label{eq:data}
 \beta_{\lambda, \mu}  \qquad
 \lambda, \mu \in \Delta_N, \qquad
 {\mbox{\rm $\lambda, \mu$  form an edge
 in $T$}}.
 \end{equation}
 For adjacent  $\lambda, \mu \in
 \Delta_N$ we now compute
   $\beta_{\lambda,\mu}$
   in terms of 
   (\ref{eq:data}) and the
   value function
   $\hat \beta$. 
   First assume that the edge formed by
   $\lambda, \mu$ is in $T$. Then
     $\beta_{\lambda,\mu}$ is included in
     (\ref{eq:data}), so we are done.
     Next assume that the edge formed by $\lambda, \mu$ 
     is not in $T$.
     There exists a unique path in $\Delta_N$
     from $\mu$ to  $\lambda$ that involves only edges
     in $T$.
     Denote this path by
     $\lbrace \lambda_i\rbrace_{i=1}^{n}$. By construction
     $\lambda_1 = \mu$ and $\lambda_{n} = \lambda$.
     Define $\lambda_{0} = \lambda$ and note that
     $\lbrace \lambda_i\rbrace_{i=0}^n$ is a cycle in $\Delta_N$.
Apply 
Proposition
\ref{thm:cyclevalue}
 to this cycle.
 For this cycle, the scalar in
 Proposition
 \ref{thm:cyclevalue}(i) 
is equal to 
 $\prod_{i=1}^n
 \beta_{\lambda_{i-1},\lambda_i}$,
and the scalar in  
 Proposition
 \ref{thm:cyclevalue}(ii) is determined by
 $\hat \beta$.
Therefore
 $\prod_{i=1}^n
 \beta_{\lambda_{i-1},\lambda_i}$
 is
 determined by $\hat \beta$.
 In the product
 $\prod_{i=1}^n
 \beta_{\lambda_{i-1},\lambda_i}$
 consider the $i$-factor  for $1 \leq i \leq n$.
 This $i$-factor is
  $\beta_{\lambda,\mu}$ for $i=1$,
  and is included in
  (\ref{eq:data}) for $2\leq i \leq n$.
  Using these comments we routinely solve for
   $\beta_{\lambda,\mu}$ in terms of
   (\ref{eq:data}) and
   $\hat \beta$.
   \medskip

\noindent The scalars
(\ref{eq:data}) are ``free'' in the following sense.

\begin{proposition}
\label{prop:free}
Assume $N\geq 2$, and
let $\psi$ denote a value function on
$\Delta_{N-2}$. Let $T$ denote a spanning tree of
$\Delta_N$.
Consider a collection of scalars 
\begin{equation}
\label{eq:betadata2}
 b_{\lambda, \mu} \in \mathbb F, 
 \qquad 
\lambda, \mu \in \Delta_N, \qquad
{\mbox{\rm $\lambda, \mu$  form an edge 
in $T$}}
\end{equation}
such that
 $b_{\lambda, \mu} 
 b_{\mu,\lambda} = 1$ for all
 $\lambda, \mu \in \Delta_N$ that form
 an edge in $T$. Then
there exists a unique edge-labelling $\beta$ of $\Delta_N$
that has value function $\psi$ and
$\beta_{\lambda, \mu} = b_{\lambda, \mu}$
for all 
 $\lambda, \mu \in \Delta_N$ that form
 an edge in $T$.
\end{proposition}
\begin{proof}
We first show that $\beta$ exists.
To do this we mimic the argument from below
(\ref{eq:data}).
For adjacent  $\lambda, \mu \in \Delta_N$
we define a scalar $\beta_{\lambda,\mu}$
as follows. First assume that 
  the edge formed by  
   $\lambda, \mu$ is in $T$. Define
    $\beta_{\lambda,\mu}= b_{\lambda, \mu}$.
     Next assume that the edge formed by $\lambda, \mu$ 
     is not in $T$.
     There exists a unique path in $\Delta_N$
     from $\mu$ to  $\lambda$ that involves only edges
     in $T$.
     Denote this path by
     $\lbrace \lambda_i\rbrace_{i=1}^{n}$. By construction
     $\lambda_1 = \mu$ and $\lambda_{n} = \lambda$.
     Define $\lambda_{0} = \lambda$ and note that
     $\lbrace \lambda_i\rbrace_{i=0}^n$ is a cycle in $\Delta_N$.
  Define $\varepsilon =1$ (resp. $\varepsilon = -1$) if this cycle
  is clockwise (resp. counterclockwise).
Define $\beta_{\lambda,\mu}=c^{\varepsilon}/d$ where
 $d=\prod_{i=2}^n
 b_{\lambda_{i-1},\lambda_i}$ and $c$
 is the product of the $\psi$-values of the 
white 3-cliques surrounded by the cycle.
We have defined the scalar $\beta_{\lambda,\mu}$ for
all adjacent $\lambda, \mu \in \Delta_N$. 
One checks that these scalars satisfy
the conditions of Definition
\ref{def:EL}. This gives an edge-labelling $\beta$ 
of $\Delta_N$. By construction $\beta$ has value
function $\psi$ and
$\beta_{\lambda, \mu} = b_{\lambda, \mu}$
for all 
 $\lambda, \mu \in \Delta_N$ that form
 an edge in $T$.
 We have shown that  $\beta$
exists. The edge-labelling $\beta$ is unique by the discussion
below 
(\ref{eq:data}).
\end{proof}

\begin{corollary} 
\label{cor:beta2}
Assume $N\geq 2$, and let $\psi$ denote a value function on
$\Delta_{N-2}$. Let $T$ denote a spanning tree of $\Delta_N$.
Then there exists a unique edge-labelling $\beta$ of $\Delta_N$
that has value function $\psi$ and
$\beta_{\lambda, \mu} = 1$
for all 
 $\lambda, \mu \in \Delta_N$ that form
 an edge in $T$.
\end{corollary}
\begin{proof}
Apply Proposition
\ref{prop:free},
with $b_{\lambda, \mu}=1$ for all 
$\lambda, \mu \in \Delta_N$ that form
an edge in $T$.
\end{proof}

\noindent
Assume $N\geq 2$. By Definition
\ref{def:el} we get a map
$\beta\mapsto \hat \beta$
 from the set 
of  edge-labellings of $\Delta_N$, to the set of 
value functions on $\Delta_{N-2}$. This map is
surjective by
Corollary \ref{cor:beta2}. We now consider the
issue of injectivity.

\begin{lemma}
\label{lem:simval}
Assume $N\geq 2$.
Let $\beta$ and $\beta'$ denote edge-labellings
of $\Delta_N$. Then the following are equivalent:
\begin{enumerate}
\item[\rm (i)] $\beta$ and $\beta'$ are similar;
\item[\rm (ii)] $\beta$ and $\beta'$ have the
same value function.
\end{enumerate}
\end{lemma}
\begin{proof} 
${\rm (i) \Rightarrow (ii)}$  Let $\lambda, \mu, \nu$ denote
locations in $\Delta_N$ that form a white 3-clique. We show
that 
\begin{equation}
\label{eq:show}
\beta_{\lambda,\mu}
\beta_{\mu,\nu}
\beta_{\nu,\lambda}
=
\beta'_{\lambda,\mu}
\beta'_{\mu,\nu}
\beta'_{\nu,\lambda}.
\end{equation}
By Definition
\ref{def:ELsim} there exist nonzero 
$\kappa_\lambda, 
\kappa_\mu, 
\kappa_\nu \in \mathbb F$ such that
\begin{equation}
\beta_{\lambda,\mu} \kappa_\lambda = 
\beta'_{\lambda,\mu} \kappa_\mu,
\qquad
\beta_{\mu,\nu} \kappa_\mu = 
\beta'_{\mu,\nu} \kappa_\nu,
\qquad
\beta_{\nu,\lambda} \kappa_\nu = 
\beta'_{\nu,\lambda} \kappa_\lambda.
\label{eq:howadj}
\end{equation}
Line (\ref{eq:show}) is routinely verified using
(\ref{eq:howadj}).
\\
\noindent ${\rm (ii) \Rightarrow (i)}$
Let $T$ denote a spanning tree of $\Delta_N$.
Fix a location $\nu \in \Delta_N$.
For $\lambda \in \Delta_N$
(and
with reference to Definition
\ref{def:tildecalB})
define
$\kappa_\lambda = \beta'_\omega /\beta_\omega$
where $\omega$ denotes the 
 unique path in $\Delta_N$ from $\lambda$ to
$\nu$ that involves only
edges in $T$.
Note that $\kappa_\lambda \not=0$.
Using 
Proposition \ref{thm:cyclevalue} one checks
that
$\beta_{\lambda,\mu} \kappa_\lambda = 
\beta'_{\lambda,\mu} \kappa_\mu $ for all
adjacent $\lambda, \mu$ in $\Delta_N$.
Therefore $\beta, \beta'$ are similar in
view of Definition
\ref{def:ELsim}.
 \end{proof}

\noindent Assume $N\geq 2$, and let $\beta$ denote an edge-labelling of 
$\Delta_N$. Recall from Definition
\ref{def:el}
the value function $\hat \beta$ on $\Delta_{N-2}$.

\begin{proposition}
\label{prop:simVal}
With the above notation,
the map $\beta \mapsto \hat \beta$ induces a bijection
between the following two sets:
\begin{enumerate}
\item[\rm (i)] the similarity classes of edge-labellings for $\Delta_N$;
\item[\rm (ii)] the value functions on $\Delta_{N-2}$.
\end{enumerate}
\end{proposition}
\begin{proof}
By Lemma
\ref{lem:simval} the map
 $\beta \mapsto \hat \beta$ induces an injective function
from set (i) to set (ii).
The function is surjective by
Corollary \ref{cor:beta2}.
Therefore the function is a bijection.
\end{proof}

\section{ Billiard Arrays and value functions}

From now until Proposition
\ref{prop:rell}, fix
an integer $N \geq 2$.
\medskip


\noindent To motivate the next result, we review a few
points.
Let ${\rm{BA}}_N(\mathbb F)$ denote
the set of isomorphism classes of Billiard Arrays over $\mathbb F$ 
that have diameter $N$.
Let 
 ${\rm{CBA}}_N(\mathbb F)$ denote
the set of  similarity classes of Concrete Billiard Arrays over $\mathbb F$ that have 
diameter $N$.
 Let ${\rm{EL}}_N(\mathbb F)$ denote
the set of similarity classes of edge-labellings on $\Delta_N$.
 Let ${\rm{VF}}_N(\mathbb F)$ denote
the set of value functions on 
$\Delta_N$.
The map $B \mapsto \hat B$ induces a function
 ${\rm{BA}}_N(\mathbb F) \to
 {\rm{VF}}_{N-2}(\mathbb F)$ which we will denote by $\theta $.
In Lemma
\ref{lem:assocIso} we displayed a bijection
 ${\rm{CBA}}_N(\mathbb F) \to
 {\rm{BA}}_N(\mathbb F)$ which we will denote by $f$.
In Proposition
\ref{prop:simSim} we displayed a bijection
 ${\rm{CBA}}_N(\mathbb F) \to
 {\rm{EL}}_N(\mathbb F)$ which we will denote by $g$.
In Proposition
\ref{prop:simVal} 
we displayed a bijection
 ${\rm{EL}}_N(\mathbb F) \to
 {\rm{VF}}_{N-2}(\mathbb F)$ which we will denote by $h$.
\begin{lemma} 
\label{lem:dcom}
With the above notation, the following
diagram commutes:

\begin{equation*}
\begin{CD}
\mbox{ ${\rm{CBA}}_N(\mathbb F)$  } @>f> >
                \mbox{ ${\rm{BA}}_N(\mathbb F)$} 
           \\ 
          @VgVV                     @VV\theta V \\
                \mbox{${\rm{EL}}_N(\mathbb F) $} @>>h> 
                \mbox{$ {\rm{VF}}_{N-2}(\mathbb F)$}
                   \end{CD}
\end{equation*}

\noindent Moreover $\theta$ is a bijection.
\end{lemma}
\begin{proof}
The diagram commutes by
Lemma \ref{lem:diagcom}. It follows that
$\theta$ is a bijection.
\end{proof}

\noindent We emphasize one aspect of Lemma
\ref{lem:dcom}.

\begin{corollary}
\label{thm:baVf}
The map $B\mapsto {\hat B}$
 induces a bijection between the
following two sets:
\begin{enumerate}
\item[\rm (i)] the isomorphism classes of Billiard Arrays over $\mathbb F$
that have
diameter $N$;
\item[\rm (ii)] the value functions on $\Delta_{N-2}$.
\end{enumerate}
\end{corollary}
\begin{proof} 
The map $\theta$ in Lemma
\ref{lem:dcom} is a bijection.
\end{proof}

\noindent
Summarizing the above discussion, 
we obtain a bijection between
any two of the following sets:
\begin{enumerate}
\item[$\bullet$] the isomorphism classes of Billiard Arrays over $\mathbb F$
that have diameter $N$;
\item[$\bullet$] the similarity classes of Concrete Billiard Arrays over $\mathbb F$ that have 
diameter $N$;
\item[$\bullet$] the similarity classes of edge-labellings for $\Delta_N$;
\item[$\bullet$] the value functions on $\Delta_{N-2}$.
\end{enumerate}

\noindent We have some remarks.
For the rest of this section fix $N \in \mathbb N$.
Let $V$ denote a vector space over $\mathbb F$ with
dimension $N+1$.

\begin{proposition}
\label{prop:rell}
Let $\mathcal B$ and $\mathcal B'$ denote Concrete Billiard Arrays on $V$. Then
the following are equivalent:
\begin{enumerate}
\item[\rm (i)]  $\mathcal B$ and $\mathcal B'$ are relatives;
\item[\rm (ii)]  $\mathcal B$ and $\mathcal B'$ are associates  and isomorphic. 
\end{enumerate}
\end{proposition}
\begin{proof}
${\rm (i) \Rightarrow (ii)}$ 
By Lemma
\ref{lem:fourrel}.
\\
\noindent ${\rm (ii) \Rightarrow (i)}$ By
Definition
\ref{def:assoc} and
since  
 $\mathcal B, \mathcal B'$ are associates,
there exist nonzero scalars
$\lbrace \kappa_\lambda\rbrace_{\lambda \in \Delta_N}$ in $\mathbb F$
such that $\mathcal B'_\lambda  = \kappa_\lambda \mathcal B_\lambda$
for all $\lambda \in \Delta_N$. 
For all adjacent $\lambda, \mu \in \Delta_N$
we have
$
\tilde {\mathcal B}_{\lambda, \mu}\kappa_\lambda 
= 
\tilde {\mathcal B}'_{\lambda, \mu}\kappa_\mu
$
by Lemma
\ref{lem:kassociates},
and 
$\tilde {\mathcal B}_{\lambda, \mu}
= 
\tilde {\mathcal B}'_{\lambda, \mu}
$
since $\mathcal B, \mathcal B'$ are isomorphic,
 so
 $\kappa_\lambda = \kappa_\mu$.
Consequently there exists $0 \not=\kappa \in \mathbb F$
such that $\kappa_\lambda = \kappa$ for all $\lambda \in \Delta_N$.
Now $\mathcal B$ and $\mathcal B'$ are relatives by Definition
\ref{def:rel}.
\end{proof}

\begin{proposition}
Let $B$ denote a Billiard Array on $V$.
For an $\mathbb F$-linear map
$\sigma :V\to V$ 
the following are equivalent:
\begin{enumerate}
\item[\rm (i)] the map $\sigma $ is an isomorphism
of Billiard Arrays from $B$ to $B$;
\item[\rm (ii)] there exists $0 \not=\kappa \in \mathbb F$
such that $\sigma = \kappa I$.
\end{enumerate}
\end{proposition}
\begin{proof} 
${\rm (i) \Rightarrow (ii)}$
Let $\mathcal B$ denote a  Concrete Billiard Array on $V$
that 
corresponds to $B$. Consider the Concrete Billiard Array
$\mathcal B':\Delta_N\to V$, $\lambda \mapsto \sigma(\mathcal B_\lambda)$.
By construction $\sigma$ is an isomorphism of
Concrete Billiard Arrays  from $\mathcal B$ to $\mathcal B'$;
 therefore $\mathcal B$ and $\mathcal B'$ are isomorphic.
 By construction $\mathcal B'$ corresponds to $B$; therefore
 $\mathcal B$ and $\mathcal B'$ are associates.
Now  $\mathcal B$ and $\mathcal B'$ are relatives by
Proposition \ref{prop:rell}. By Definition
\ref{def:rel}
there exists 
$0 \not=\kappa \in \mathbb F$ such that 
$\mathcal B'_\lambda = \kappa \mathcal B_\lambda$
for all $\lambda \in \Delta_N$. Now $\kappa I$
is an isomorphism of Concrete Billiard Arrays from
$\mathcal B$ to $\mathcal B'$.
By
Definition
\ref{def:cbaIso},
there does not exist another
isomorphism of Concrete Billiard Arrays from
$\mathcal B$ to $\mathcal B'$.
Therefore $\sigma=\kappa I$.
\\
\noindent
${\rm (ii) \Rightarrow (i)}$ Clear.
\end{proof}

\noindent We make a definition for later use.
\begin{definition}
\label{def:BBB}
\rm Let $\mathcal B$ denote a Concrete Billiard Array on $V$,
and let $B$ denote the corresponding Billiard Array on $V$. Let $C$
denote a white 3-clique in $\Delta_N$.
By the {\it $\mathcal B$-value} of $C$
we mean the $B$-value of $C$,
which is the same as the 
$\tilde {\mathcal B}$-value of $C$
by Lemma
\ref{lem:diagcom}.
\end{definition}


\section{Examples of Concrete Billiard Arrays}

\noindent Throughout this section the following notation is in
effect. 
Let $x, y,  z$ denote mutually commuting indeterminates.
Let $P=\mathbb F \lbrack x, y,  z\rbrack$ denote the $\mathbb F$-algebra
consisting of the polynomials in $ x,  y,  z$ 
that have all coefficients
in $\mathbb F$. For
 $n \in \mathbb N$ let $P_n$ denote
the subspace of $P$ consisting of the homogeneous polynomials
that have total degree $n$. 
The sum $P = \sum_{n\in \mathbb N} P_n$ is direct. 
Fix an integer $N\geq 1$.
In this section we give some examples of Concrete Billiard Arrays
of diameter $N$.
For these examples, the underlying vector space will be
a  subspace
of $P_N$.
\medskip

\noindent We now describe our first example.

\begin{definition}\rm
\label{def:CBAv1}
For each location $\lambda = (r,s,t) $ in $\Delta_N$ define
\begin{equation*}
{\mathcal B}_\lambda = ( x- y)^r (y-z)^s (z-x)^t.
\end{equation*}
\end{definition}
\noindent We will show that the function
$\mathcal B$ from Definition
\ref{def:CBAv1}
is a Concrete Billiard Array.

\begin{lemma}\label{lem:sumis0}
Let $\lambda$, $\mu$, $\nu$ denote locations in
$\Delta_N$ that form a black 3-clique. Then
\begin{equation}
\mathcal B_\lambda + 
\mathcal B_\mu + 
\mathcal B_\nu =0.
\label{eq:sumis0}
\end{equation}
\end{lemma}
\begin{proof} By Lemma
\ref{lem:white} we may take
\begin{eqnarray*}
\lambda = (r+1,s,t), \qquad
\mu = (r,s+1,t), \qquad
\nu = (r,s,t+1)
\end{eqnarray*}
with $(r,s,t) \in \Delta_{N-1}$.
By Definition
\ref{def:CBAv1},
\begin{eqnarray*}
&&\mathcal B_\lambda = (x-y)^{r+1}(y-z)^s (z-x)^t,
\\
&&\mathcal B_\mu = (x-y)^r(y-z)^{s+1} (z-x)^t,
\\
&&\mathcal B_\nu = (x-y)^r (y-z)^s (z-x)^{t+1}.
\end{eqnarray*}
From this we routinely obtain
(\ref{eq:sumis0}).
\end{proof}

\begin{lemma}
\label{lem:CBAcheck}
The function $\mathcal B$ from Definition
\ref{def:CBAv1} is a Concrete Billiard Array.
\end{lemma}
\begin{proof} We show that $\mathcal B$ satisfies the
two conditions in Definition
\ref{def:cba}. Concerning Definition
\ref{def:cba}(i), let $L$ denote a line in
$\Delta_N$ and consider the locations in $L$.
We show that their images under $\mathcal B$ are
linearly independent. Without loss we may assume
that $L$ is a $1$-line. Denote the cardinality
of $L$ by $i+1$. The locations in $L$ are listed
in 
(\ref{eq:oneline}).
For these locations 
their images under $\mathcal B$ are
\begin{equation}
(x-y)^{N-i}(y-z)^{i-j}(z-x)^j \qquad \qquad 0 \leq j \leq i.
\label{eq:line2}
\end{equation}
We show that the vectors 
(\ref{eq:line2}) are linearly independent. The vectors
\begin{equation}
(y-z)^{i-j}(z-x)^j \qquad \qquad 0 \leq j \leq i
\label{eq:line3}
\end{equation}
are linearly independent; to see this set $z=0$ in
(\ref{eq:line3}) and note that $x,y$ are algebraically
independent. It follows that the vectors
(\ref{eq:line2}) are linearly independent.
We have shown that $\mathcal B$ satisfies Definition
\ref{def:cba}(i).
The function $\mathcal B$ satisfies 
Definition
\ref{def:cba}(ii) by Lemma 
\ref{lem:sumis0}.
\end{proof}

\noindent For the 
 Concrete Billiard Array $\mathcal B$ in
Definition
\ref{def:CBAv1}
 we now compute
 the corresponding edge labelling and value function.

\begin{lemma}
\label{lem:CBAv1EL}
Let $\mathcal B$ denote the Concrete Billiard Array from
Definition
\ref{def:CBAv1}.
For adjacent locations $\lambda, \mu$  in $\Delta_N$ we have
$\tilde {\mathcal B}_{\lambda,\mu} = 1$.
\end{lemma}
\begin{proof} By Lemma
\ref{lem:edgegivesclique} there exists a location
$\nu \in \Delta_N$ such that $\lambda$, $\mu$, $\nu$
form a black 3-clique. Now
$\mathcal B_\lambda + 
\mathcal B_\mu + 
\mathcal B_\nu = 0$ by Lemma
\ref{lem:sumis0}. Comparing this with
(\ref{eq:ld}) we find that in
(\ref{eq:ld}) the terms $\mathcal B_\mu$
and $\mathcal B_\nu$ have coefficient 1. 
Therefore
 $\tilde {\mathcal B}_{\lambda, \mu}=1$.
\end{proof}

\begin{proposition}  
Let $\mathcal B$ denote the Concrete Billiard Array from
Definition
\ref{def:CBAv1}.
Then each white 3-clique in $\Delta_N$ has ${\mathcal B}$-value 1.
\end{proposition}
\begin{proof}
By Lemmas
\ref{lem:bcval},
\ref{lem:CBAv1EL} along with Definition
\ref{def:BBB}.
\end{proof}

\noindent We are done with our first example.  We now describe our second
example.
\medskip  
  
 \noindent 
We will use the following notation. Fix $0 \not=q \in \mathbb F$
such that $q \not=1$.
For elements $a,b$ in any $\mathbb F$-algebra, define
\begin{equation*}
(a,b;q)_n = (a-b)(a-bq)(a-bq^2)\cdots (a-bq^{n-1})
\qquad \qquad n=0,1,2,\ldots
\end{equation*}
We interpret 
$(a,b;q)_0 = 1$.

\begin{definition} 
\label{def:CBAvq}
Pick three scalars $\overline x$,
$\overline y$,
$\overline z$ in $\mathbb F$ such that
${\overline x}\,
{\overline y}\,
{\overline z} = q^{N-1}$.
For each location $\lambda = (r,s,t) $ in $\Delta_N$ define
\begin{equation*}
{\mathcal B}_\lambda = 
(x \overline x, y;q)_r 
(y \overline y, z;q)_s 
(z \overline z, x;q)_t. 
\end{equation*}
\end{definition}

\noindent We are going to show that the function
$\mathcal B$ from Definition
\ref{def:CBAvq}
is a Concrete Billiard Array.

\noindent 
Pick a location $(r,s,t) \in
\Delta_{N-1}$ and consider the corresponding
black 3-clique in $\Delta_N$ from Lemma
\ref{lem:white}: 
\begin{equation*}
\lambda = (r+1,s,t),
\qquad
\mu = (r,s+1,t),
\qquad
\nu = (r,s,t+1).
\end{equation*}

\begin{lemma}
\label{lem:weightedsum}
With the above notation,
\begin{equation}
\label{eq:weightedsum}
a \mathcal B_\lambda +
b \mathcal B_\mu +
c \mathcal B_\nu = 0
\end{equation}
where each of $a,b,c$ is nonzero and
\begin{equation}
\label{eq:abcdata}
b/a = q^r/\overline y,
\qquad \qquad
c/b = q^s/\overline z,
\qquad \qquad
a/c = q^t/\overline x.
\end{equation}
\end{lemma}
\begin{proof} By Definition 
\ref{def:CBAvq} and the construction,
\begin{eqnarray*}
&&{\mathcal B}_\lambda = 
(x \overline x, y;q)_{r+1} 
(y \overline y, z;q)_s 
(z \overline z, x;q)_t,
\\
&&{\mathcal B}_\mu = 
(x \overline x, y;q)_r 
(y \overline y, z;q)_{s+1} 
(z \overline z, x;q)_t,
\\
&&{\mathcal B}_\nu = 
(x \overline x, y;q)_r 
(y \overline y, z;q)_s 
(z \overline z, x;q)_{t+1}.
\end{eqnarray*}
Define
\begin{eqnarray*}
F = 
(x \overline x, y;q)_r 
(y \overline y, z;q)_s 
(z \overline z, x;q)_t,
\end{eqnarray*}
so
\begin{equation}
\mathcal B_\lambda = F (x \overline x - y q^r),
\qquad \quad
\mathcal B_\mu = F (y \overline y - z q^s),
\qquad \quad
\mathcal B_\nu = F (z \overline z - x q^t).
\label{eq:Flist}
\end{equation}
Using 
(\ref{eq:abcdata})
along with $r+s+t=N-1$ and 
$ \overline x \, \overline y \, \overline z = q^{N-1}$, we obtain
\begin{equation}
a (x \overline x - y q^r) +
b (y \overline y - z q^s) +
c (z \overline z - x q^t) = 0.
\label{eq:abcview}
\end{equation}
Equation
(\ref{eq:weightedsum}) follows from
(\ref{eq:Flist})
and
(\ref{eq:abcview}).
\end{proof}

\begin{lemma}
\label{lem:CBAcheck2}
The function $\mathcal B$ from Definition
\ref{def:CBAvq}
is a Concrete Billiard Array.
\end{lemma}
\begin{proof} We show that $\mathcal B$ satisfies the
two conditions in Definition
\ref{def:cba}. Concerning Definition
\ref{def:cba}(i), let $L$ denote a line in
$\Delta_N$ and consider the locations in $L$.
We show that their images under $\mathcal B$ are
linearly independent. Without loss we may assume
that $\mathcal B$ is a $1$-line. Denote the cardinality
of $L$ by $i+1$. The locations in $L$ are listed
in 
(\ref{eq:oneline}).
For these locations 
their images under $\mathcal B$ are
\begin{equation}
(x \overline x, y;q)_{N-i} 
(y \overline y, z;q)_{i-j} 
(z \overline z, x;q)_{j} 
\qquad \qquad 0 \leq j \leq i.
\label{eq:CBAvqline2}
\end{equation}
We show that the vectors 
(\ref{eq:CBAvqline2}) are linearly independent. The vectors
\begin{equation}
(y \overline y, z;q)_{i-j} 
(z \overline z, x;q)_{j} 
\qquad \qquad 0 \leq j \leq i
\label{eq:CBAvqline3}
\end{equation}
are linearly independent; to see this set $z=0$ in
(\ref{eq:CBAvqline3}) and note that $x,y$ are algebraically
independent. It follows that the vectors
(\ref{eq:CBAvqline2}) are linearly independent.
We have shown that $\mathcal B$ satisfies Definition
\ref{def:cba}(i).
The function $\mathcal B$ satisfies 
Definition
\ref{def:cba}(ii) by Lemma 
\ref{lem:weightedsum}.
\end{proof}

\noindent For the 
 Concrete Billiard Array $\mathcal B$ in
Definition
\ref{def:CBAvq}
 we now compute
 the corresponding edge labelling and value function.

\begin{lemma}
\label{lem:CBAvqEL}
With the notation from above Lemma
\ref{lem:weightedsum},
\begin{eqnarray*}
&&
\tilde {\mathcal B}_{\lambda,\mu}  =  q^r/{\overline y},
\qquad
\tilde {\mathcal B}_{\mu,\nu}  =  q^s/{\overline z},
\qquad
\tilde {\mathcal B}_{\nu,\lambda}  = q^t/{\overline x},
\\
&&
\tilde {\mathcal B}_{\mu,\lambda}  =  {\overline y} /q^{r},
\qquad
\tilde {\mathcal B}_{\nu,\mu}  = {\overline z}/ q^{s},
\qquad
\tilde {\mathcal B}_{\lambda,\nu}  =  {\overline x} /q^{t}.
\end{eqnarray*}
\end{lemma}
\begin{proof} Compare
Lemma 
\ref{lem:3frac}
and
Lemma 
\ref{lem:weightedsum}.
\end{proof}

\begin{proposition}
\label{prop:VFq}
Let $\mathcal B$ denote the Concrete Billiard Array from
Definition
\ref{def:CBAvq}.
Then each white 3-clique in $\Delta_N$ has $ {\mathcal B}$-value $q$.
\end{proposition}
\begin{proof} Assume $N\geq 2$; otherwise
$\Delta_N$ has no white 3-clique. Let $(r,s,t) \in \Delta_{N-2}$
and consider the corresponding white 3-clique in $\Delta_N$
from Lemma
\ref{lem:black}. This 3-clique consists of the locations
\begin{eqnarray*}
\lambda = (r,s+1,t+1), \qquad \quad
\mu = (r+1,s,t+1), \qquad \quad
\nu = (r+1,s+1,t).
\end{eqnarray*}
Using Lemma
\ref{lem:CBAvqEL},
\begin{eqnarray*}
\tilde {\mathcal B}_{\lambda,\mu}  =  {\overline y}/q^r,
\qquad
\tilde {\mathcal B}_{\mu,\nu}  =  {\overline z}/q^s,
\qquad
\tilde {\mathcal B}_{\nu,\lambda}  = {\overline x}/q^t.
\end{eqnarray*}
By Lemma 
\ref{lem:bcval}
and Definition \ref{def:BBB}
the $\mathcal B$-value of the above
white 3-clique is 
$\tilde {\mathcal B}_{\lambda,\mu}
\tilde {\mathcal B}_{\mu,\nu} 
\tilde {\mathcal B}_{\nu,\lambda}$
which is equal to 
\begin{eqnarray*}
\frac{
{\overline x}\,
{\overline y}\,
{\overline z}
}{
q^{r+s+t}} = \frac{q^{N-1}}{q^{N-2}} = q.
\end{eqnarray*}
\end{proof}

\begin{corollary} For the Concrete Billiard Array
$\mathcal B$ from Definition
\ref{def:CBAvq}, the similarity class is independent of
the choice of $\overline x$, $\overline y$, $\overline z$
and depends only on $q$, $N$. 
\end{corollary}
\begin{proof} 
By Lemma
\ref{lem:dcom}
and
Proposition
\ref{prop:VFq}.
\end{proof}

\begin{corollary} For the Concrete Billiard Array
$\mathcal B$ from Definition
\ref{def:CBAvq}, the isomorphism class of the
corresponding Billiard Array is independent of
the choice of $\overline x$, $\overline y$, $\overline z$
and depends only on $q$, $N$. 
\end{corollary}
\begin{proof} 
By Corollary
\ref{thm:baVf} and
Proposition \ref{prop:VFq}.
\end{proof}

\section{The Lie algebra
 $\mathfrak{sl}_2$ and the quantum algebra
 $U_q(\mathfrak{sl}_2)$}

\noindent 
In this section we use Billiard Arrays to describe
the finite-dimensional irreducible modules for the Lie
algebra 
$\mathfrak{sl}_2$ and the quantum algebra
 $U_q(\mathfrak{sl}_2)$.
We now recall
$\mathfrak{sl}_2$. We will use the equitable basis,
which was introduced in 
\cite{ht}
and comprehensively described in \cite{bt}.
Until the end of Corollary
 \ref{eq:charzero},
assume that the characteristic of
$\mathbb F$ is not 2.

\begin{definition}
\label{def:sl2A}
\rm
\cite[Lemma~3.2]{ht}
Let $\mathfrak{sl}_2$ denote the Lie algebra over $\mathbb F$
with basis $x,y,z$ and Lie bracket
\begin{equation}
\label{eq:uqrelsq1}
\lbrack{ x,y}\rbrack = 2 {x}+2{ y}, \qquad 
\lbrack{ y,z}\rbrack = 2{ y}+2 { z}, \qquad 
\lbrack{ z,x}\rbrack = 2{ z}+2{ x}.
\end{equation}
\end{definition}

\noindent 
Until the end of Corollary
 \ref{eq:charzero}
the following notation is in effect.
Fix $N \in \mathbb N$. Let $V$ denote a vector space over $\mathbb F$
with dimension $N+1$. Let $B$ denote a Billiard
Array on $V$. Assume that each white 3-clique in
$\Delta_N$ has $B$-value $1$. Using $B$ we will turn
$V$ into
a $\mathfrak{sl}_2$-module. We acknowledge that our construction
is essentially the same as
the one in
 \cite[Proposition 8.17]{bt}; the details 
 given here
 are meant to illuminate the role played by the maps $\tilde B_{\lambda, \mu}$.
Recall the
$B$-decompositions
 $\lbrack \eta, \xi\rbrack$ of $V$ from Definition
\ref{def:Bdec}.

\begin{definition}
\label{def:XYZq1}
\rm
Define 
$X,Y,Z$
in 
 ${\rm End}(V)$ such that
for $0 \leq i \leq N$,
$X-(2i-N)I$
(resp. 
$Y-(2i-N)I$)
(resp. 
$Z-(2i-N)I$)
vanishes on component $i$ of
the $B$-decomposition $\lbrack 2,3\rbrack$
(resp. $\lbrack 3,1\rbrack$)
(resp. $\lbrack 1,2\rbrack$).
\end{definition}

\noindent The next result is meant to clarify Definition
\ref{def:XYZq1}.

\begin{lemma}
\label{lem:clarXYZq1}
Pick a location $\lambda = (r,s,t)$ on the boundary of $\Delta_N$.
Then on $B_\lambda$,
\begin{eqnarray*}
&&X=(2t-N)I=(N-2s)I \qquad \qquad {\mbox{ if $r=0$}}; 
\\
&&Y=(2r-N)I=(N-2t)I \qquad \qquad {\mbox{ if $s=0$}}; 
\\
&&Z=(2s-N)I=(N-2r)I \qquad \qquad {\mbox{ if $t=0$}}.
\end{eqnarray*}
\end{lemma}
\begin{proof} By Lemma
\ref{lem:BdecClar} and
Definition
\ref{def:XYZq1}.
\end{proof}

\noindent
Pick $(r,s,t) \in
\Delta_{N-1}$ and consider the corresponding
black 3-clique
in $\Delta_N$ from Lemma
\ref{lem:white}:
\begin{equation*}
\lambda = (r+1,s,t),
\qquad
\mu = (r,s+1,t),
\qquad
\nu = (r,s,t+1).
\end{equation*}

\begin{proposition}
\label{prop:XYZonWhiteq1}
With the above notation,
for each abrace
\begin{equation*}
u \in B_\lambda, \quad \qquad 
v \in B_\mu, \qquad \quad 
w \in B_\nu
\end{equation*}
we have
\begin{eqnarray}
&&Xu = (N-2t)v + (2s-N)w,
\label{eq:XYZonWhiteq1}
\\
&&
Yv = (N-2r)w + (2t-N)u,
\nonumber
\\
&&Zw = (N-2s)u + (2r-N)v.
\nonumber
\end{eqnarray}
\end{proposition}
\begin{proof}
We verify
(\ref{eq:XYZonWhiteq1}).
Our proof is by induction on $r$. First assume $r=0$, 
so that $s+t=N-1$. By construction
$\mu = (0,N-t,t)$ 
and
 $\nu = (0,s, N-s)$.
So $Xv = (2t-N)v$ and
$Xw = (N-2s)w$  in view of Lemma
\ref{lem:clarXYZq1}.
 To obtain
(\ref{eq:XYZonWhiteq1}), in the equation $u+v+w=0$
apply $X$ to each term and evaluate the result using
the above comments. We have verified
(\ref{eq:XYZonWhiteq1}) for $r=0$.
Next assume $r\geq 1$. 
For the location $(r-1,s,t+1) \in \Delta_{N-1}$ the
corresponding black 3-clique in $\Delta_N$ is
\begin{equation*}
\lambda' = (r,s,t+1),
\qquad
\mu' = (r-1,s+1,t+1),
\qquad
\nu' = (r-1,s,t+2).
\end{equation*}
Observe $\lambda'=\nu$.
By Lemma 
\ref{lem:brace2} and since $0 \not=w \in B_\nu= B_{\lambda'}$ there
exists a unique abrace
\begin{eqnarray*}
u' \in B_{\lambda'},
\qquad \quad
v' \in B_{\mu'},
\qquad \quad
w' \in B_{\nu'}
\end{eqnarray*}
such that $u'=w$.
By Definition
\ref{def:brace}
$u'+v'+w'=0$.
By induction
\begin{equation}
\label{eq:indX1okq1}
Xu' = (N-2t-2)v' +(2s-N)w'.
\end{equation}
For the location $(r-1,s+1,t) \in \Delta_{N-1}$ the
corresponding black 3-clique in $\Delta_N$ is
\begin{equation*}
\lambda'' = (r,s+1,t),
\qquad
\mu'' = (r-1,s+2,t),
\qquad
\nu'' = (r-1,s+1,t+1).
\end{equation*}
Observe $\lambda''=\mu$.
By Lemma 
\ref{lem:brace2} and since 
$0 \not=v \in B_\mu= B_{\lambda''}$ there
exists a unique abrace
\begin{eqnarray*}
u'' \in B_{\lambda''},
\qquad \quad
v'' \in B_{\mu''},
\qquad \quad
w'' \in B_{\nu''}
\end{eqnarray*}
such that $u''=v$.
By Definition
\ref{def:brace}
$u''+v''+w''=0$. By induction
\begin{equation}
\label{eq:indX2okq1}
Xu'' = (N-2t)v'' +(2s+2-N)w''.
\end{equation}
We claim that $w''=v'$.
We now prove the claim. Observe 
$\mu'=\nu''$. 
The three locations
\begin{eqnarray*}
\mu'=\nu'', \quad \qquad
\lambda'=\nu, \quad \qquad
\lambda''=\mu
\end{eqnarray*}
run clockwise around a white 3-clique in $\Delta_N$,
which we recall has $B$-value $1$.
By construction $v',u'$ is a brace for the
edge $\mu',\lambda'$ so
$\tilde B_{\mu',\lambda'}$ sends $v'\mapsto u'$.
Similarly
$w,v$ is a brace for the
edge $\nu,\mu$ so
$\tilde B_{\nu,\mu}$ sends  $w\mapsto v$.
Similarly
$u'',w''$ is a brace for the
edge $\lambda'',\nu''$ so
$\tilde B_{\lambda'',\nu''}$ sends $u''\mapsto w''$.
Consider the composition
\begin{equation*}
\begin{CD} 
B_{\mu'}  @>>  \tilde B_{\mu',\lambda'} >  
 B_{\lambda'}=B_\nu  @>> \tilde B_{\nu,\mu} > B_{\mu} = B_{\lambda''}
               @>>\tilde B_{\lambda'',\nu''} > B_{\nu''}=B_{\mu'}.
                  \end{CD}
\end{equation*}
On one hand, this composition sends
\begin{eqnarray*}
v' \mapsto u'=w \mapsto v=u''\mapsto w''.
\end{eqnarray*}
On the other hand, by Definition
\ref{def:orientedBval}
this composition is 
equal to the identity map
on $B_\mu'$. Therefore $w''=v'$ and
the claim is proved.
Now to obtain
(\ref{eq:XYZonWhiteq1}), in the equation $u+v+w=0$
apply $X$ to each term and evaluate the result using
(\ref{eq:indX1okq1}),
(\ref{eq:indX2okq1}) and nearby comments.
 We have verified equation
(\ref{eq:XYZonWhiteq1}). The remaining two equations
are similarly verified.
\end{proof}

\noindent We now reformulate
Proposition
\ref{prop:XYZonWhiteq1}.

\begin{corollary}
\label{prop:XYZonWhiteRefq1}
Referring to the notation above
Proposition
\ref{prop:XYZonWhiteq1},
 the following {\rm (i)--(iii)} hold.
\begin{enumerate}
\item[\rm (i)] On $B_\lambda$,
\begin{equation}
X = (N-2t)\tilde B_{\lambda,\mu} + (2s-N) \tilde B_{\lambda,\nu}.
\label{eq:xyq1}
\end{equation}
\item[\rm (ii)] On $B_\mu$,
\begin{equation*}
Y = (N-2r)\tilde B_{\mu,\nu} +(2t-N) \tilde B_{\mu,\lambda}.
\end{equation*}
\item[\rm (iii)] On $B_\nu$,
\begin{equation*}
Z = (N-2s)\tilde B_{\nu,\lambda} +(2r-N) \tilde B_{\nu,\mu}.
\end{equation*}
\end{enumerate}
\end{corollary}
\begin{proof} Use
Lemma
\ref{lem:mapsforabrace}
and
Proposition
\ref{prop:XYZonWhiteq1}.
\end{proof}

\noindent Recall the vectors $\alpha, \beta, \gamma$ from line
(\ref{eq:alphaBeta}).

\begin{proposition}
\label{prop:reformP2q1}
For each location $\lambda = (r,s,t) $ in $\Delta_N$ the
following hold on $B_\lambda$:
\begin{eqnarray*}
&&
X - (N-2s)I = 2r \tilde B_{\lambda, \lambda-\alpha},
\qquad  \quad
X - (2t-N)I = -2r\tilde B_{\lambda, \lambda+\gamma},
\\
&&
Y - (N-2t)I = 2s\tilde B_{\lambda,\lambda-\beta},
\qquad  \quad
Y - (2r-N)I = -2s\tilde B_{\lambda, \lambda+\alpha},
\\
&&
Z - (N-2r)I = 2t\tilde B_{\lambda, \lambda-\gamma},
\qquad  \quad
Z - (2s-N)I = -2t\tilde B_{\lambda, \lambda+\beta}.
\end{eqnarray*}
\end{proposition}
\begin{proof} We verify the equations  in the top row.
First assume $r=0$, so that $s+t=N$.
In each equation the left-hand side is zero by Lemma
\ref{lem:clarXYZq1}. In each equation the right-hand side
is zero since $r=0$.
Next assume $r\geq 1$.  Define
\begin{equation*}
\mu=(r-1,s+1,t)=\lambda - \alpha,
\qquad \quad 
\nu=(r-1,s,t+1)=\lambda + \gamma.
\end{equation*}
The locations $\lambda,\mu,\nu$ form a black 3-clique in $\Delta_N$
that corresponds to the location $(r-1,s,t)$ in $\Delta_{N-1}$ under
the bijection of Lemma
\ref{lem:white}.
Apply
Corollary \ref{prop:XYZonWhiteRefq1}(i) to this 3-clique
and in the resulting equation
(\ref{eq:xyq1})
eliminate 
$\tilde B_{\lambda,\mu}$ or 
$\tilde B_{\lambda,\nu}$ using
(\ref{eq:sumit}).
We have verified the equations  in the top row.
The other equations are similarly verified.
\end{proof}

\begin{corollary}
For each location $\lambda = (r,s,t)$ in $\Delta_N$,
\begin{eqnarray*}
&&
(X-(N-2s)I)B_\lambda \subseteq B_{\lambda -\alpha},
\qquad \qquad
(X-(2t-N)I)B_\lambda \subseteq B_{\lambda +\gamma},
\\
&&
(Y-(N-2t)I)B_\lambda \subseteq B_{\lambda-\beta},
\qquad \qquad
(Y-(2r-N)I)B_\lambda \subseteq B_{\lambda+\alpha},
\\
&&
(Z-(N-2r)I)B_\lambda \subseteq B_{\lambda-\gamma},
\qquad \qquad
(Z-(2s-N)I)B_\lambda \subseteq B_{\lambda+\beta}.
\end{eqnarray*}
Moreover, equality holds in each inclusion provided that
the characteristic of $\mathbb F$ is 0 or greater than $N$.
\end{corollary}
\begin{proof} By Note
\ref{note:maps} and Proposition
\ref{prop:reformP2q1}.
\end{proof}

\noindent Our next goal is to show that the elements
$X,Y,Z$ from Definition
\ref{def:XYZq1} satisfy the defining  relations
for 
$\mathfrak{sl}_2$ given in 
(\ref{eq:uqrelsq1}).

\begin{lemma}
\label{lem:twoviewsq1}
For each  location $\lambda = (r,s,t) $ in $\Delta_N$
the following hold on $B_\lambda$:
\begin{eqnarray*}
&&
XY-2Y=
X(2r-N)+
Y(N-2s) + (N+2)(2t-N)I
=YX+2X,
\\
&&
YZ-2Z=
Y(2s-N)+
Z(N-2t) + (N+2)(2r-N)I
=ZY+2Y,
\\
&&
ZX-2X=
Z(2t-N)+
X(N-2r) + (N+2)(2s-N)I
=XZ+2Z.
\end{eqnarray*}
\end{lemma}
\begin{proof}
We verify the equation on the left in the bottom row.
For this equation let $M$ denote the left-hand side minus
the right-hand side. We show $M=0$ on $B_\lambda$.
First assume $r=0$, so that $s+t=N$.
By Lemma
\ref{lem:clarXYZq1}
$X=(2t-N)I$ on $B_\lambda$.
By these comments we routinely obtain $M=0$ on $B_\lambda$.
Next assume $r\geq 1$. Observe $\lambda+\gamma=(r-1,s,t+1) \in \Delta_N$.
By Lemma
\ref{lem:Binverses} the maps
$\tilde B_{\lambda, \lambda+\gamma}:B_\lambda \to B_{\lambda+\gamma}$
and
$\tilde B_{\lambda+\gamma, \lambda}:B_{\lambda+\gamma} \to B_{\lambda}$
are inverses. Using Proposition 
\ref{prop:reformP2q1} we find that on $B_\lambda$,
\begin{equation*}
X - (2t-N)I = -2r\tilde B_{\lambda, \lambda+\gamma},
\end{equation*}
and on $B_{\lambda+\gamma}$,
\begin{equation*}
Z - (N-2r+2)I = 2(t+1)\tilde B_{\lambda+\gamma, \lambda}.
\end{equation*}
Therefore on $B_\lambda$,
\begin{equation}
(Z-(N-2r+2)I)(X-(2t-N)I) = -4r(t+1)I.
\label{eq:prelimq1}
\end{equation}
Evaluating (\ref{eq:prelimq1}) using $r+s+t=N$ we find that
$M=0$ on $B_\lambda$.
We have verified the equation on the left in the bottom row.
The remaining equations
are similarly verified.
\end{proof}

\begin{proposition}
\label{lem:qWeylAltq1}
The elements $X,Y,Z$ from Definition
\ref{def:XYZq1} satisfy
\begin{eqnarray*}
XY-YX= 2X+2Y,
\qquad 
YZ-ZY= 2Y+2Z,
\qquad
ZX-XZ= 2Z+2X.
\label{eq:qWeylAltq1}
\end{eqnarray*}
\end{proposition}
\begin{proof} By Lemma
\ref{lem:twoviewsq1}
these equations hold on $B_\lambda$
for all $\lambda \in \Delta_N$. The result holds
in view of Corollary
\ref{lem:BspanV2}.
\end{proof}

\begin{theorem}
\label{thm:sl2A}
Let $V$ denote a vector space over $\mathbb F$
with dimension $N+1$. Let $B$ denote a Billiard
Array on $V$. Assume that each white 3-clique in
$\Delta_N$ has $B$-value $1$.
Then there exists a unique 
$\mathfrak{sl}_2$-module structure on $V$
such that for $0 \leq i \leq N$,
${ x}-(2i-N)I$
(resp. 
${ y}-(2i-N)I$)
(resp. 
${ z}-(2i-N)I$)
vanishes on component $i$ of
the $B$-decomposition $\lbrack 2,3\rbrack$
(resp. $\lbrack 3,1\rbrack$)
(resp. $\lbrack 1,2\rbrack$).
The $\mathfrak{sl}_2$-module $V$ is irreducible, provided
that the characteristic of $\mathbb F$ is 0 or
greater than $N$.
\end{theorem}
\begin{proof} The 
 $\mathfrak{sl}_2$-module structure exists by
Proposition
\ref{lem:qWeylAltq1}.
It is unique by construction. The last assertion
of the theorem is readily checked.
\end{proof}

\begin{definition}
\label{def:nuxyzq1}
\rm
\cite[Section~2]{bt}
Define ${\nu_x,\nu_y,\nu_z}$  in 
$\mathfrak{sl}_2$ by
\begin{equation*}
-2 { \nu_x} ={y}+{ z},
\qquad
-2{ \nu_y}={z}+{ x},
\qquad
-2{ \nu_z} = { x}+{ y}.
\end{equation*}
\end{definition}

\noindent Our next goal is to describe the
actions of
${ \nu_x, \nu_y, \nu_z}$ on $\lbrace B_{\lambda}\rbrace_{\lambda \in \Delta_N}$.

\begin{proposition}
\label{lem:nuActionq1}
For each location $\lambda =(r,s,t)$ in  $\Delta_N$
the following hold on $B_\lambda$:
\begin{eqnarray*}
{ \nu_x} =
s\tilde B_{\lambda,\lambda+\alpha}
-t\tilde B_{\lambda,\lambda-\gamma},
\qquad 
{ \nu_y} =
t\tilde B_{\lambda,\lambda+\beta}
-r\tilde B_{\lambda,\lambda-\alpha},
\qquad
{ \nu_z} =
r\tilde B_{\lambda,\lambda+\gamma}
-s\tilde B_{\lambda,\lambda-\beta}.
\end{eqnarray*} 
\end{proposition}
\begin{proof} Evaluate the equations in
 Definition
\ref{def:nuxyzq1} using
 Proposition
\ref{prop:reformP2q1}.
\end{proof}

\begin{corollary}
\label{cor:nuActionq1}
For each location $\lambda \in \Delta_N$,
\begin{equation*}
{\nu_x} B_{\lambda} \subseteq 
B_{\lambda + \alpha}
+
B_{\lambda -\gamma},
\qquad 
{ \nu_y} B_{\lambda} \subseteq
B_{\lambda+\beta}
+
B_{\lambda -\alpha},
\qquad 
{\nu_z} B_{\lambda} \subseteq 
B_{\lambda+\gamma}
+
B_{\lambda-\beta}.
\end{equation*}
\end{corollary}
\begin{proof} By Note
\ref{note:maps} and Proposition
\ref{lem:nuActionq1}.
\end{proof}

\noindent Recall the $B$-flags $\lbrack 1\rbrack,
 \lbrack 2\rbrack,
 \lbrack 3\rbrack$ from
Definition \ref{lem:threeflags}.

\begin{proposition}
\label{prop:3flagsRevq1}
Assume that the characteristic of $\mathbb F$
is 0 or greater than $N$.
Then the $B$-flags 
$\lbrack 1 \rbrack $,
$\lbrack 2 \rbrack $,
$\lbrack 3 \rbrack $ are, respectively,
\begin{equation*}
\lbrace \nu^{N-i}_{ x}V\rbrace_{i=0}^N,
\qquad
\lbrace \nu^{N-i}_{ y} V\rbrace_{i=0}^N,
\qquad
\lbrace \nu^{N-i}_{ z}V\rbrace_{i=0}^N.
\end{equation*}
\end{proposition}
\begin{proof}
Consider  the $B$-flag $\lbrack 1 \rbrack$.
Denote this by $\lbrace U_i\rbrace_{i=0}^N$.
Denote the $B$-decomposition $\lbrack 1,2 \rbrack$ by
$\lbrace V_i\rbrace_{i=0}^N$.
Recall from 
 Lemma
\ref{lem:BdecClar} that
for $0 \leq i \leq N$,
 $V_i$ is included in $B$
at location $(N-i,i,0)$.
Recall from Lemma
\ref{lem:decIndFlag} that $U_i = V_0+\cdots + V_i$ for
$0 \leq i \leq N$. 
By Corollary
\ref{cor:nuActionq1} we find 
${ \nu_x} V_i \subseteq V_{i-1}$ for $1 \leq i \leq N$
and ${ \nu_x} V_0=0$. 
Going back to Proposition
\ref{lem:nuActionq1}
and invoking our assumption about the characteristic of $\mathbb F$,
we see that in fact
${ \nu_x} V_i=V_{i-1}$
for $1 \leq i \leq N$. By the above comments
${ \nu_x} U_i = U_{i-1}$ for $1 \leq i \leq N$. 
Now since $U_N=V$ we obtain
$U_i = \nu^{N-i}_{ x}V$ for $0 \leq i \leq N$. We have shown
that the $B$-flag $\lbrack 1\rbrack $ is equal to
$\lbrace \nu^{N-i}_{ x}V\rbrace_{i=0}^N$.
The remaining assertions are similarly shown.
\end{proof}

\begin{corollary} \label{eq:charzero}
Assume that $\mathbb F$ has characteristic
$0$. Let $V$ denote an
irreducible 
$\mathfrak{sl}_2$-module of dimension $N+1$.
Then
\begin{enumerate}
\item[\rm (i)] the following are totally opposite flags on $V$:
\begin{equation}
\label{eq:Sl2list}
\lbrace \nu^{N-i}_xV\rbrace_{i=0}^N,
\qquad
\lbrace \nu^{N-i}_yV\rbrace_{i=0}^N,
\qquad
\lbrace \nu^{N-i}_zV\rbrace_{i=0}^N.
\end{equation}
\item[\rm (ii)] for the corresponding Billiard
Array on $V$, the value of each white 3-clique is $1$.
\end{enumerate}
\end{corollary}
\begin{proof} 
By
\cite[Theorem 7.2]{humphreys},
up to isomorphism there exists a unique irreducible
$\mathfrak{sl}_2$-module with dimension $N+1$.
The 
$\mathfrak{sl}_2$-module from Theorem
\ref{thm:sl2A} is irreducible with dimension $N+1$.
Therefore the 
$\mathfrak{sl}_2$-module $V$ is isomorphic
to the 
$\mathfrak{sl}_2$-module from Theorem
\ref{thm:sl2A}.
Via the isomorphism we identify these 
$\mathfrak{sl}_2$-modules.
Consider the Billiard Array $B$ on $V$ from
Theorem
\ref{thm:sl2A}.
By Theorem 
\ref{thm:forward}
and Proposition
\ref{prop:3flagsRevq1}, the three sequences
in line
(\ref{eq:Sl2list}) are totally opposite flags on $V$,
and $B$ is the corresponding Billiard Array. By the assumption
of Theorem
\ref{thm:sl2A}, the $B$-value of each white $3$-clique is 1.
\end{proof}

\noindent 
We are done discussing 
 $\mathfrak{sl}_2$.
For the rest of this section assume
the field $\mathbb F$ is arbitrary.
Fix a nonzero $q \in \mathbb F$ such that
$q^2 \not=1$. We recall 
the quantum algebra $U_q(\mathfrak{sl}_2)$.
We will use the equitable presentation, which was introduced in
\cite{equit}.

\begin{definition}
\label{def:uqA}
\rm 
\cite[Theorem~2.1]{equit}
Let 
$U_q(\mathfrak{sl}_2)$ denote the associative $\mathbb F$-algebra
with generators $x,y^{\pm 1},z$ and  relations
$yy^{-1}=1$, $y^{-1}y=1$,
\begin{equation}
\frac{qxy-q^{-1}yx}{q-q^{-1}} = 1, \qquad
\frac{qyz-q^{-1}zy}{q-q^{-1}} = 1, \qquad
\frac{qzx-q^{-1}xz}{q-q^{-1}} = 1.
\label{eq:uqrels}
\end{equation}
\end{definition}

\noindent For the rest of this section
the following notation is in effect.
Fix $N \in \mathbb N$.
Let $V$ denote a vector space over $\mathbb F$
with dimension $N+1$. Let $B$ denote a Billiard
Array on $V$. Assume that each white 3-clique in
$\Delta_N$ has $B$-value $q^{-2}$. Using $B$ we will turn
$V$ into
a $U_q(\mathfrak{sl}_2)$-module. Recall the
$B$-decompositions
 $\lbrack \eta, \xi\rbrack$ of $V$ from Definition
\ref{def:Bdec}.

\begin{definition}
\label{def:XYZ}
\rm
Define 
$X,Y,Z$
in 
 ${\rm End}(V)$ such that
for $0 \leq i \leq N$,
$X-q^{N-2i}I$
(resp. 
$Y-q^{N-2i}I$)
(resp. 
$Z-q^{N-2i}I$)
vanishes on component $i$ of
the $B$-decomposition $\lbrack 2,3\rbrack$
(resp. $\lbrack 3,1\rbrack$)
(resp. $\lbrack 1,2\rbrack$).
Note that each of $X,Y,Z$ is invertible.
\end{definition}

\noindent The next result is meant to clarify Definition
\ref{def:XYZ}.

\begin{lemma}
\label{lem:clarXYZ}
Pick a location $\lambda = (r,s,t)$ on the boundary of $\Delta_N$.
Then on $B_\lambda$,
\begin{eqnarray*}
&&X=q^{N-2t}I=q^{2s-N}I \qquad \qquad {\mbox{ if $r=0$}}; 
\\
&&Y=q^{N-2r}I=q^{2t-N}I \qquad \qquad {\mbox{ if $s=0$}}; 
\\
&&Z=q^{N-2s}I=q^{2r-N}I \qquad \qquad {\mbox{ if $t=0$}}.
\end{eqnarray*}
\end{lemma}
\begin{proof} By Lemma
\ref{lem:BdecClar} and
Definition
\ref{def:XYZ}.
\end{proof}

\noindent
Pick $(r,s,t) \in
\Delta_{N-1}$ and consider the corresponding
black 3-clique
in $\Delta_N$ from Lemma
\ref{lem:white}:
\begin{equation*}
\lambda = (r+1,s,t),
\qquad
\mu = (r,s+1,t),
\qquad
\nu = (r,s,t+1).
\end{equation*}

\begin{proposition}
\label{prop:XYZonWhite}
With the above notation,
for each abrace
\begin{equation*}
u \in B_\lambda, \quad \qquad 
v \in B_\mu, \qquad \quad 
w \in B_\nu
\end{equation*}
we have
\begin{eqnarray}
&&Xu = -q^{N-2t}v - q^{2s-N}w,
\label{eq:XYZonWhite}
\\
&&
Yv = -q^{N-2r}w - q^{2t-N}u,
\nonumber
\\
&&Zw = -q^{N-2s}u - q^{2r-N}v.
\nonumber
\end{eqnarray}
\end{proposition}
\begin{proof}
Our proof is similar to the proof of Proposition
\ref{prop:XYZonWhiteq1}; we give the details for the sake of clarity. 
We verify
(\ref{eq:XYZonWhite}).
Our proof is by induction on $r$. First assume $r=0$, 
so that $s+t=N-1$. By construction
$\mu = (0,N-t,t)$ 
and
 $\nu = (0,s, N-s)$.
So $Xv = q^{N-2t}v$ and
$Xw = q^{2s-N}w$  in view of Lemma
\ref{lem:clarXYZ}.
 To obtain
(\ref{eq:XYZonWhite}), in the equation $u+v+w=0$
apply $X$ to each term and evaluate the result using
the above comments. We have verified
(\ref{eq:XYZonWhite}) for $r=0$.
Next assume $r\geq 1$. 
For the location $(r-1,s,t+1) \in \Delta_{N-1}$ the
corresponding black 3-clique in $\Delta_N$ is
\begin{equation*}
\lambda' = (r,s,t+1),
\qquad
\mu' = (r-1,s+1,t+1),
\qquad
\nu' = (r-1,s,t+2).
\end{equation*}
Observe $\lambda'=\nu$.
By Lemma 
\ref{lem:brace2} and since $0 \not=w \in B_\nu= B_{\lambda'}$ there
exists a unique abrace
\begin{eqnarray*}
u' \in B_{\lambda'},
\qquad \quad
v' \in B_{\mu'},
\qquad \quad
w' \in B_{\nu'}
\end{eqnarray*}
such that $u'=w$.
By Definition
\ref{def:brace}
$u'+v'+w'=0$.
By induction
\begin{equation}
\label{eq:indX1ok}
Xu' = -q^{N-2t-2}v' - q^{2s-N}w'.
\end{equation}
For the location $(r-1,s+1,t) \in \Delta_{N-1}$ the
corresponding black 3-clique in $\Delta_N$ is
\begin{equation*}
\lambda'' = (r,s+1,t),
\qquad
\mu'' = (r-1,s+2,t),
\qquad
\nu'' = (r-1,s+1,t+1).
\end{equation*}
Observe $\lambda''=\mu$.
By Lemma 
\ref{lem:brace2} and since 
$0 \not=v \in B_\mu= B_{\lambda''}$ there
exists a unique abrace
\begin{eqnarray*}
u'' \in B_{\lambda''},
\qquad \quad
v'' \in B_{\mu''},
\qquad \quad
w'' \in B_{\nu''}
\end{eqnarray*}
such that $u''=v$.
By Definition
\ref{def:brace}
$u''+v''+w''=0$. By induction
\begin{equation}
\label{eq:indX2ok}
Xu'' = -q^{N-2t}v'' - q^{2s+2-N}w''.
\end{equation}
We claim that $w''=q^{-2}v'$.
We now prove the claim. Observe 
$\mu'=\nu''$. 
The three locations
\begin{eqnarray*}
\mu'=\nu'', \quad \qquad
\lambda'=\nu, \quad \qquad
\lambda''=\mu
\end{eqnarray*}
run clockwise around a white 3-clique in $\Delta_N$,
which we recall has $B$-value $q^{-2}$.
By construction $v',u'$ is a brace for the
edge $\mu',\lambda'$ so
$\tilde B_{\mu',\lambda'}$ sends $v'\mapsto u'$.
Similarly
$w,v$ is a brace for the
edge $\nu,\mu$ so
$\tilde B_{\nu,\mu}$ sends  $w\mapsto v$.
Similarly
$u'',w''$ is a brace for the
edge $\lambda'',\nu''$ so
$\tilde B_{\lambda'',\nu''}$ sends $u''\mapsto w''$.
Consider the composition
\begin{equation*}
\begin{CD} 
B_{\mu'}  @>>  \tilde B_{\mu',\lambda'} >  
 B_{\lambda'}=B_\nu  @>> \tilde B_{\nu,\mu} > B_{\mu} = B_{\lambda''}
               @>>\tilde B_{\lambda'',\nu''} > B_{\nu''}=B_{\mu'}.
                  \end{CD}
\end{equation*}
On one hand, this composition sends
\begin{eqnarray*}
v' \mapsto u'=w \mapsto v=u''\mapsto w''.
\end{eqnarray*}
On the other hand, by Definition
\ref{def:orientedBval}
this composition is 
equal to $q^{-2}$ times the identity map
on $B_\mu'$. Therefore $w''=q^{-2}v'$ and
the claim is proved.
Now to obtain
(\ref{eq:XYZonWhite}), in the equation $u+v+w=0$
apply $X$ to each term and evaluate the result using
(\ref{eq:indX1ok}),
(\ref{eq:indX2ok}) and nearby comments.
 We have verified equation
(\ref{eq:XYZonWhite}). The remaining two equations
are similarly verified.
\end{proof}

\noindent We now reformulate
Proposition
\ref{prop:XYZonWhite}.

\begin{corollary}
\label{prop:XYZonWhiteRef}
Referring to the notation above
Proposition
\ref{prop:XYZonWhite},
 the following {\rm (i)--(iii)} hold.
\begin{enumerate}
\item[\rm (i)] On $B_\lambda$,
\begin{equation}
X = - q^{N-2t}\tilde B_{\lambda,\mu} - q^{2s-N} \tilde B_{\lambda,\nu}.
\label{eq:xy}
\end{equation}
\item[\rm (ii)] On $B_\mu$,
\begin{equation*}
Y = - q^{N-2r}\tilde B_{\mu,\nu} - q^{2t-N} \tilde B_{\mu,\lambda}.
\end{equation*}
\item[\rm (iii)] On $B_\nu$,
\begin{equation*}
Z = - q^{N-2s}\tilde B_{\nu,\lambda} - q^{2r-N} \tilde B_{\nu,\mu}.
\end{equation*}
\end{enumerate}
\end{corollary}
\begin{proof} Use
Lemma
\ref{lem:mapsforabrace}
and
Proposition
\ref{prop:XYZonWhite}.
\end{proof}


\begin{proposition}
\label{prop:reformP2}
For each location $\lambda = (r,s,t) $ in $\Delta_N$ the
following hold on $B_\lambda$:
\begin{eqnarray*}
&&
X - q^{2s-N}I = (q^{2s-N}-q^{N-2t})\tilde B_{\lambda, \lambda-\alpha},
\qquad  \quad
X - q^{N-2t}I = (q^{N-2t}-q^{2s-N})\tilde B_{\lambda, \lambda+\gamma},
\\
&&
Y - q^{2t-N}I = (q^{2t-N}-q^{N-2r})\tilde B_{\lambda,\lambda-\beta},
\qquad  \quad
Y - q^{N-2r}I = (q^{N-2r}-q^{2t-N})\tilde B_{\lambda, \lambda+\alpha},
\\
&&
Z - q^{2r-N}I = (q^{2r-N}-q^{N-2s})\tilde B_{\lambda, \lambda-\gamma},
\qquad  \quad
Z - q^{N-2s}I = (q^{N-2s}-q^{2r-N})\tilde B_{\lambda, \lambda+\beta}.
\end{eqnarray*}
\end{proposition}
\begin{proof} We verify the equations  in the top row.
First assume $r=0$, so that $s+t=N$.
In each equation the left-hand side is zero by Lemma
\ref{lem:clarXYZ}. In each equation the right-hand side
is zero since $2s-N=N-2t$.
Next assume $r\geq 1$.  Define
\begin{equation*}
\mu=(r-1,s+1,t)=\lambda - \alpha,
\qquad \quad 
\nu=(r-1,s,t+1)=\lambda + \gamma.
\end{equation*}
The locations $\lambda,\mu,\nu$ form a black 3-clique in $\Delta_N$
that corresponds to the location $(r-1,s,t)$ in $\Delta_{N-1}$ under
the bijection of Lemma
\ref{lem:white}.
Apply
Corollary \ref{prop:XYZonWhiteRef}(i) to this 3-clique
and in the resulting equation
(\ref{eq:xy})
eliminate 
$\tilde B_{\lambda,\mu}$ or 
$\tilde B_{\lambda,\nu}$ using
(\ref{eq:sumit}).
We have verified the equations  in the top row.
The other equations are similarly verified.
\end{proof}

\begin{corollary}
For each location $\lambda = (r,s,t)$ in $\Delta_N$,
\begin{eqnarray*}
&&
(X-q^{2s-N}I)B_\lambda \subseteq B_{\lambda -\alpha},
\qquad \qquad
(X-q^{N-2t}I)B_\lambda \subseteq B_{\lambda +\gamma},
\\
&&
(Y-q^{2t-N}I)B_\lambda \subseteq B_{\lambda-\beta},
\qquad \qquad
(Y-q^{N-2r}I)B_\lambda \subseteq B_{\lambda+\alpha},
\\
&&
(Z-q^{2r-N}I)B_\lambda \subseteq B_{\lambda-\gamma},
\qquad \qquad
(Z-q^{N-2s}I)B_\lambda \subseteq B_{\lambda+\beta}.
\end{eqnarray*}
Moreover, equality holds in each inclusion provided that
$q^{2i}\not=1$ for $1 \leq i \leq N$. 
\end{corollary}
\begin{proof} By Note
\ref{note:maps} and Proposition
\ref{prop:reformP2}.
\end{proof}


\noindent Our next goal is to show that the elements
$X,Y,Z$ from Definition
\ref{def:XYZ} satisfy the defining  relations
for 
$U_q(\mathfrak{sl}_2)$ given in 
(\ref{eq:uqrels}).

\begin{lemma}
\label{lem:twoviews}
For each  location $\lambda = (r,s,t) $ in $\Delta_N$
the following hold on $B_\lambda$:
\begin{eqnarray*}
&&
q(I-XY)=
-q^{N-2r+1}(X-q^{N-2t}I) - 
q^{2s-N-1}(Y-q^{2t-N}I)=
q^{-1}(I-YX),
\\
&&
q(I-YZ)=
-q^{N-2s+1}(Y-q^{N-2r}I) - 
q^{2t-N-1}(Z-q^{2r-N}I)
=q^{-1}(I-ZY),
\\
&&
q(I-ZX)=
-q^{N-2t+1}(Z-q^{N-2s}I) - 
q^{2r-N-1}(X-q^{2s-N}I)
=q^{-1}(I-XZ).
\end{eqnarray*}
\end{lemma}
\begin{proof}
We verify the equation on the left in the bottom row.
For this equation let $M$ denote the left-hand side minus
the right-hand side. We show $M=0$ on $B_\lambda$.
First assume $r=0$, so that $s+t=N$.
By Lemma
\ref{lem:clarXYZ}
$X=q^{N-2t}I$ on $B_\lambda$.
By these comments we routinely obtain $M=0$ on $B_\lambda$.
Next assume $r\geq 1$. Observe $\lambda+\gamma=(r-1,s,t+1) \in \Delta_N$.
By Lemma
\ref{lem:Binverses} the maps
$\tilde B_{\lambda, \lambda+\gamma}:B_\lambda \to B_{\lambda+\gamma}$
and
$\tilde B_{\lambda+\gamma, \lambda}:B_{\lambda+\gamma} \to B_{\lambda}$
are inverses. Using Proposition 
\ref{prop:reformP2} we find that on $B_\lambda$,
\begin{equation*}
X - q^{N-2t}I = (q^{N-2t}-q^{2s-N})\tilde B_{\lambda, \lambda+\gamma},
\end{equation*}
and on $B_{\lambda+\gamma}$,
\begin{equation*}
Z - q^{2r-2-N}I = (q^{2r-2-N}-q^{N-2s})\tilde B_{\lambda+\gamma, \lambda}.
\end{equation*}
Therefore on $B_\lambda$,
\begin{equation}
(Z-q^{2r-2-N}I)(X-q^{N-2t}I) = (q^{2r-2-N}-q^{N-2s})(q^{N-2t}-q^{2s-N})I.
\label{eq:prelim}
\end{equation}
Evaluating (\ref{eq:prelim}) using $r+s+t=N$ we find that
$M=0$ on $B_\lambda$.
We have verified the equation on the left in the bottom row.
The remaining equations
are similarly verified.
\end{proof}

\begin{lemma}
\label{lem:qWeylAlt}
The elements $X,Y,Z$ from Definition
\ref{def:XYZ} satisfy
\begin{eqnarray*}
&&q(I-XY)= q^{-1}(I-YX),
\\
&&q(I-YZ)= q^{-1}(I-ZY),
\\
&&q(I-ZX)= q^{-1}(I-XZ).
\label{eq:qWeylAlt}
\end{eqnarray*}
\end{lemma}
\begin{proof} By Lemma
\ref{lem:twoviews}
these equations hold on $B_\lambda$
for all $\lambda \in \Delta_N$. The result holds
in view of Corollary
\ref{lem:BspanV2}.
\end{proof}

\begin{proposition}
\label{lem:qWeyl}
The elements $X,Y,Z$ from Definition
\ref{def:XYZ} satisfy
\begin{equation*}
\frac{qXY-q^{-1}YX}{q-q^{-1}} = I, \qquad
\frac{qYZ-q^{-1}ZY}{q-q^{-1}} = I, \qquad
\frac{qZX-q^{-1}XZ}{q-q^{-1}} = I.
\label{eq:qWeyl}
\end{equation*}
\end{proposition}
\begin{proof} These equations are a reformulation of
the equations in
Lemma
\ref{lem:qWeylAlt}.
\end{proof}

\begin{theorem}
\label{thm:uqsl2}
Let $V$ denote a vector space over $\mathbb F$
with dimension $N+1$. Let $B$ denote a Billiard
Array on $V$. Assume that each white 3-clique in
$\Delta_N$ has $B$-value $q^{-2}$.
Then there exists a unique 
$U_q(\mathfrak{sl}_2)$-module structure on $V$
such that for $0 \leq i \leq N$,
$x-q^{N-2i}I$
(resp. 
$y-q^{N-2i}I$)
(resp. 
$z-q^{N-2i}I$)
vanishes on component $i$ of
the $B$-decomposition $\lbrack 2,3\rbrack$
(resp. $\lbrack 3,1\rbrack$)
(resp. $\lbrack 1,2\rbrack$).
The $U_q(\mathfrak{sl}_2)$-module $V$ is irreducible, provided that
$q^{2i}\not=1$ for $1 \leq i \leq N$. 
\end{theorem}
\begin{proof} The 
 $U_q(\mathfrak{sl}_2)$-module structure exists by
Proposition
\ref{lem:qWeyl}. It is unique by construction. The last assertion
of the theorem is readily checked.
\end{proof}

\begin{definition}
\label{def:nuxyz}
\rm
\cite[Definition~3.1]{uawe}
Let $\nu_x, \nu_y, \nu_z$ denote the following elements in
 $U_q(\mathfrak{sl}_2)$:
\begin{eqnarray*}
&&
\nu_x = q(1-yz) = 
q^{-1}(1-zy),
\\
&&
\nu_y = q(1-zx) = 
q^{-1}(1-xz),
\\
&&
\nu_z = q(1-xy) = 
q^{-1}(1-yx).
\end{eqnarray*}
\end{definition}

\noindent Our next goal is to describe the
actions of
$\nu_x, \nu_y, \nu_z$ on $\lbrace B_{\lambda}\rbrace_{\lambda \in \Delta_N}$.

\begin{proposition}
\label{lem:nuAction}
For each location $\lambda =(r,s,t)$ in  $\Delta_N$
the following hold on $B_\lambda$:
\begin{eqnarray*}
&&\nu_x =
q^{t-2s-1}(q^t-q^{-t})\tilde B_{\lambda,\lambda-\gamma}-
q^{2t-s+1}(q^s-q^{-s})\tilde B_{\lambda,\lambda+\alpha},
\\
&&
\nu_y =
q^{r-2t-1}(q^r-q^{-r})\tilde B_{\lambda,\lambda-\alpha}-
q^{2r-t+1}(q^t-q^{-t})\tilde B_{\lambda,\lambda+\beta},
\\
&&
\nu_z =
q^{s-2r-1}(q^s-q^{-s})\tilde B_{\lambda,\lambda-\beta}-
q^{2s-r+1}(q^r-q^{-r})\tilde B_{\lambda,\lambda+\gamma}.
\end{eqnarray*} 
\end{proposition}
\begin{proof} Evaluate the equations in
Lemma
\ref{lem:twoviews} using
 Proposition
\ref{prop:reformP2} and Definition
\ref{def:nuxyz}.
\end{proof}

\begin{corollary}
\label{cor:nuAction}
For each location $\lambda \in \Delta_N$,
\begin{equation*}
\nu_x B_{\lambda} \subseteq
B_{\lambda + \alpha}+
B_{\lambda -\gamma},
\qquad 
\nu_y B_{\lambda} \subseteq
B_{\lambda+\beta}+
B_{\lambda -\alpha},
\qquad 
\nu_z B_{\lambda} \subseteq
B_{\lambda+\gamma}+
B_{\lambda-\beta}.
\end{equation*}
\end{corollary}
\begin{proof} By Note
\ref{note:maps} and Proposition
\ref{lem:nuAction}.
\end{proof}

\noindent Recall the $B$-flags $\lbrack 1\rbrack,
 \lbrack 2\rbrack,
 \lbrack 3\rbrack$ from
Definition \ref{lem:threeflags}.

\begin{proposition}
\label{prop:UqTOF}
Assume that $q^{2i} \not=1$ for $1 \leq i \leq N$.
Then the $B$-flags 
$\lbrack 1 \rbrack $,
$\lbrack 2 \rbrack $,
$\lbrack 3 \rbrack $ are, respectively,
\begin{equation*}
\lbrace \nu^{N-i}_xV\rbrace_{i=0}^N,
\qquad
\lbrace \nu^{N-i}_yV\rbrace_{i=0}^N,
\qquad
\lbrace \nu^{N-i}_zV\rbrace_{i=0}^N.
\end{equation*}
\end{proposition}
\begin{proof}
Similar to the proof of
Proposition \ref{prop:3flagsRevq1}.
\end{proof}

\begin{corollary} \label{eq:qnotroot}
Assume that $q$ is not a root of unity.
 Let $V$ denote an
irreducible 
 $U_q(\mathfrak{sl}_2)$-module with dimension $N+1$.
Then
\begin{enumerate}
\item[\rm (i)] the following are totally opposite flags on $V$:
\begin{equation}
\label{eq:Uqsl2list}
\lbrace \nu^{N-i}_xV\rbrace_{i=0}^N,
\qquad
\lbrace \nu^{N-i}_yV\rbrace_{i=0}^N,
\qquad
\lbrace \nu^{N-i}_zV\rbrace_{i=0}^N.
\end{equation}
\item[\rm (ii)] for the corresponding Billiard
Array on $V$, the value of each white 3-clique is $q^{-2}$.
\end{enumerate}
\end{corollary}
\begin{proof} By 
\cite[Theorem~2.1]{equit}
and
\cite[Theorem~2.6]{jantzen},
each of $x,y,z$
is diagonalizable on $V$. Moreover by
\cite[Theorem~2.1]{equit} and
\cite[Theorem~2.6]{jantzen},
there exists
$\varepsilon \in \lbrace 1,-1\rbrace$ such that
for each of $x,y,z$ the eigenvalues on $V$
are $\lbrace \varepsilon q^{N-2i} \rbrace_{i=0}^N$.
The 
$U_q(\mathfrak{sl}_2)$-module $V$ has type $\varepsilon$ in the sense of
\cite[Section~5.2]{jantzen}. Replacing $x,y,z$
by $\varepsilon x, 
\varepsilon y, 
\varepsilon z$ respectively, the type becomes 1
and $\nu_x, \nu_y,\nu_z$ are unchanged.
By 
\cite[Theorem~2.6]{jantzen}, up to isomorphism
there exists a unique irreducible
$U_q(\mathfrak{sl}_2)$-module with type 1 and 
dimension $N+1$.
The 
$U_q(\mathfrak{sl}_2)$-module from Theorem
\ref{thm:uqsl2}
is irreducible, with type 1 and dimension $N+1$.
Therefore the 
$U_q(\mathfrak{sl}_2)$-module $V$ is isomorphic
to the 
$U_q(\mathfrak{sl}_2)$-module from Theorem
\ref{thm:uqsl2}. Via the isomorphism we identify
these 
$U_q(\mathfrak{sl}_2)$-modules.
Consider the Billiard Array $B$ on $V$ from Theorem
\ref{thm:uqsl2}.
By Theorem
\ref{thm:forward}
and Proposition 
\ref{prop:UqTOF},
the three sequences in line (\ref{eq:Uqsl2list})
are totally opposite flags on $V$, and $B$
is the corresponding Billiard Array. By the
assumption of Theorem
\ref{thm:uqsl2}, the $B$-value of each white 
3-clique is $q^{-2}$.
\end{proof}

\section{Acknowledgments}
The author thanks Jae-ho Lee and Kazumasa Nomura 
for giving this paper a close reading and offering  valuable
suggestions.

\noindent Paul Terwilliger \hfil\break
\noindent Department of Mathematics \hfil\break
\noindent University of Wisconsin \hfil\break
\noindent 480 Lincoln Drive \hfil\break
\noindent Madison, WI 53706-1388 USA \hfil\break
\noindent email: {\tt terwilli@math.wisc.edu }\hfil\break


\end{document}